\documentclass{article}[12pt]

\PassOptionsToPackage{numbers, compress}{natbib}
\usepackage{amsmath,amssymb,amsthm,fullpage,mathrsfs,pgf,tikz,caption,subcaption,mathtools}
\usepackage[ruled,vlined,linesnumbered]{algorithm2e}
\usepackage{natbib}
\usepackage{amsmath,amssymb,amsthm,mathtools}
\usepackage[utf8]{inputenc} 
\usepackage[T1]{fontenc}    
\usepackage{hyperref}       
\usepackage{url}            
\usepackage{booktabs}       
\usepackage{amsfonts}       
\usepackage{nicefrac}       
\usepackage{microtype}      
\usepackage[margin=1in]{geometry}
\usepackage{bbm}
\usepackage{enumitem}
\usepackage{xcolor}
\usepackage{cleveref}

\usepackage[T1]{fontenc}
\usepackage{lmodern}

\usepackage{caption}

\usepackage{amsfonts}
\usepackage{amssymb}
\usepackage{amsmath}
\usepackage{hyperref}
\usepackage{mathrsfs}
\usepackage{centernot}
\usepackage{mathdots}
\usepackage{stmaryrd}
\usepackage[all]{xy}

\usepackage{mathtools}

\usepackage{color}

\newtheorem{thm}{Theorem}[section]
\newtheorem{lem}[thm]{Lemma}

\newtheorem{cor}[thm]{Corollary}

\newtheorem{que}[thm]{Question}

\theoremstyle{definition}

\newtheorem{example}[thm]{Example}
\newtheorem{remark}[thm]{Remark}

\theoremstyle{plain}

\newcommand{\Step}[1]{\smallskip\noindent {\it Step #1.}} 

\newcommand{\rem}[1]{}

\newcommand{\C}{\mathbb{C}}

\newcommand{\F}{\mathbb{F}}

\newcommand{\N}{\mathbb{N}}

\newcommand{\R}{\mathbb{R}}
\newcommand{\Z}{\mathbb{Z}}

\newcommand{\HH}{{\mathrm{H}}}

\newcommand{\calA}{{\mathcal{A}}}
\newcommand{\calB}{{\mathcal{B}}}
\newcommand{\calC}{{\mathcal{C}}}

\newcommand{\calE}{{\mathcal{E}}}
\newcommand{\calF}{{\mathcal{F}}}
\newcommand{\calG}{{\mathcal{G}}}

\newcommand{\calM}{{\mathcal{M}}}

\newcommand{\calR}{{\mathcal{R}}}

\newcommand{\veps}{\varepsilon}

\newcommand{\suchthat}{\,:\,}
\newcommand{\where}{\,|\,}

\newcommand{\leftmod}{{\setminus}}

\newcommand{\Circs}[1]{\left( #1 \right)}

\newcommand{\Trings}[1]{\left< #1 \right>}

\newcommand{\floor}[1]{\lfloor {#1} \rfloor}

\newcommand{\ceil}[1]{\lceil {#1} \rceil}

\DeclareMathOperator{\id}{id} %
\DeclareMathOperator{\im}{im} %
\DeclareMathOperator{\ord}{ord} %
\DeclareMathOperator{\rank}{rank}

\DeclareMathOperator{\Span}{span} %
\DeclareMathOperator{\Spec}{Spec} %
\DeclareMathOperator{\supp}{supp} %
\DeclareMathOperator{\Tr}{Tr} %
\newcommand{\nGL}[2]{\mathrm{GL}_{#2}({#1})}

\numberwithin{equation}{section} 

\usepackage{enumitem} 

\usepackage{dsfont} 

\DeclareMathOperator{\res}{res} 
\DeclareMathOperator{\dist}{dist}
\DeclareMathOperator{\cb}{cbe} 
\DeclareMathOperator{\cse}{cse} 
\DeclareMathOperator{\ccd}{ccd} 

\newcommand{\from}{\leftarrow}

\newcommand{\aug}[1]{{#1}}  

\newcommand{\schoose}[2]{{\textstyle {{#1} \choose {#2}}}}

\title{The Cheeger Inequality and Coboundary Expansion: \\
Beyond Constant Coefficients}

\date{}

\author{Uriya A.\ First \\
Department of Mathematics\\
University of Haifa
\and
Tali Kaufman\\
Computer Science Department, \\  
Bar-Ilan University}

\begin{document}

\maketitle

\begin{abstract}
The Cheeger constant of a graph, or equivalently its coboundary expansion,
quantifies the expansion of the graph.
This notion assumes an implicit choice of a coefficient group, namely, $\F_2$.
In this paper, we study Cheeger-type inequalities for graphs endowed with a 
generalized coefficient group, called a \emph{sheaf}; this is motivated
by applications to \emph{cosystolic expansion} and \emph{locally testable codes}.
We prove that a graph is a good spectral expander if and only if it has
good coboundary expansion relative to any (resp.\ some) \emph{constant} sheaf, 
or equivalently, relative to any `ordinary' coefficient group. 
We moreover show that sheaves that are close to being constant in a well-defined
sense are also good coboundary expanders, provided that their underlying graph
is an expander, thus giving the first example of good  coboundary expansion in non-cosntant
sheaves on sparse graphs. By contrast, we observe that for general sheaves on graphs, it is impossible to relate
the expansion of the graph and the coboundary expansion of the sheaf.

We  specialize our results to sheaves on (finite) \emph{spherical buildings}.
Specifically,
we  show that the normalized second eigenvalue of the (weighted)
graph underlying a  \emph{$q$-thick} 
$d$-dimensional  spherical building  is
$O(\frac{1}{\sqrt{q}-3d})$ if $q>9d^2$. Plugging this into our results about coboundary expansion  
gives explicit lower bounds on the coboundary expansion of some constant and non-constant sheaves
on spherical buildings; for a fixed dimension $d$, the bounds approach 
a constant as the thickness $q$ grows.

Along the way, we prove 
a  new  version of the Expander Mixing Lemma for  $r$-partite weighted
graphs.

\end{abstract}

\section{Introduction}

\subsubsection*{Expander Graphs}

Informally, a (finite) graph is called an expander if relatively many edges cross
between every set of vertices and its complement.  
More precisely,
if $X$ is a graph and $w:X\to\R_+$ is a function
assigning non-negative weights to the vertices and edges of $X$,
then the expansion of the weighted graph $(X,w)$ is quantified
by its \emph{Cheeger constant}, 
\begin{equation}\label{EQ:Cheeger-const}
h(X,w)=\min_{\emptyset\neq S\subsetneq X(0)}\frac{w(E(S,X(0) -S))}{\min\{w(S),w(X(0) -S)\}}.
\end{equation}
Here, $X(0)$ is the   set of vertices of $X$ and $E(A,B)$ denotes the set of edges with one
vertex in $A$ and the other in $B$.
One says that $(X,w)$ is an \emph{$\veps$-combinatorial expander} if $h(X)\geq \veps$.

In what follows, we shall   assume that the weight function $w$  satisfies some normalization 
conditions that are listed in \S\ref{subsec:weights}. In particular,
we require that $w(X(0))=w(X(1))=1$, where
$X(1)$ is the set of edges of $X$.
For example, when $X$ is a regular graph, one can take $w$ to be uniform,
i.e., set $w(v)=\frac{1}{|X(0)|}$
for every vertex $v\in X(0)$ and $w(e)=\frac{1}{|X(1)|}$ for every edge $e\in X(1)$.

It is a celebrated fact that    $h(X,w)$ can be bounded from below  
using  
the eigenvalues of the normalized adjacency matrix of $(X,w)$;
we recall its definition in \S\ref{subsec:expansion}. In more detail, if $\lambda_2(X,w)$
is the second-largest eigenvalue of this matrix (the largest  is $1$),
then $h(X,w)\geq 1-\lambda_2(X,w)$
(\cite[Theorem~4.4(1)]{Kaufman_2021_amplified_local_testability_preprint}, for instance). We shall say that $(X,w)$ is a \emph{$\lambda$-spectral
expander} ($\lambda\in [-1,1]$)
if all the eigenvalues of its adjacency matrix except for $1$ (counted with multiplicity $1$)
lie in the interval
$[-\lambda,\lambda]$, and write $\lambda(X,w)$ for the largest $\lambda$ for which this holds.
Thus, $h(X,w)\geq 1-\lambda (X,w)$.

\subsubsection*{Coboundary Expansion}

Meshulam--Wallach \cite{Meshulam_2009_homological_connectivity}
and
Gromov \cite{Gromov_2010_expanders_and_top_II}, 
following the earlier work of Linial--Meshulam \cite{Linial_2006_homological_connectivity},
observed that 
the  $\veps$-expansion condition for graphs
can be restated in terms of cohomology with $\F_2$-coefficients, and thus be generalized
to higher dimensions if $X$ is a (weighted) simplicial complex.
This type of expansion is quantified by the \emph{coboundary expansion} of $X$ in the relevant dimension,
and coincides with the Cheeger constant in dimension $0$.
Recent works studying    the coboundary expansion of simplicial complexes in dimensions $>0$ include \cite{Dotterrer_2012_coboundary_expanders},
\cite{Lubotzky_2015_random_latin_squares},
\cite{Kaufman_2016_isoperimetic_inequalities},
\cite{Lubotzky_2016_expansion_of_buildings},
\cite{Lobotzky_2019_random_steiner_systems},
\cite{Dinur_2022_near_coverings},
\cite{Kaufman_2021_coboundary_cosystolic_exp_strong_sym_preprint}.

Recall that the $0$-dimensional coboundary expansion of a weighted graph $(X,w)$ can be defined as follows:
First, view $X$ as a $1$-dimensional simplicial complex, which means that
we add an empty face of dimension $-1$ to $X$.
We write $X(i)$  ($i\in\{-1,0,1\}$) for
the set of $i$-dimensional faces of $X$.
For every edge $e\in X(1)$, choose one its vertices, denote it as $e^+$ and denote the other vertex  as $e^-$.
Recall that an \emph{$i$-cochain} on $X$ with coefficients in $\F_2$ is
an assignment of an element of $\F_2$ to every $i$-face of $X$, i.e., a vector $f\in \F_2^{X(i)}$.
We   write $C^i=C^i(X,\F_2)=\F_2^{X(i)}$ and denote the $x$-coordinate of $f\in C^i$ as $f(x)$.
The   coboundary maps  $d_{-1}: C^{-1}\to C^{0}$ and $d_0:C^0\to C^1$
are now defined by
\begin{align*}
(d_{-1} f)(v)& =f(\emptyset), \\
(d_0 f)(e)& = f(e^+)-f(e^-). 
\end{align*} 
Clearly, $d_0\circ d_{-1}=0$. Thus, 
$B^0=B^0(X,\F_2):=\im (d_{-1})$ --- called the space of \emph{$0$-boundaries} on $X$ ---
is contained in $Z^0=Z^0(X,\F_2):=\ker(d_{0})$ --- the space of \emph{$0$-cochains} on $X$.
The \emph{coboundary expansion} of $X$ in dimension $0$ measures the expansion of $0$-cochains under $d_0$,
taking into account that $B^0$ must be mapped to $0$.
Formally, this is the largest non-negative real number $\cb_0(X,w)$ such
that
\[
\|d_0 f\| \geq \cb_0(X,w) \cdot \dist(f,B^0)\qquad\forall f\in C^0,
\]
where $\|\cdot\| $ and $\dist(,) $ are the weighted Hamming norm and distance
(in $\F_2^{X(0)}$ or $\F_2^{X(1)}$) given by $\|f\|=w(\supp f)$
and $\dist(f,g)=w(\supp(f-g))$. We   say that
$(X,w)$ is an \emph{$\veps$-coboundary expander} in dimension $0$ if $\veps\leq \cb_0(X,w)$.

It is straightforward to see that $\cb_0(X,w)$ coincides with the Cheeger constant $h(X,w)$.
However, the description of $h(X,w)$ via cochains
reveals that we have made an implicit choice of   coefficient group, namely, $\F_2$.
Indeed, the definition of $\cb_0(X,w)$ still makes sense if  replace $\F_2$
with with another abelian group,
and it is natural to ask what can be said about the coboundary expansion in this case.
 
Our first main result, Theorem~\ref{TH:expanders-are-cb-exps}, 
answers this question. Let $R$ be a nontrivial abelian group,
and let    $\cb_0(X,w,R)$  be the $0$-dimensional 
coboundary expansion of $(X,w)$ when the coefficient group is taken to be $R$. 
Then, similarly to the known inequality
$\cb_0(X,w,\F_2)=h(X,w)\geq 1-\lambda_2(X,w)$, we show that
\[
\cb_0(X,w,R)\geq 1-\lambda_2(X,w).
\]
Moreover, while $\cb_0(X,w,R)$ may vary with $R$, we always have
\[\frac{1}{2} h(X,w)\leq \cb_0(X,w,R)\leq h(X,w),\] 
so $(X,w,R)$ is a good coboundary expander whenever $(X,w)$ is a good expander (Corollary~\ref{CR:different-groups}).

That said, the  goal of this paper is to establish lower bounds on the $0$-dimensional coboundary expansion w.r.t.\ even
more general coefficient systems, called \emph{sheaves}.
Beside the independent interest in allowing more flexible coefficient
systems, 
coboundary expansion of sheaves
has implications to coding theory that we explain later in this introduction.
Briefly, as shown in \cite{First_2024_cosyst_exp_posets_preprint} (see also \cite{First_2024_cosyst_exp_posets_stoc}), sheaves with good coboundary expansion
are an important ingredient in constructing good \emph{cosystolic expanders},
which in turn give rise to \emph{locally testable codes}.
Our results here are specifically
needed in \cite[\S9--10]{First_2023_sheaves_on_complexes_preprint}\footnote{
	Section~1--8 of \cite{First_2023_sheaves_on_complexes_preprint} are now subsumed
	by \cite{First_2024_cosyst_exp_posets_preprint}. The remaining sections
	of
	\cite{First_2023_sheaves_on_complexes_preprint} are planned to be subsumed by future work 
	featuring stronger results.
} for this purpose.

\subsubsection*{Sheaves}

The common meaning of a \emph{sheaf} in the mathematical literature is a sheaf
on a topological space; such sheaves are ubiquitous to topology and algebraic geometry.
The sheaves that we consider here, however, are discrete, more elementary analogues that are defined
over  \emph{cell complexes}  
and
are also known in the literature as \emph{cellular sheaves}
or \emph{local systems}.
They were first introduced by Shepard \cite{Shepard_1985_cellular_descrip_of_der_cat_PhD}
and studied further by Curry \cite{Curry_2014_sheaves_cosheaves_PhD};
a concise treatment can be found in \cite{Hansen_2019_spectral_thy_of_sheaves},
\cite[\S5]{First_2024_cosyst_exp_posets_preprint}
or \cite[\S4]{First_2023_sheaves_on_complexes_preprint}.

For simplicity, we restrict our discussion here to sheaves on graphs
and recall the general definition later in \S\ref{subsec:sheaves}.
Similarly
to \cite{First_2024_cosyst_exp_posets_preprint}, and unlike \cite{Shepard_1985_cellular_descrip_of_der_cat_PhD} and
\cite{Curry_2014_sheaves_cosheaves_PhD},  we shall
view all graphs as $1$-dimensional simplicial complexes
and 
take the empty face of a graph into account when defining a sheaf on it.

Let $X$ be  a graph. A  \emph{sheaf}   $\calF$ on $X$ consists of 
\begin{enumerate}[label=(\arabic*)]
	\item an abelian group $\calF(x)$ for every  $x\in X=X(-1)\cup X(0)\cup X(1)$, and
	\item a group homomorphism $\res^\calF_{y\from x}:\calF(x)\to\calF(y)$ for all $x\subsetneq y\in X$
\end{enumerate}
such that
\begin{equation}\label{EQ:sheaf-cond-intro}
	\res^\calF_{e\from v}\circ\res^\calF_{v\from \emptyset} = \res^\calF_{e\from \emptyset}
\end{equation}
for every edge $e$ and vertex $v$ of $e$.
The maps $\res^{\calF}_{y\from x}$
are   the \emph{restriction maps} of $\calF$.
We will usually drop the superscript $\calF$ from $\res^\calF_{y\from x}$
when there is no risk of confusion.

The simplest example of a  sheaf on $X$ is obtained by choosing an abelian group
$R$ and setting
$\calF(x)=R$ for every $x\in X$  and $\res^{\calF}_{y\from x}=\id_R$ for every $x\subsetneq y\in X$.
We denote this   sheaf by $\aug{R}_X$; sheaves of this form (up to isomorphism)
are called \emph{constant} sheaves on $X$.

Note that if one takes $\calF(\emptyset)=0$, then condition \eqref{EQ:sheaf-cond-intro} holds automatically. This gives rise to numerous examples of sheaves. More sophisticated
examples will be considered later.

\subsubsection*{Coboundary Expansion of Sheaves on Graphs}

Let $(X,w)$ be a weighted graph.
We can replace the role of the coefficient group $R$ in the definition
of $\cb_0(X,w,R)$ with a general   sheaf on $X$.

In more detail, let $\calF$ be a    sheaf on $X$ such that $\calF(x)\neq 0$ for every nonempty
face $x\in X$.
The \emph{$i$-cochains} ($i\in\{-1,0,1\}$)
of $X$ with coefficients in $\calF$ are memebers of $C^i=C^i(X,\calF):=\prod_{x\in X}\calF(x)$.
That is, every $f\in C^i$ consists of a collection $(f(x))_{x\in X(i)}$ where $f(x)\in\calF(x)$ for
every $x\in X(i)$.
We define the coboundary maps $d_{-1}:C^{-1}\to C^0$ and $d_0:C^0\to C^1$
as in the case of $\F_2$-coefficients, but with the difference that the restriction maps
of $\calF$ are invoked:
\begin{align}\label{EQ:differential-sheaves}
(d_{-1}f)(v) &= \res_{v\from \emptyset} f(\emptyset)  & & \forall v\in X(0), \\
(d_{0} f)(e) &= \res_{e\from e^+} f(e^+) -\res_{e\from e^-} f(e^-) & &\forall e\in X(1).\nonumber
\end{align}
Again, we have $d_0\circ d_{-1}=0$, and so $B^0(X,\calF):=\im  d_{-1} \subseteq \ker d_0 =: Z^0(X,\calF)$. We shall abbreviate $B^0(X,\calF)$ and $Z^0(X,\calF)$
to $B^0$ and $Z^0$ when there is no risk of confusion.
The quotient $\HH^0(X,\calF):=Z^0 /B^0 $ is the $0$-th \emph{cohomology group} of $\calF$. 
The $0$-coboundary expansion of $(X,w,\calF)$, or just $\calF$,
is the smallest non-negative real number $\cb_0(X,w,\calF)$ such that
\[
\|d_0 f\| \geq \cb_0(X,w,\calF) \dist(f,B^0)\qquad\forall f\in C^0(X,\calF),
\]
where again, $\dist( , )$ is the weighted Hamming distance  on $C^0$ or $C^1$
given by $\dist(f,g)=w(\supp(f-g))$. 
Note that $\cb_0(X,w,\calF)$ can be positive only if $B^0=Z^0$,
or equivalently,
$\HH^0(X,\calF)=0$.

Our earlier discussion of coboundary expansion with coefficients in an abelian group
$R$ can now be seen as addressing the
special case of a constant  sheaf $\aug{R}_X$. 

\medskip

The main result of this work bounds the $0$-dimensional
coboundary expansion of some  special \emph{non-constant}   sheaves
using the spectral expansion $\lambda(X,w)$ of their underlying weighted graph $(X,w)$.
This gives rise to the first examples of \emph{non-constant} sheaves on sparse graphs
having good coboundary expansion.
Before describing this result, we first explain why the non-constant case is   difficult
to handle.

\subsubsection*{What Cannot Be Said in General}

Let $\calF$ be a    sheaf on a weighted graph $(X,w)$.
In contrast with constant case $\calF=\aug{R}_X$, when $\calF$
is general,
using 
(some function of) $\lambda(X,w)$ in
order to bound $\cb_0(X,w,\calF)$ from below  is impossible, even if we impose the necessary requirement $Z^0(X,\calF)=B^0(X,\calF)$.
Indeed, since the restriction maps are not required to be injective,
there is no reason to expect that the boundary of
a $0$-cochain $f\in C^0(X,\calF)$ will have support   that is proportional in weight to that of $f$. The following example
makes this intuition precise.

\begin{example}
	Let $(X,w)$ be any weighted  graph, and let $R$ be a nonzero abelian group.
	Define a sheaf $\calF$ on $X$ by setting:
	\begin{itemize}
		\item $\calF(x)=R$ for every $x\in X(0)\cup X(1)$,
		\item $\calF(\emptyset)=R^{X(0)}$,
		\item $\res^\calF_{e\from v}=0$ for every $e\in X(1)$ and $v\in X(0)$ with $v\subseteq e$.
		\item $\res^{\calF}_{v\from \emptyset}:R^{X(0)}\to R$ is the projection onto the
		$v$-component for every $v\in X(0)$.
	\end{itemize}	 
	One readily checks that $B^0(X,\calF)=Z^0(X,\calF)=R^{X(0)}$ and that $d_0(f)=0$ for every $f\in C^0(X,\calF)$. Thus,
	$\cb_0(X,w,\calF)=0$, regardless of what   $\lambda(X,w)$ or $h(X,w)$ are.
\end{example}

It is tempting to hope that the problem
highlighted in the example would be solved
if we would  require all the restriction maps $\res_{e\to v}$ ($v\in X(0)$, $e\in X(1)$)
to be injective. However,
we show that this is still not the case in the more sophisticated Example~\ref{EX:locally-constant-bad-cbe}.

\subsubsection*{The Sheaves which We Study}

Since addressing the general case is futile, we focus in this work on a special kind of  
sheaves, namely, \emph{quotients} of  constant   sheaves by a ``small'' subsheaf.

In more detail, let $(X,w)$ be a weighted graph and let $\calF$ be a  sheaf on $X$. As expected,
a \emph{subsheaf} of $\calF$ is a  sheaf $\calG$ on $X$ such that $\calG(x)$ is a subgroup 
of $\calF(x)$
for every $x\in X$, and the restriction maps of $\calG$ agree with those of $\calF$.
In this case, one can form the quotient sheaf $\calF/\calG$, which assigns
every $x\in X$ the abelian group $\calF(x)/\calG(x)$, and has the evident restriction maps;
see \cite[Example~4.1]{First_2023_sheaves_on_complexes_preprint} for details.
In the special case where $\calF=\aug{R}_X$ for an abelian group $R$,
specifying a subsheaf of $\calF$ amounts merely to specifying a subgroup $\calG(x)$ of
$R$ for every $x\in X$, subject to the requirement that $\calG(x)\subseteq \calG(y)$
whenever $x\subseteq y$. This can be simplified even further: choose a subgroup $R_x\subseteq R$
for every $x\in X$ (including $x=\emptyset$), 
and then put $\calG(x)=\sum_{y\subseteq x} R_y$. That is, put:
\begin{itemize}
	\item $\calG(\emptyset)=R_\emptyset$,
	\item $\calG(v)=R_\emptyset+R_v$ for all $v\in X(0)$, and
	\item $\calG(e)=R_\emptyset+R_u+R_v+R_e$ for all $e\in X(1)$, where $u,v$ are the vertices of $e$.
\end{itemize}
One can quickly reduce to the case where $R_\emptyset=0$, so we will assume this henceforth.

Now consider the quotient sheaf $\aug{R}_X/\calG$.
As the following example shows, even in this restricted setting, $\cb_0(X,w,\aug{R}_X/\calG)$ may  be $0$ if no assumption
is made on the subgroups $\{R_x\}_{x\in X}$.

\begin{example}\label{EX:problem-in-quotient-sheaves}
Suppose that $(X,w)$
is a weighted
graph and  $R$ is a vector space $V$ of some large dimension   over a field $\F$.
Choose some nonconstant $f\in C^0(X,\aug{V}_X)=V^{X(0)}$ and put $g=d_0 f\in C^1(X,\aug{V}_X)$.
Set $R_v=0$ for every vertex $v\in X(0)$, and for every $e\in X(1)$,
let $R_e=\F\cdot g(e)$ --- a subspace of $V$ of dimension $1$ or $0$.
While $\aug{V}_X/\calG$ may seem very ``close'' to $\aug{V}_X$, we actually
have $\cb_0(X,w,\aug{V}_X/\calG)=0$ even when $\cb_0(X,w,\aug{V}_X )\geq \frac{1}{2}h(X,w)$ is large.
Indeed, since $\calG(v)=0$ for every $v\in X(0)$, we have $C^0(X,V_X)=C^0(X,V_X/\calG)$, and so  
we may view $f$ as an element of $C^0(X,\aug{V}_X/\calG)-B^0(X,\aug{V}_X/\calG)$.
By construction $d_0f=0$ in $C^1(X,\aug{V}_X/\calG)$, so $\HH^0(X,V_X/\calG)\neq 0$
and $\cb_0(X,V_X/\calG)=0$. 
\end{example}

The problem demonstrated in the example can be overcome by imposing some \emph{linear disjointness}
assumptions on the $\{R_x\}_{x\in X}$. Here, a finite
collection of subgroups $\{R_i\}_{i\in I}$ of $R$ is said to be
\emph{linearly disjoint}
if the summation map $(r_i)_{i\in I} \mapsto \sum_i r_i: \prod_{i\in I} R_i\to R$ is injective.
For instance,   in Example~\ref{EX:problem-in-quotient-sheaves}, 
if the edges $e_1,\dots,e_\ell$ form a cycle in $X$,
then $g(e_1)+\dots+g(e_\ell)=0$, which means that $R_{e_1},\dots,R_{e_\ell}$
are \emph{not} linearly disjoint (unless all the $g(e_i)$ are $0$). 
Our main result says that if we do impose some linear disjointness assumptions,
then $\cb_0(X,w,\aug{V}_X/\calG)$ will be large provided
that $(X,w)$ is a good spectral expander.

\subsubsection*{The Main Result}

Let $(X,w)$ be a weighted graph on $n$ vertices, let $R$ be an abelian group and let $\{R_x\}_{x\in X}$
be subgroups of $R$ with $R_\emptyset=0$. Define the subsheaf $\calG$ of $\aug{R}_X$ as before,
and
suppose that the following   linear disjointness assumptions holds: 
\begin{enumerate}[label=(\arabic*)]
	\item  For every subgraph $Y$ of $X$ which is either a cycle
		of length $\leq \ceil{\frac{2}{3}n}$ or a path of a length $\leq 2$
		(see \S\ref{subsec:complexes}),
		the summation map $\prod_{y\in Y(0)\cup Y(1)} R_y\to R $ is injective.
	\item For every distinct $u,v\in X(0)$, we have $R_u\cap R_v=0$.
\end{enumerate}
We show in Theorem~\ref{TH:cbe-for-quotient-sheaves} that under these hypotheses,
we have
\begin{equation}\label{EQ:main-res-intro}
\cb_0(X,w,\aug{R}_X/\calG)\geq \frac{2}{5}-\frac{8}{5}\lambda(X,w)-\frac{7}{5}t,
\end{equation}
where 
$t$ is the maximum of $\frac{w(e)}{w(v)}$ taken over all pairs of a vertex $v$
and edge $e$ with $v\subseteq e$.
(For example, $t=\frac{2}{k}$    
if $X$ is a $k$-regular graph and $w$ is given by $w(v)=\frac{1}{|X(0)|}=\frac{1}{n}$ for $v\in X(0)$
and $w(e)=\frac{1}{|X(1)|}=\frac{2}{kn}$ for $e\in X(1)$.)

We also prove a variant of this result for \emph{$(r+1)$-partite weighted graphs},
i.e., the underlying graphs of $(r+1)$-partite weighted $r$-dimensional simplicial complexes.
Such weighted graphs always have $-\frac{1}{r}$ as an eigenvalue, which makes
the right hand side of \eqref{EQ:main-res-intro} negative if $r$ is too small, e.g., if $X$ is a bipartite graph.
In Theorem~\ref{TH:cbe-for-partite-quotient-sheaves}, we show that the eigenvalue $-\frac{1}{r}$
can be ignored at the expense of getting a slightly smaller (but still positive) 
lower bound on $\cb_0(X,w,\aug{R}_X/\calG)$.

\begin{remark}
The reason why \eqref{EQ:main-res-intro}
fails in Example~\ref{EX:problem-in-quotient-sheaves} (for $\lambda(X,w)$ and $t$ sufficiently
small) is because condition (1) does not hold. 
Indeed, we noted earlier  that for the $\{R_x\}_{x\in X}$ of that example, if $e_1,\dots,e_\ell$ are the edges of a cycle in $X$,
then
the groups $R_{e_1},\dots,R_{e_\ell}$ are not linearly disjoint, unless all are zero.

Nevertheless, if we further assume that all the $R_e$ ($e\in X(1)$)
are $0$, then it may be possible to relax condition (1).
Under this assumption, it seems plausible that   a result similar
to Theorem~\ref{TH:cbe-for-quotient-sheaves} should hold
when   every  
$o(n)$ of the $R_v$ ($v\in X(0)$)  are linearly disjoint, but we do not know how to push below $\Theta(n)$
--- even getting our present result  was difficult. 
Such an improvement is desirable since presently, in order to assert
\eqref{EQ:main-res-intro}, $R$ needs to be very large with respect to the $R_v$
if they are nonzero (e.g., if $R$ is a vector space
and all the $R_v$ are subspaces of dimension $c$,
then $\dim R$ must be at least $\frac{2}{3}cn$).
Either way, as we shall shortly see, despite the restrictive condition (1),
Theorem~\ref{TH:cbe-for-quotient-sheaves}   has some applications.
\end{remark}

\subsubsection*{About The Proof }

The proof of   Theorems~\ref{TH:cbe-for-quotient-sheaves}
and~\ref{TH:cbe-for-partite-quotient-sheaves} is somewhat involved. 
Broadly speaking, given a $0$-cochain $f\in C^0(X,\aug{R}_X/\calG)$
such that its coboundary has 
small support, we restrict $f$ to special subgraphs $Y$ of $X$, e.g.\ cycles,
showing that $f $ agrees with some   $g\in B^0(X,\aug{R}_X/\calG)$  on that $Y$ (here $g$  
depends on $Y$). We consider the maximal subgraphs $Y$ having the property that $g$
exists and study their structure to eventually
show that at least one of them has  a large weight. The existence of this large
suchgraph means that $f$ cannot be too far from $B^0(X,\aug{R}_X/\calG)$, 
and that is enough to get \eqref{EQ:main-res-intro} with a little
more work.

\subsubsection*{A Side Result: Expander Mixing Lemma for Weighted $(r+1)$-Partite Graphs}

In the course of proving our main result, we   established
variants of the Expander Mixing Lemma (EML) for
weighted graphs (Theorem~\ref{TH:EML}) and  
weighted $(r+1)$-partite graphs 
(Theorem~\ref{TH:EML}), which were not available in the literature. 
Given a weighted graph $(X,w)$, 
our weighted EML states that for every $A,B\subseteq X(0)$
with $\alpha=w(A)$ and $\beta=w(B)$,
one has
$|\frac{1}{2}w(E_{\ord}(A,B))-\alpha\beta|\leq  \lambda(X,w)\sqrt{\alpha(1-\alpha)\beta(1-\beta)}$,
where $ E_{\ord}(A,B) $ is the set of   oriented edges from $A$ to $B$;
this is well-known when $X$ is regular and $w$ is uniform.
When $X$ is $(r+1)$-partite, $-\frac{1}{r}$ is an eigenvalue of $(X,w)$ (occurring with multiplicity $r$),
which means that $\lambda(X,w)\geq \frac{1}{r}$.
Our EML for weighted $(r+1)$-partite graphs is similar to the non-partite one, except one is allowed
to ignore the eigenvalue $-\frac{1}{r}$ when defining $\lambda(X,w)$.
These variants of the EML
may be useful elsewhere.

\subsubsection*{The Case of Spherical Buildings}

\emph{Spherical buildings} are an important  class of similicial complexes
admitting special structural properties; see \cite{Abramenko_2008_Buildings}.
Gromov conjectured that the coboundary expansion
of all spherical buildings of dimension $\leq d$
is bounded away from zero in all dimensions if the coefficient group is $\F_2$,
and this was affirmed in \cite{Lubotzky_2016_expansion_of_buildings} for the coefficient group
$\F_2$ (with a natural choice of weights), 
and for a general constant coefficient group   in \cite{Kaufman_2018_cosystolic_expanders}.
We improve these results for $0$-coboundary expansion, and then apply our main result
to bound from below the $0$-coboundary expansion of some non-constant sheaves on spherical buildings.
We note that here it is crucial that
our main result applies in the generality of multipartite graphs carrying a general weight
function.

In more detail, let $(X,w)$ be the weighted graph underlying a finite $r$-dimensional
$q$-thick spherical building; such $X$ is $(r+1)$-partite.
We show in Theorem~\ref{TH:good-expansion-of-buildings}
that all the eigenvalues of $(X,w)$  except  
$1$ and $-\frac{1}{r}$   lie in the interval
$[-r\lambda,\lambda]$ with $\lambda= O(\frac{1}{\sqrt{q}-3r})$,
provided $q>(3r)^2$.
By plugging this   into our main results, we conclude that
$\cb_0(X,w,\aug{R}_X)\geq 1-O(\frac{1}{\sqrt{q}-3r})$ for every abelian group 
$R\neq 0$ (Corollary~\ref{CR:cbe-buildings-constant}), and
for a subsheaf $\calG$ of $\aug{R}_X$ as before, we have
\[
\cb_0(X,w,\aug{R}_X/\calG)\geq  \frac{2r}{5r+2}-O(\frac{r^2}{\sqrt{q}-3r}),
\]
provided (1) and (2) hold (Corollary~\ref{CR:expansion-of-quotient-sheaves-on-buildings}). These results also
apply to  finite simplicial complexes
covered by a $q$-thick \emph{affine building} of dimension $r\geq 2$.

We also note that our bound $\lambda_2(X,w)=O(\frac{1}{\sqrt{q}-3r})$ (for $q> 9r^2$)
implies that the $q$-thick spherical building 
$X$ is an  $O(\frac{1}{\sqrt{q}-3r})$-skeleton expander
(see \cite[\S7.2]{First_2024_cosyst_exp_posets_preprint}, for instance).
This improves a result of Evra and the second author
\cite[Theorem~5.19]{Evra_2016_cosystolic_expanders_arxiv_version}
who showed that $X$ is an $O(\frac{2^r(r+1)!}{\sqrt{q}})$-skeleton expander.

\subsubsection*{Implications to Cosystolic Expansion}

Cosystolic expansion is a more lax version of coboundary expansion.
It was   introduced in \cite{Dotterrer_2018_topological_overlap},
\cite{Kaufman_2016_isoperimetic_inequalities}, \cite{Evra_2016_cosystolic_expanders_arxiv_version}
in order to extend the reach of Gromov's  work \cite{Gromov_2010_expanders_and_top_II} 
which relates the coboundary expansion of a simplicial complex 
and the minimal amount of overlapping
forced by  mapping it into Euclidean space of the same dimension.
It was further noted in \cite{Kaufman_2014_high_dim_expanders_property_testing}
that cosystolic expansion is related to \emph{locally testable codes} and \emph{quantum CSS codes}.
These works all use $\F_2$ as the implicit coefficient group, but,
as observed by the authors in 
\cite{First_2023_sheaves_on_complexes_preprint}
and \cite{First_2024_cosyst_exp_posets_preprint}, cosystolic expansion can be defined for any sheaf on 
a simplicial complex, or even a graded poset.

Restricting to the case of weighted graphs,
the \emph{$0$-cosystolic expansion} of a sheaf  $\calF$ on a weighted graph $(X,w)$ is the
smallest non-negative real number $\cse_0(X,w,\calF)$ for which
\begin{equation}\label{EQ:cse-zero}
\|d_0 f\| \geq \cse_0(X,w,\calF) \dist(f,Z^0(X,\calF))
\qquad \forall \, f\in C^0(X,\calF).
\end{equation}
Thus, if $Z^0(X,\calF)=B^0(X,\calF)$, then $\cse_0(X,w,\calF)$ is the same
thing as $\cb_0(X,w,\calF)$. However, while  $\cb_0(X,w,\calF)=0$ if $Z^0\neq B^0$,
the $0$-cosystolic expansion may be positive in this case. 

Cosystolic expansion
is usually studied together with the \emph{$0$-cocycle distance}, 
\[\ccd_0(X,w,\calF)=\sup\{\|f\|\where f\in Z^0-B^0\}\]
and a family of sheaved graphs  is said to have good $0$-cosystolic expansion if both their
$0$-cosystolic expansion and $0$-cocycle distance are bounded away from $0$.

Constructing \emph{dense} coboundary expanders turned out to be a key ingredient  in constructing
\emph{sparse} cosystolic expanders. (Here, as usual, a family of graphs or simplicial complexes is called sparse
	if the number of faces of every member is linear in the number of its 
	vertices, and dense otherwise.)
More precisely, if $\calF$ is a sheaf on a (weighted) simplicial complex $X$,
then $\calF$ is a good $i$-cosystolic expander when the restriction of $\calF$ to each
proper \emph{link} of $X$ has sufficiently good coboundary expansion in a range of dimensions.
This principle, sometimes called the \emph{local criterion for cosystolic expansion},
was first observed in \cite{Kaufman_2016_isoperimetic_inequalities} for triangle complexes
and the constant sheaf $ (\F_2)_X$, and has  been repeatedly improved in a series of works
\cite{Evra_2016_cosystolic_expanders},
\cite{Kaufman_Mass_2021_cosystolic_non_abelean},
\cite[\S8]{First_2023_sheaves_on_complexes_preprint},
\cite[Theorem~7]{Kaufman_Mass_2022_improved_cosystole},
\cite{Dikstein_2023_cbe_cse_without_dep_on_dim_deg}
until it was shown to hold for all sheaves on all graded posets in \cite{First_2024_cosyst_exp_posets_preprint}
(for constant sheaves,
better  expansion constants are achieved in
\cite[Theorem~7]{Kaufman_Mass_2022_improved_cosystole} and
\cite{Dikstein_2023_cbe_cse_without_dep_on_dim_deg}).

Plugging our main results about coboundary expansion and
their specializations to spherical buildings  
into to the local criterion for cosystolic expansion
gives rise to new examples of good sparse $0$-cosystolic expanders, which are detailed
in \cite[Theorem~9.5]{First_2023_sheaves_on_complexes_preprint}.
These good $0$-cosystolic expanders are non-constant sheaves on   \emph{Ramanujan
complexes}; our results here are used for checking that the proper links of these sheaved
simplicial complexes, which are sheaved spherical buildings, have large $0$-coboundary expansion.

\subsubsection*{Implications to Locally Testable Codes}

Our results have some implications to \emph{locally testable codes} (LTCs).
We refer the reader to \cite[\S2.3]{First_2024_cosyst_exp_posets_preprint},
for instance, for a concise exposition of all the relevant definitions.

Informally, an LTC is an error correcting code
$C\subseteq \Sigma^n$ (as usual, $n\to \infty$ and $C$ varies with $n$,
but $\Sigma$ remains constant) together with a probabilistic
algorithm, called a \emph{tester}, which can decide with high probability
whether a word in $\Sigma^n$ belongs to $C$ by reading just $O(1)$ of its letters.
More precisely, the chances of correctly detecting that a word $f$ is not in $C$
are at least proportional to the  relative Hamming distance of $f$ from $C$.
The question of whether there exist
LTCs which are \emph{good}, i.e.,
have  \emph{rate} and \emph{relative distance} bounded away from $0$, was open until it
was answer on the positive in \cite{Dinur_2022_ltcs_const_rate} and \cite{Panteleev_2022_good_quantum_codes}  independently.

It was observed in
\cite{Kaufman_2014_high_dim_expanders_property_testing},
\cite{First_2023_sheaves_on_complexes_preprint}
and \cite{First_2024_cosyst_exp_posets_preprint} 
that sheaves with good cosystolic expansion give rise to LTCs with linear distance.
Indeed, restricting to the case of sheaves on graphs for simplicitly,
suppose that $\calF$ is a sheaf on a weighted graph $(X,w)$ for which
there is  an abelian group $\Sigma\neq 0$ such that $\calF(v)=\Sigma$ for every
$v\in X(0)$. Then $C^0(X,\calF)=\Sigma^{X(0)}$, and so we may think of
$Z^0=Z^0(X,\calF)$ as a code inside $\Sigma^{X(0)}$. The code $Z^0$ --- called a \emph{$0$-cocycle code} --- has a natural   tester:
Given $f\in \Sigma^{X(0)}=C^0$,
pick an edge $e\in X(1)$ at random according to the distribution $w|_{X(1)}$,
read the letters $f(e^+)$ and $f(e^-)$, and accept $f$ (i.e., estimate
that it belongs to $Z^0$)
if and only if $d_0 f(e)=\res_{e\from e^+} f(e^+)-\res_{e\from e-}f(e^-)$ is $0$,
cf.\ \eqref{EQ:differential-sheaves}.
Clearly, this tester accepts all $f$ in $Z^0$, and the probability
that it rejects an $f\in C^0-Z^0$ is precisely $\|d_0f\|$. 
Putting $\mu=\cse_0(X,w,\calF)$, we may therefore rewrite
\eqref{EQ:cse-zero} as
\begin{equation}\label{EQ:soundness}
\mathrm{Prob}(\text{tester rejects $f$})
\geq 
\mu \dist(f,Z^0)
\qquad
\forall\, f\in \Sigma^{X(0)}.
\end{equation}
In particular, provided that $\mu>0$, the further away $f$ is in (weighted) Hamming
distance from $Z^0$, the higher the rejection chances are.
Assuming $w$ is uniform on the vertices of $X$, 
\eqref{EQ:soundness} is equivalent to saying that the tester we described
has \emph{soundness} $\mu$. Since an LTC is a code with a tester having soundness bounded away from $0$,
a family of sheaves with $0$-cosystolic expansion bounded away from $0$ would give rise to an LTC.
Similarly, assuming that $\calF(\emptyset)=0$, the $0$-cocycle distance $\ccd_0(X,w,\calF)$ is the relative distance of the code $Z^0\subseteq C^0=\Sigma^{X(0)}$.
Thus, sheaves with good cosystolic expansion give rise locally testable codes
with linear distance. 

Our results about $0$-coboundary expansion of graphs were
used together with the local criterion for cosystolic expansion
in \cite[\S12, Cor.~12.5]{First_2023_sheaves_on_complexes_preprint}
to construct good $0$-cosystolic expanders which give rise to LTCs
with linear distance and conjecturally large rate.

We note in passing that while the first good LTCs  of \cite{Dinur_2022_ltcs_const_rate} and \cite{Panteleev_2022_good_quantum_codes} did not arise
from sheaves with good cosystolic expansion, we 
showed in \cite{First_2024_cosyst_exp_posets_preprint} 
(see also \cite{First_2024_cosyst_exp_posets_stoc})
that there are sheaves (on square complexes) which do
give rise to good LTCs and the fact that they are indeed
LTCs may be verified using the local criterion for cosystolic expansion,
but this does not use the results of this paper.

We finally remark that our main result 
also implies that sheaves of the form $\aug{R}_X/\calG$ also have
good $0$-cosystolic expansion (Corollary~\ref{CR:cse-quotient-sheaves})
and thus gives  new examples
of LTCs. These examples, however, have very small rate, or very large alphabet.

\subsubsection*{Outline}

The paper is organized as follows:
Section~\ref{sec:prelim} is preliminary and recalls necessary facts about simplicial complexes,
weight functions, spectral expansion and multipartite graphs.
In Section~\ref{sec:mixing}, we prove weighted and
multipartite versions of the Cheeger Inequality
and the Expander Mixing Lemma.
Section~\ref{sec:sheaves} recalls sheaves on graphs,
their cohomology and their expansion.
In Section~\ref{sec:cbe-constant}, we show that an expander graph with a constant   sheaf
is a good coboundary expander in dimension $0$, but that this may fail for \emph{locally constant sheaves}.
Section~\ref{sec:cbe-quotients} is dedicated to proving our main result: 
the   quotient of a constant  sheaf on a sufficiently good expander graph 
by a ``small'' subsheaf
has good $0$-coboundary expansion.
This result is applied to finite spherical buildings in Section~\ref{sec:buildings}.
Finally, in Section~\ref{sec:questions}, we raise some questions about potential improvements
of our results.

\subsection*{Acknowledgments}

Uriya First is supported by an ISF grant no.\ 721/24.
Tali Kaufman is supported by an ERC grant and an ISF grant.

\section{Preliminaries}
\label{sec:prelim}

\subsection{Simplicial Complexes}
\label{subsec:complexes}

Recall that a simplicial complex $X$ is a nonempty set of finite sets
with the property that $x\in X$ implies $y\in X$ for every $y\subseteq x$.
If not indicated otherwise, similicial complexes (and graphs)   are assumed to be finite.

The elements of $X$ are called \emph{faces}.
The dimension of a face $x\in X$ is $\dim x=|x|-1$ and the dimension of $X$
is the supremum of the dimensions of its faces.

A face of dimension $i$
in  $X$ is also called an $i$-face, and the set of $i$-faces is denoted $X(i)$.
Notice that $X$ has a single face of dimension $-1$, namely, the empty face $\emptyset$.
The \emph{vertex set} of $X$ is the union of all the faces of $X$ and is denoted
$V(X)$; every face of $X$ is a subset of $V(X)$. There is a one-to-one correspondence between $V(X)$ and $X(0)$ given
by $v\leftrightarrow \{v\}$,
and when there is no risk of confusion, we will treat vertices as $0$-faces
and vice versa. Members of $X(1)$ are also called edges. 

In this work, a  \emph{graph} is a  simplicial complex of dimension $1$ or less.
In particular, graphs are non-oriented and have no loops or double edges
(but double edges can be accounted for using   weight functions discussed in \S\ref{subsec:weights}).

The $k$-dimensional skeleton of $X$ is the simplicial complex
$X({\leq} k):=\bigcup_{i=-1}^k X(i)$.
Given $z\in X$ and $A\subseteq X$,   we write 
\begin{align*}
A_{\supseteq z}&=\{y\in A\suchthat y\supseteq z\}, \\
A_{\subseteq z}&=\{y\in A\suchthat y\subseteq z\}, \\
A_z&=\{y-z\where y\in A_{\supseteq z}\}.
\end{align*}
For example, if $v\in X(0)$ is a $0$-face, then
$X(1)_{\supseteq v}$ is the set of edges  containing $x$,
and $X(1)_v$ is the set $0$-faces adjacent to $v$.
For a general $z\in X$, the set $X_z$ is known as the \emph{link} of $X$ at $z$ and is   a simiplicial
complex.

An ordered edge in $X$ is a pair $(u,v)\in V(X)\times V(X)$ such that $\{u,v\}$ is an edge in $X$.
The set of ordered edges in $X$ is denoted $X_{\ord}(1)$.
Given $A,B\subseteq X(0)$,   we write $E (A,B)$ for
the set of   edges in $X$  with one vertex in $A$ and the other in $B$.
Likewise, 
$E_{\ord} (A,B)$ is the set of ordered edges $(u,v)\in X_{\ord}(1)$ with  $\{u\}\in A$ and
$\{v\}\in B$.
We also let $E(A)=E(A,A)$ and $E_{\ord}(A)=E_{\ord}(A,A)$.

The simplicial complex $X$ is said to be \emph{pure} of dimension $d$
($0\leq d\in \Z$) if every face of $X$ is contained in a $d$-face.
We then say that $X$ is a \emph{$d$-complex} for short.
The $k$-dimesnional skeleton of a $d$-complex is a $k$-complex for
all $k\in\{0,\dots,d\}$.

\medskip

Suppose now that $X$ is a graph and let $Y\subseteq X$
be a subset. 
We say that $Y$
is   a \emph{cycle} of length $\ell$ ($\ell\geq 3$)
if $Y$ is a subgraph of $X$ that is isomorphic to the cycle graph
on $\ell$ vertices.
Given $x,y\in X(0)$ ($x=y$ is allowed), we say
that $Y$ is  a
\emph{closed   path}\footnote{
	Here, ``closed'' should be understood as topologically
	closed, rather than having the same start and end point.
} of length $\ell$ ($\ell\geq 0$) from $x$
to $y$ if $Y$ is a subgraph of $X$
such that
$x,y \in Y(0)$,
$Y(0)-\{x,y\}$
contains exactly
$\ell-1$ distinct $0$-faces  
$x_1,\dots,x_{\ell-1}$,
and $Y(1)$ consists of  exactly $\ell$ edges which are
$x\cup x_1,x_1\cup x_2,\dots,x_{\ell-2}\cup x_{\ell-1},x_{\ell-1}\cup y $.
An \emph{open path} from $x$ to $y$  
is a subset $Z\subseteq X$ of the form $Y-\{x,y\}$,
where $Y$ is a closed path from $x$ to $y$ in $X$;
the length of $Z$ is defined to be the length of $Y$. An open path $Z$ is   \emph{not}
a subgraph of $X$, but we shall still write $Z(i)=Z\cap X(i)$ for $i\in\{0,1\}$.

The following easy lemma, whose proof is left to the reader, will be needed
in the sequel.

\begin{lem}\label{LM:cycle-minus-subgraph}
Let $X$ be a graph, let $X'$ be a subgraph of $X$
and let $C$ be a cycle in $X$ meeting $X'$.
Then   $C-X'$ is a disoint union of open paths.
\end{lem}

\subsection{Weights}
\label{subsec:weights}
 
A weight function on a $d$-complex $X$ is a function $w:X\to \R_+$
such that
\begin{enumerate}[label=(W\arabic*)]
\item  $\sum_{y\in X(d)} w(y)=1$,
\item  $w(x) =\schoose{d+1}{\dim x+1}^{-1} \sum_{y\in X(d)_{\supseteq x}} w(y)$ for all $x\in X$
with $\dim x<d$.
\end{enumerate}
We then say that $(X,w)$ is a \emph{weighted  $d$-complex};
if $d=1$ we say that $(X,w)$ is a \emph{weighted graph}.
If we   strengthen (W1) to   $w(y)=|X(d)|^{-1}$ for all $y\in X(d)$,
then this recovers the   weight functions considered in 
\cite{Lubotzky_2016_expansion_of_buildings} and \cite{Kaufman_2018_cosystolic_expanders}.

Given a weight   function $w:X\to \R_+$ and $A\subseteq X$, we write $w(A)=\sum_{a\in A}w(a)$.
A similar convention applies to subsets of $X_{\ord}(1)$, where the weight of an ordered
edge is defined to be the weight of its underlying unordered edge.

The conditions (W1) and (W2) imply that $w $ restricts to a probability measure on $X(d)$,
and that for  $x\in X(i)$ with $i\in \{-1,\dots,d-1\}$, 
the value  $w(x)$ is the probability of obtaining $x$ by choosing
a $d$-face $y\in X(d)$ according to $w$ and then choosing an $i$-face of $y$
uniformly at random. Consequently, $w(X(i))=1$ for all $i\in\{-1,\dots,d-1\}$.
It is also straightforward to check that (W2) implies
\begin{align}\label{EQ:weight-of-containing-cells}
w(X(\ell)_{\supseteq x})=\schoose{\ell+1}{k+1} w(x)
\end{align}
for all $-1\leq k\leq \ell\leq d$ and $x\in X(k)$.
As a result, if $(X,w)$ is a weighted $d$-complex,
then for every $k\in\{0,\dots,d\}$,
its $k$-dimensional skeleton $(X(\leq\! k),w|_{X(\leq k)})$ is  
a weighted $k$-complex.

\begin{example}\label{EX:regular-graph}
	(i) Let  $X$ be a $d$-complex. 
	The \emph{canonical weight function} $w_X:X\to \R_+$
	is defined
	by putting $w_X(y)=\frac{1}{|X(d)|}$ for all $y\in X(d)$ and defining $w_X$ on the other
	faces using the formula in  (W2).
	
	(ii) If  $X$ is a $k$-regular graph on $n$ vertices, then 
	$X$ is a pure $1$-dimensional simplicial complex, and
	the canonical weight
	function assigns every edge of $X$ the weight $\frac{2}{kn}$,
	every vertex of $X$ the weight $\frac{1}{n}$ and the value $1$ to the empty face.
	
	(iii) If $X$ is a  $d$-complex,
	and $0\leq k<d$, then $(X(\leq k),w_X|_{X(\leq k)})$
	is a weighted $k$-complex, but $w_X|_{X(\leq k)}$ is in general different
	from $w_{X(\leq k)}$, the canonical weight function of $X({\leq k})$.
\end{example}

\begin{remark}\label{RM:weights-a-la-Oppenheim}
	In \cite[Definition~2.1]{Oppenheim_2015_vanishing_of_cohomology} and \cite[Definition~3.2]{Horak_2013_spectra_of_Laplace_ops}, a balanced weight function
	on a $d$-complex is defined to be a function $m:X\to \R_+$
	such that $m(x)=(d-\dim x)! \cdot m(X(d)_{\supseteq x})$ for all $x\in X $.
	If $(X,w)$ in a weighted $d$-complex in our sense,
	then $m:X\to \R_+$ defined by $m(x)=\frac{(d+1)!}{(\dim x+1)!}w(x)$ is a balanced
	weight function in
	the sense of \cite{Oppenheim_2015_vanishing_of_cohomology} and
	\cite{Horak_2013_spectra_of_Laplace_ops}.
\end{remark}

\subsection{Expansion of Weighted Graphs}
\label{subsec:expansion}

Let $(X,w)$ be  a weighted graph. Given $i\in \{0,1\}$, write $C^i(X,\R)$ for the set of functions
$f:X(i)\to \R$. We endow $C^i(X,\R)$ with the inner product defined by
\begin{align*}
\Trings{f,g}=\frac{1}{(i+1)!}\sum_{x\in X(i)}f(x)g(x)w(x)
\end{align*}
for all $f,g\in C^i(X,\R)$. Given $A\subseteq X(0)$, we write $1_A$ for the function
in $C^0(X,\R)$ taking the value $1$ on $A$ and $0$ on $X(0)-A$.

The  \emph{weighted adjacency operator} of $(X,w)$,
denoted  $\calA=\calA_{X,w}$,
and the \emph{weighted Laplacian} of $(X,w)$,
denoted $\Delta=\Delta_{X,w}$,
are the linear operators from $ C^0(X,\R)$ to itself defined
by    
\begin{align*}
(\calA f)(v)&=\sum_{e\in X(1)_{\supseteq v}}\frac{w(e)}{2w(v)}f(e-v),\\
(\Delta f)(v)&=f(v)-\sum_{e\in X(1)_{\supseteq v}}\frac{w(e)}{2w(v)}f(e-v),
\end{align*}
for all $f\in C^0(X,\R)$ and $v\in X(0)$. We have $\calA =\id-\Delta$.
For example, if $X$ is a $k$-regular graph and $w$ is the canonical weight function of $X$,
then $\calA_{X,w}$ is the ordinary adjacency operator of $X$ scaled
by $\frac{1}{k}$, cf.\ Example~\ref{EX:regular-graph}(ii).

Fix a  linear ordering $L$ on $V(X)$. Given an edge $e=\{u,v\}\in X(1)$
with $u<v$ relative to $L$,
we write $e^+=\{v\}$ and $e^-=\{u\}$. The \emph{$0$-coboundary} map is the linear operator
$d_0=d_0^L:C^0(X,\R)\to C^1(X,\R)$ defined by
\begin{align*}
(d_0 f)(e)=f(e^+)-f(e^-),
\end{align*}
for all $f\in C^0(X,\R)$, $e\in X(1)$.
The  weighted  \emph{$1$-boundary} map is the dual operator 
$d_0^*:C^1(X,\R)\to C^0(X,\R)$ relative to the inner products
of $C^0(X,\R)$ and $C^1(X,\R)$.

\begin{lem}\label{LM:adjacency-basic-properties}
	Under the previous assumptions:
	\begin{enumerate}[label=(\roman*)]
		\item $\Delta= d_0^*d_0$. In particular, $\Delta$ is positive semidefinite and
		$\calA=\id-\Delta$ is self-adjoint.
		\item $ \Spec \Delta\subseteq [0,2]$ and $ \Spec \calA\subseteq [-1,1]$.
		\item 
		$\Delta 1_{X(0)}=0$ and $\calA 1_{X(0)}=1_{X(0)}$.
	\end{enumerate}
\end{lem}

\begin{proof}
	(i) This follows from \cite[Proposition~2.11]{Oppenheim_2015_vanishing_of_cohomology}.
	(Consult Remark~\ref{RM:weights-a-la-Oppenheim}. For a general weighted $d$-complex
	$(X,w)$, the inner product of $C^i(X,\R)$ used in {\it op.~cit.} is
	given by $\Trings{f,g}=\frac{(d+1)!}{(i+1)!}\sum_{x\in X(i)}f(x)g(x)w(x)$.
	The Laplacian $\Delta$ is denoted $\Delta^+_0$ in {\it op.~cit.})
	
	(ii) By (i), it is enough to prove that $\|\calA\|\leq 1$.
	Let $f\in C^0(X,\R)$ and $v\in X(0)$.
	By (W2), $\sum_{e\in X(1)_{\supseteq v}} \frac{w(e)}{2w(v)}=\frac{2w(v)}{2w(v)}=1$.
	Now, by Jensen's inequality,
	\[\left[\sum_{e\in X(1)_{\supseteq v}} \frac{w(e)}{2w(v)}f(e-v)\right]^2
	\leq \sum_{e\in X(1)_{\supseteq v}} \frac{w(e)}{2w(v)} f(e-v)^2 .\]
	Using this, we get
	\begin{align*}
	\|\calA f\|^2
	&=
	\sum_{v\in X(0)} w(v) \left[\sum_{e\in X(1)_{\supseteq v}} \frac{w(e)}{2w(v)}f(e-v)\right]^2
	\\
	&\leq \sum_{v\in X(0)} w(v) \sum_{e\in X(1)_{\supseteq v}}\frac{w(e)}{2w(v)} f(e-v)^2   
	=
	\sum_{v\in X(0)} \sum_{e\in X(1)_{\supseteq v}}  \frac{w(e)}{2 }f(e-v)^2
	\\
	&=\sum_{u\in X(0)}  \sum_{e\in X(1)_{\supseteq u}} \frac{w(e)}{2 }f(u)^2
	=\sum_{u\in X(0)}w(u)f(u)^2=\|f\|^2,
	\end{align*}
	which is what we want.

	(iii) Let $v\in X(0)$. Then $(\calA 1_{X(0)})(v)=\sum_{e\in X(0)_{\supseteq v}} \frac{w(e)}{2w(v)}=1=1_{X(0)}(v)$. This implies that $\Delta(1_{X(0)})=1_{X(0)}-\calA 1_{X(0)}=0$.
\end{proof}

	Let $(X,w)$ be a weighted graph.
	We define 
	\[
	C^0_{\circ}(X,\R)=1_{X(0)}^\perp =\{f\in C^0(X,\R)\suchthat \sum_{v\in X(0)}w(v) f(v)=0\}.
	\]
	By Lemma~\ref{LM:adjacency-basic-properties},
	$\Delta$ and $\calA$ take $C^0_{\circ}(X,\R)$ to itself. 
	Given $\mu\leq \lambda$ in $\R$, we say that $(X,w)$ is a \emph{ $[\mu,\lambda]$-spectral expander},
	or just a $[\mu,\lambda]$-expander for short,
	if $\Spec (\calA|_{C^0_{\circ}(X,\R)})\subseteq [\mu,\lambda]$.
	We also write $\lambda(X,w)$ for the smallest
	$\lambda\geq 0$ such that $\Spec(\calA|_{C^0_{\circ}(X,\R)})\subseteq [\lambda,-\lambda]$.

	\begin{remark}
		(i) At this level of generality, a weighted graph
		can be a $[\mu,\lambda]$-expander even
		when both $\mu$ and $\lambda$ are negative.
		For example, consider a complete graph $X$ on $n+1$ vertices with its canonical
		weight function $w=w_X$. It is well-known that eigenvalues of $\calA_{X,w}$
		are $1,-\frac{1}{n},\dots,-\frac{1}{n}$ (including multiplicities),
		so $(X,w)$ is a $[-\frac{1}{n},-\frac{1}{n}]$-expander. 
		
		(ii) We shall see below that if $(X,w)$ is a bipartite weighted
		graph, then $-1\in\Spec \calA_{X,w}$, so such a weighted
		graph is a $[\mu,\lambda]$-expander
		only if $\mu\leq -1$. In this case, we also have $\lambda(X,w)=1$. 
	\end{remark}

	We now recall Kaufman and Oppenheim's
	version of the Cheeger Inequality for graphs 
	\cite[Theorem~4.4(1)]{Kaufman_2021_amplified_local_testability_preprint},
	which	
	relates  $[-1,\lambda]$-expansion and 
	the Cheeger constant $h(X,w)$  (see the Introduction). 
	To that end, it is convenient to introduce
	the following variation on the Cheeger constant $h(X,w)$, namely,
	\[
		h'(X,w)=\min_{\emptyset\neq A\subsetneq X(0)} \frac{w(E(A,X(0)-A))}{2w(A)w(X(0)-A)} 
	\]
	(this is denoted $h_G$ in {\it op.\ cit.}).
	Informally, $h'(X,w)$ is minimum possible ratio between  
	the weight of the edges leaving $A$, and the expected weight if $(X,w)$ were to behave
	like a random graph.
	Since $\min\{\alpha,1-\alpha\}\leq 2\alpha(1-\alpha)\leq 2\min\{\alpha,1-\alpha\}$
	for all $\alpha\in [0,1]$, we have
	\begin{equation}\label{EQ:h-h-prime-relation}
	\frac{1}{2}h(X,w)\leq h'(X,w)\leq h(X,w).
	\end{equation}

	\begin{thm}[{\cite[Theorem~4.4(1)]{Kaufman_2021_amplified_local_testability_preprint}}]
		\label{TH:Cheeger}
		Let $(X,w)$ be a weighted graph which is also a $[-1,\lambda]$-expander.
		Then 
		\[h'(X,w)\geq 1-\lambda.\]
		That is, for every  $A\subseteq X$,
		we have
		$
		w(E(A,X(0)-A))\geq (1-\lambda)\cdot 2 w(A)(1-w(A))
		$.
	\end{thm}

	A theorem of Friedland and Nabben \cite[Theorem~2.1]{Friedland_2002_Cheeger_type_ineqs}
	implies  a converse to the theorem, namely,
	if $h(X,w)\geq \veps$, then $(X,w)$ is a 
	$[-1,\sqrt{1-\frac{\veps^2}{4}}]$-expander.
	(Specifically, assuming $V(X)=\{1,\dots,n\}$,
	apply\cite[Theorem~2.1]{Friedland_2002_Cheeger_type_ineqs}
	with $w_{i,j}=w(\{i,j\})$ and  $d_i=2w(\{i\})$.
	The numbers $\delta_i$ defined {\it op.\ cit.} are $2w(\{i\})$
	in our notation, and the constant $i(X,w)$ considered there equals $\frac{1}{2}h(X,w)$
	in our notation.)

\subsection{Partite Simplicial Complexes}
\label{subsec:partite-complexes}

	Recall that an $(r+1)$-partite simplicial complex
	is a tuple $(X,V_0,\dots,V_r)$
	such that $X$ is a simplicial complex, $V_0,\dots,V_r$
	is a partition of the set of vertices $V(X)$ (in particular, $V_i\neq \emptyset$
	for all $i$), and every face $x\in X$ contains at most one vertex
	from each $V_i$, i.e., $|x\cap V_i|\leq 1$ for all $i\in\{0,\dots,r\}$.
	We then write $X_{\{i\}}$ for the set of $0$-faces having their vertex in $V_i$.
	We say that $(X,V_0,\dots,V_r)$ is pure if every face of $X$ is contained in an
	$(r+1)$-face; in this case $\dim X=r $.
	A bipartite graph is just $2$-partite simplicial complex.	
	
	A \emph{weighted $(r+1)$-partite simplicial} complex
	is a tuple $(X,w,V_0,\dots,V_r)$ such that $(X,V_0,\dots,V_r)$
	is a pure $(r+1)$-partite simplicial complex and $w$ is a weight function on
	the $(r+1)$-complex $X$.
	
	When there is no risk of confusion, we will  suppress the partition
	$(V_0,\dots,V_r)$ from the notation, writing simply that $X$ is an $(r+1)$-partite complex,
	or that $(X,w)$ is a weighted $(r+1)$-partite simplicial complex.

	\begin{lem}\label{LM:partite-basis-props}
		Let $(X,w,V_0,\dots,V_r)$ be a weighted   $(r+1)$-partite simplicial complex.
		Then: 
		\begin{enumerate}[label=(\roman*)]
			\item 
			$w(X_{\{i\}})=\frac{1}{r+1}$ for all $i\in\{0,\dots,r\}$.
			\item 
			$w(E(X_{\{i\}},X_{\{j\}}))=\frac{2}{r(r+1)}$ for all distinct $i,j\in \{0,\dots,r\}$.		
			\item The subspace $L$ of $C^0(X,\R)$ spanned by 
			$1_{X_{\{0\}}},\dots,1_{X_{\{r\}}}$ 
			is invariant under $\calA=\calA_{X,w}$.
			The $r+1$ eigenvalues of $\calA$ on this subspace
			are $1,-\frac{1}{r},\dots,-\frac{1}{r}$ (including multiplicities).
			If the link  $X_z$ is connected for all 
			$z\in X({\leq r-2})$, then 
			$L$ is the sum of the $1$-eigenspace
			and $-\frac{1}{r}$-eigenspace of $\calA$.
		\end{enumerate}
	\end{lem}
	
	\begin{proof}
		We prove (i) and (ii) together.
		Given $I\subseteq \{0,\dots,r\}$, let $X_I$ denote
		the set of $(|I|-1)$-faces of $X$ having 
		a vertex in $V_i$ for all $i\in I$. 
		We claim that $w(X_I)=\schoose{r+1}{|I|}^{-1}$. Indeed,
		\begin{align*}
		w(X_{I})
		&=\sum_{x\in X_{I}} w(x)=
		\sum_{x\in X_{I}} \schoose{r+1}{|I|}^{-1}\sum_{y\in X(r+1)_{\supseteq x}} w(y)
		\\
		&=
		\schoose{r+1}{|I|}^{-1} \sum_{y\in X(r+1)} \sum_{x\in (X_{I})_{\subseteq y}} w(y) 
		=
		\schoose{r+1}{|I|}^{-1} \sum_{y\in X(r+1)} 1\cdot w(y)\\
		&=
		\schoose{r+1}{|I|}^{-1}  w(X(r+1))=\schoose{r+1}{|I|}^{-1} ,
		\end{align*}
		where the fourth equality holds because every $(r+1)$-face 
		contains exactly one face in $X_I$. (i) and (ii)
		now follow    by taking $I=\{i\}$ and $I=\{i,j\}$, respectively.
		
		Part (iii) follows from  \cite[Proposition~5.2]{Oppenheim_2015_vanishing_of_cohomology}
		and its proof.
	\end{proof}
	
	Given a weighted pure $(r+1)$-partite simplicial complex $(X,w,V_0,\dots,V_r)$,
	we 
	write $C^0_{\diamond}(X,\R)$ for $L^\perp$, where $L=\Span_{\R}\{1_{X_{\{0\}}},\dots,1_{X_{\{r\}}}\}$.
	In view of Lemma~\ref{LM:partite-basis-props}(iii),
	when regarding the spectrum of $\calA=\calA_{X,w}$ on $C^0(X,\R)$,
	it is reasonable to set aside the eigenvalues occurring on the subspace $L$.
	Thus,  
	we say that $(X,w,V_0,\dots,V_r)$
	is an \emph{
	$(r+1)$-partite $[\mu,\lambda]$-expander} if the spectrum
	of $\calA$ on $C_{\diamond}^0(X,\R)$ is contained in the interval $[\mu,\lambda]$.

	Oppenheim 
	\cite[Lemma~5.5]{Oppenheim_2015_vanishing_of_cohomology}\footnote{
		There is a typo in this source: the last inequality
		should be ``$1+\frac{1}{n}(1-\lambda(X))\leq \kappa(X)\leq 1+n(1-\lambda(X))$''.		
	}	
	showed that if $(X,w)$ is a weighted $(r+1)$-partite $[-1,\lambda]$-expander with $0\leq \lambda\leq \frac{1}{r}$,
	then $(X,w)$ is actually an $r$-partite $[-r\lambda,\lambda]$-expander.
	When, $r=1$, i.e., when $(X,w)$ is a weighted bipartite graph, it is further
	known that  the (multi-)set $\Spec \calA_{X,w}$ is symmetric around $0$;
	this follows from the following  well-known lemma.
	
	\begin{lem}\label{LM:bipartite-sym-spec}
		Let $(X,w,V_0,V_1)$ be a weighted bipartite graph,
		let $\lambda\in \R$ and let $f\in C^0_\diamond(X,\R)$
		be a $\lambda$-eigenfunction of $\calA=\calA_{X,w}$.
		Define $f'\in C^0_\diamond(X,\R)$
		by $f'(v)=f(v)$ if $v\in  X_{\{0\}}$ and $f'(v)=-f(v)$
		otherwise. Then $\calA f'=-\lambda f'$.
	\end{lem}

\section{Mixing Lemmas for Weighted and Multipartite Graphs}
\label{sec:mixing}

In this section, we prove two   versions
of the
Expander Mixing Lemma (EML)  applying to  weighted
graphs, and $(r+1)$-partite weighted simplical complexes,
respectively.
Non-weighted and bipartite-non-weighted  versions
of the EML are well-known,   but
the  weighted and multi-partite versions that we establish 
here seem missing in the literature.
Both results will be needed in the sequel
to establish our main results about coboundary expansion.

Recall that given a weighted graph $(X,w)$ and $A\subseteq X(0)$,
we write $1_A\in C^0(X,\R)$ for the function taking the value $1$ on $A$
and $0$ elsewhere.

\begin{lem}\label{LM:inner-prod-formula}
	Let $(X,w)$ be a weighted graph and let $A,B\subseteq X(0)$.
	Then
	\[ \Trings{\calA_{X,w} 1_A,1_B}=\frac{1}{2}w(E_{\ord}(A,B)). \]
\end{lem}

\begin{proof} By unfolding the definitions,
$\Trings{\calA 1_A,1_B}$ evaluates to  
	\begin{align*}
	\sum_{v\in X(0)} (\calA 1_A)(v)& 1_B(v) w(v) 
	=\sum_{v\in X(0)}\sum_{e\in X(1)_{\supseteq v}} \frac{w(e)}{2w(v)}1_A(e-v)1_B(v)w(v) 
	\\	
	&=\frac{1}{2}\sum_{v\in X(0)}\sum_{e\in X(1)_{\supseteq v}}1_A(e-v)1_B(v)w(e)
	=\frac{1}{2}w(E_{\ord}(A,B)). \qedhere
	\end{align*}
\end{proof}

\begin{thm}[Weighted Expander Mixing Lemma]
	\label{TH:EML}
	Let $(X,w)$ be a weighted graph 
	which is also a $[\mu,\lambda]$-expander ($-1\leq \mu\leq \lambda\leq 1$).
	Let $A,B\subseteq X(0)$ and put $\alpha=w(A)$, $\beta=w(B)$. Then:
	\begin{enumerate}[label=(\roman*)]
		\item $|\frac{1}{2}w(E_{\ord}(A,B))-\alpha\beta|
		\leq \max\{|\lambda|,|\mu|\}\sqrt{\alpha\beta(1-\alpha)(1-\beta)}$.
		\item $\mu \alpha (1-\alpha)
		\leq
		w(E(A))-\alpha^2
		\leq
		\lambda \alpha (1-\alpha)$.
	\end{enumerate}
\end{thm}

\begin{proof}
	By assumption, the eigenvalues of   $\calA$ on $C^0_{\circ}(X,\R)$ live
	in $[\mu,\lambda]$. Since $\calA$ is self-adjoint, this means that
	for every $f,g\in C^0_{\circ}(X,\R)$, we have
	\begin{eqnarray}
	& \mu \|f\|^2\leq \Trings{\calA f,f}\leq \lambda\|f\|^2, & \label{EQ:EML-eq1}\\
	& |\Trings{\calA f,g}|\leq \max\{|\lambda |,|\mu |\}\|f\| \|g\|. & \label{EQ:EML-eq2}
	\end{eqnarray}
	
	Put $f:=1_A-\alpha 1_{X(0)}$ and $g:=1_B-\beta 1_{X(0)}$.
	We first check that $f,g\in C^0_{\circ}(X,\R)$ and $\|f\|^2=\alpha(1-\alpha)$, $\|g\|^2=\beta(1-\beta)$.
	Indeed,
	\[
	\Trings{f,1_{X(0)}}=\Trings{1_A,1_{X(0)}}-\alpha\Trings{1_{X(0)},1_{X(0)}}
	=\sum_{x\in X(0)}(1_A(x) w(x)-\alpha w(x))=w(A)-\alpha w(X(0))=0,
	\]
	\begin{align*} 
	\|f\|^2&=
	\sum_{x\in A}(1-\alpha)^2 w(x)+\sum_{x\in X(0)-A}\alpha^2 w(x) 
	=
	\alpha(1-\alpha)^2+(1-\alpha)\alpha^2=\alpha(1-\alpha),  
	\end{align*}
	and a similar computation gives the analogous conclusions for $g$.

	Now, observe that
	\begin{align*}
	\Trings{\calA 1_A, 1_B} &= 
	\Trings{\calA (\alpha 1_{X(0)})+\calA f, \beta 1_{X(0)} + g}
	=
	\Trings{\calA (\alpha 1_{X(0)}),\beta 1_{X(0)}} +\Trings{\calA f,g} \\
	& = \alpha \beta \Trings{1_{X(0)},1_{X(0)}} +\Trings{\calA f,g} =\alpha\beta +
	\Trings{\calA f,g},
	\end{align*}
	where in the third equality we used Lemma~\ref{LM:adjacency-basic-properties}(iii).
	By Lemma~\ref{LM:inner-prod-formula},
	$\Trings{\calA 1_A, 1_B}=\frac{1}{2}w(E_{\ord}(A,B))$, so  
	\begin{align}\label{EQ:EML-eq3}
	\frac{1}{2}w(E_{\ord}(A,B)) - \alpha\beta =\Trings{\calA f,g}.
	\end{align}
	Now, by \eqref{EQ:EML-eq2}, 
	\[
	|\frac{1}{2}w(E_{\ord}(A,B)) - \alpha\beta|\leq \max\{|\lambda|,|\mu|\}\|f\|\|g\|
	=\max\{|\lambda|,|\mu|\}\sqrt{\alpha\beta(1-\alpha)(1-\beta)},
	\]
	which proves (i). 
	Also, taking   $A=B$ in \eqref{EQ:EML-eq3} and using \eqref{EQ:EML-eq1}
	gives
	\[
	\mu\alpha(1-\alpha)=\mu\|f\|^2	
	\leq \frac{1}{2} w(E_{\ord} (A)) - \alpha^2 \leq \lambda \|f\|^2=\lambda\alpha(1-\alpha),
	\]
	and (ii) follows because 
	$\frac{1}{2}w(E_{\ord}(A))=w(E(A))$.
\end{proof}

Recall from Lemma~\ref{LM:partite-basis-props} 
that if $(X,w,V_0,\dots,V_r)$
is a weighted   $(r+1)$-partite simplicial complex,
then $-\frac{1}{r}$ is an eigenvalue of $\calA$.
As a result, the constant $\max\{|\lambda|,|\mu|\}$ appearing in Theorem~\ref{TH:EML}(i)
is at least $\frac{1}{r}$. In particular, when $X$ is a bipartite graph,
this constant is $1$, and Theorem~\ref{TH:EML}(i) gives 
almost no  information about  $w(E_{\ord}(A,B))$.
We remedy this   in the following theorem.

\begin{thm}[Expander Mixing Lemma for Weighted $(r+1)$-Partite Graphs]
\label{TH:EML-bipartite}
Let $(X,w)$ be a weighted  $(r+1)$-partite 
simplicial complex that is an $(r+1)$-partite $[-\lambda,\lambda]$-expander. Let $A,B\subseteq X(0)$
and put $\alpha=w(A)$, $\beta=w(B)$,
$\alpha_i = w(A\cap X_{\{i\}})$ and $\beta_i = w(B\cap X_{\{i\}})$ ($i\in\{0,\dots,r\}$).
Then:
\begin{enumerate}[label=(\roman*)]
	\item If there are $T,S\subseteq \{0,\dots,r\}$ such
	that   $A\subseteq \cup_{i\in T} X_{\{i\}}$,
	$B\subseteq \cup_{j\in S} X_{\{j\}}$ and $S\cap T=\emptyset$,
	then
	$|\frac{1}{2}w(E(A,B))-\frac{r+1}{r}\alpha\beta|\leq \lambda(r+1)
	\sqrt{\alpha\beta (\frac{|T|}{r+1}-\alpha)(\frac{|S|}{r+1}-\beta)}$.
	\item In general,
	$
	\left|\frac{1}{2}w(E_{\ord}(A,B))-\frac{r+1}{r}[\alpha\beta-
	\sum_{i=0}^r\alpha_i\beta_i]\right|
	\leq \lambda r\sqrt{\alpha\beta(1-\alpha)(1-\beta)}$.	
	In particular, 
	\begin{align*}
	\frac{1}{2}w(E_{\ord}(A,B)) & \leq \frac{r+1}{r}[\alpha\beta-\sum_{i=0}^r\alpha_i\beta_i]
	+ \lambda r \sqrt{\alpha\beta(1-\alpha)(1-\beta)} \\
	& \leq 
	\frac{r+1}{r} \alpha\beta 
	+ \lambda r \sqrt{\alpha\beta(1-\alpha)(1-\beta)} .
	\end{align*}
\end{enumerate}
\end{thm}

\begin{proof}
	For $i\in\{0,\dots,r\}$, let $A_i=A\cap X_{\{i\}}$ and $B_i=B\cap X_{\{i\}}$.
	Define $f_i=1_{A_i}-(r+1)\alpha_i 1_{X_{\{i\}}}$ and
	$g_i=1_{B_i}-(r+1)\beta_i 1_{X_{\{i\}}}$.
	Since $\supp(f_i)\subseteq X_{\{i\}}$,
	we have   $\Trings{1_{X_{\{j\}}},f_i}=0$ for all $j\in\{0,\dots,r\}-\{i\}$,
	whereas
	by Lemma~\ref{LM:partite-basis-props}(i),
	we also have $\Trings{1_{X_{\{i\}}},f_i}=\Trings{1_{X_{\{i\}}},1_{A_i}}-(r+1)\alpha_i \Trings{1_{X_{\{i\}}},
	1_{X_{\{i\}}}} = w(A_i)- (r+1)\alpha_i\frac{1}{r+1}=0$,
	so $f_i\in C^0_\diamond(X,\R)$. Likewise, $g_i\in C^0_\diamond(X,\R)$. 
	We further note that
	\begin{align*}
		\|f_i\|^2&=w(A_i)(1-(r+1)\alpha_i)^2+w(X_{\{i\}}-A_i)((r+1)\alpha_i)^2\\
		&=\alpha_i(1-(r+1)\alpha_i)^2+(\frac{1}{r+1}-\alpha_i)(r+1)^2\alpha_i^2=
		\alpha_i(1-(r+1)\alpha_i).
	\end{align*}
	
	Fix some distinct $i,j\in\{0,\dots,r\}$. 
	By Lemmas~\ref{LM:inner-prod-formula} and Lemma~\ref{LM:partite-basis-props}(ii),
	we have
	$\Trings{\calA 1_{\{X_i\}},1_{\{X_j\}}}=\frac{1}{2}w(E(X_i,X_j))=\frac{1}{2}\cdot\frac{2}{r(r+1)}=
	\frac{1}{r(r+1)}$.
	Now, using Lemma~\ref{LM:inner-prod-formula} again, we find that
	\begin{align}\label{EQ:EML-bipartite:eq-1}
		\frac{1}{2}w(E(A_i ,B_j))
		& =\Trings{\calA 1_{A_i},1_{B_j}}\\
		&=\Trings{\calA((r+1) \alpha_i 1_{X_{\{i\}}})+\calA f_i,
		(r+1) \beta_j 1_{X_{\{j\}}} +g_j} \nonumber \\
		&=
		\Trings{(r+1)\alpha_i\calA 1_{X_{\{i\}}},(r+1)\beta_j 1_{X_{\{j\}}}} +
		\Trings{f_i,g_j} \nonumber \\
		&=
		(r+1)^2\alpha_i\beta_j\cdot\frac{1}{r(r+1)}+\Trings{\calA f_i,g_i}
		=\frac{r+1}{r}\alpha_i\beta_j+\Trings{\calA f_i,g_i}. \nonumber
	\end{align}
	Since $(X,w)$ is an $(r+1)$-partite $[-\lambda,\lambda]$-expander,
	we have 
	\begin{align*}
		\Trings{\calA f_i,g_j}\leq \lambda\|f_i\|\|g_j\|
		=\lambda\sqrt{\alpha_i\beta_i(1-(r+1)\alpha_i)(1-(r+1)\beta_j)}.
	\end{align*}
	Together with \eqref{EQ:EML-bipartite:eq-1}, this implies that
	\begin{align}\label{EQ:EML-bipartite:eq-2}
	\left|\frac{1}{2}w(E(A_i,B_j))-\frac{r+1}{r}\alpha_i\beta_j \right|
	\leq   \lambda \sqrt{\alpha_i\beta_i(1-(r+1)\alpha_i)(1-(r+1)\beta_j)}.
	\end{align}

	It is routine, yet tedious, to check
	that the real two-variable function
	\[h(x,y)=\sqrt{xy(1-(r+1)x)(1-(r+1)y)}\] 
	is concave on 
	$[0,\frac{1}{r+1}]\times [0,\frac{1}{r+1}]$.\footnote{
		Indeed, writing
		$k=r+1$, the determinant of the Hessian matrix of $h$
		evaluates to $\frac{k^3}{4}[\frac{k}{1-kx}+\frac{1}{x}+\frac{k}{1-ky}+\frac{1}{y}-4k]$,
		which  is  positive on the interior of $[0,\frac{1}{r+1}]^2$.
		In addition,   $\frac{{\partial}^2 h}{(\partial x)^2}=-\frac{\sqrt{y(1-ky)}}{[x(1-kx )]^{3/2}}$
		is negative on that domain, so 
		the Hessian matrix is negative semidefinite.
		}
	Thus, by Jensen's inequality, for any sequence of points $\{(x_k,y_k)\}_{k=1}^t$
	in the square 
	$[0,\frac{1}{r+1}]^2$, we have
	$\sum_{k=1}^t h(x_k,y_k)\leq t h( \sum_{k}\frac{x_k}{t},\sum_k \frac{y_k}{t})$.
	We apply this with together with \eqref{EQ:EML-bipartite:eq-2} to prove (i) and (ii).
	
\smallskip	
	
	To prove (i), we consider the points $\{(\alpha_i,\beta_j)\}_{i\in S,j\in T}$.
	By \eqref{EQ:EML-bipartite:eq-2},
	Jensen's inequality and our assumptions on $A$ and $B$, we have
	\begin{align*}
		\left| \frac{1}{2}w(E(A,B)) -\frac{r+1}{r} \alpha\beta \right|  &=
		\left|\sum_{i\in T}\sum_{j\in S} \left[\frac{1}{2}w(E(A_i,B_j))-\frac{r+1}{r} \sum_{i\in T}\sum_{j\in S}
		\alpha_i\beta_j\right]\right|\\
		&\leq
		\sum_{i\in T}\sum_{j\in S}\left|\frac{1}{2}w(E(A_i,B_j))-\frac{r+1}{r} \alpha_i\beta_j\right|\\
		&\leq
		\lambda \sum_{i\in T}\sum_{j\in S} \sqrt{\alpha_i\beta_i(1-(r+1)\alpha_i)(1-(r+1)\beta_j)}\\
		&\leq \lambda |T||S|\sqrt{\frac{\alpha}{|T|}\frac{\beta}{|S|}(1-\frac{(r+1)\alpha}{|T|})
		(1-\frac{(r+1)\beta}{|S|})}\\
		&=\lambda (r+1)\sqrt{\alpha\beta (\frac{|T|}{r+1}-\alpha)(\frac{|S|}{r+1}-\beta)},
	\end{align*}
	which is what we want.
	
\smallskip
	
	To prove (ii), we consider all the points $(\alpha_i,\beta_j)$
	with $i,j\in\{0,\dots,r\}$ and $i\neq j$. By \eqref{EQ:EML-bipartite:eq-2}
	and Jensen's inequality,
	we have
	\begin{align*}
		\Big| \frac{1}{2}w(E(A,B))-\frac{r+1}{r} [\alpha\beta-\sum_{i=0}^r\alpha_i\beta_i] \Big| 
		&=
		\left| \sum_{i\neq j}\left[\frac{1}{2}w(E(A,B))	-\frac{r+1}{r} \alpha_i\beta_i\right]\right| \\
		&\leq \sum_{i\neq j} 
		\lambda \sqrt{\alpha_i\beta_i(1-(r+1)\alpha_i)(1-(r+1)\beta_j)}
		\\
		& \leq \lambda r(r+1)
		\sqrt{\frac{\alpha}{r+1}\frac{\beta}{r+1}(1-\frac{(r+1)\alpha}{r+1})(1-\frac{(r+1)\beta}{r+1})}\\
		&= 
		\lambda r\sqrt{\alpha\beta(1-\alpha)(1-\beta)},
	\end{align*}
	so we are done.
\end{proof}

\section{Sheaves on Simplicial Complexes}
\label{sec:sheaves}

We recall from   \cite{First_2023_sheaves_on_complexes_preprint} 
and \cite{First_2024_cosyst_exp_posets_preprint}
the definition of   sheaves on simplicial complexes as well as their cohomology
and   coboundary expansion, focusing particularly in the case of   sheaves on graphs.
The sheaves that we consider here are called \emph{augmented sheaves}
in  \cite{First_2023_sheaves_on_complexes_preprint}. 

\subsection{Sheaves on Simplicial Complexes}
\label{subsec:sheaves}

Let $X$ be  a simplicial complex. Following
\cite{First_2023_sheaves_on_complexes_preprint},
a  \emph{sheaf}   $\calF$ on $X$ consists of 
\begin{enumerate}[label=(\arabic*)]
	\item an abelian group $\calF(x)$ for every $x\in X$ (including the empty face), and
	\item a group homomorphism $\res^\calF_{y\from x}:\calF(x)\to\calF(y)$ for all $x\subsetneq y\in X$
\end{enumerate}
such that
\begin{equation}\label{EQ:sheaf-cond}
	\res^\calF_{z\from y}\circ\res^\calF_{y\from x} = \res^\calF_{z\from x}
\end{equation}
whenever $x\subsetneq y\subsetneq z\in X$. 
This generalizes the case of graphs considered in the introduction.
We will usually drop the superscript $\calF$ from $\res^\calF_{y\from x}$
when there is no risk of confusion.

\begin{example}\label{EX:constant-sheaf}
	Let $X$ be a simplicial complex and let $R$ be an (additive) abelian group.
	The \emph{constant   sheaf} associated
to $R$ is the   sheaf $\calF$ on $X$ determined by $\calF(x)=R$ for all $x\in X$
and $\res^{\calF}_{y\from x}=\id_R$ for all $x\subsetneq y\in X$.
	We denote this sheaf  by $\aug{R}_X$, or just $\aug{R}$ when $X$ is clear from the context.
\end{example}

There are obvious notions of subsheaves, quotient sheaves, and
homomorphisms between sheaves, see \cite[\S4.1]{First_2023_sheaves_on_complexes_preprint}.
An augmented sheaf (resp.\ sheaf) $\calF$ on a graph $X$ is said to be \emph{constant} if it is isomorphic
to the constant augmented sheaf (resp.\ sheaf) associated to some abelian group $R$.

\subsection{Sheaf Cohomology}
\label{subsec:sheaf-coh}

Let $\calF$ be an augmented sheaf on a simplicial complex $X$.
Fix a linear ordering $L$
on the vertices of $X$  and,
for every $i\in \N\cup \{-1,0\}$,
let $C^i=C^i(X,\calF)$ denote $\prod_{x\in X(i)}\calF(x)$. 
Elements of $C^i $ are called \emph{$i$-chains} with coefficients in $\calF$.
Given
$f\in C^i(X,\calF)$, we write the $x$-component of $f$ as $f(x)$.
Writing $x=\{v_0,\dots,v_i\}$ with $v_0<v_1<\dots<v_i$
relative to $L$, we let  $x_j$   denote $x-\{v_j\}$.
As in \cite[\S4.2, Remark~4.5]{First_2023_sheaves_on_complexes_preprint},
the $i$-th coboundary map $d_i:C^i \to C^{i+1} $ is defined by
\[
(d_i f)(y) = \sum_{j=0}^{i+1}(-1)^j\res_{y\from x_j}f(x_j)
\] 
for all $f\in C^i $, $y\in X(i+1)$.
For example, 
$d_{-1}$ and $d_0$ are given by the formulas from the introduction  \eqref{EQ:differential-sheaves}  if for $e\in X(1)$,
we let  $e^+$ and $e^-$ denote the larger and smaller 
vertices of $e$ relative to $L$, respectively.
We have $d_i\circ d_{i-1}=0$, and
as usual, the $i$-coboundaries, $i$-cocycles and $i$-th cohomology of $(X,\calF)$ are
defined to be
\[
B^i(X,\calF):=\im d_{i-1},\qquad
Z^i(X,\calF) :=\im d_i,\qquad
\HH^i(X,\calF)=\frac{Z^i(X,\calF)}{B^i(X,\calF)},
\]
respectively. We shall abbreviate $B^i(X,\calF)$ to $B^i$
and $Z^i(X,\calF)$ to $Z^i$ when there is no risk of confusion.

For example, if $\calF=\aug{R}_X$ for an abelian group $R$ 
(see Example~\ref{EX:constant-sheaf}), then this recoveres
the usual (reduced) cohomology theory of the simplicial complex $X$ with coefficients in $R$.
In particular,
$B^0(X,\aug{R})$ is the set of constant
functions from $X(0)$ to $R$, $Z^0(X,\aug{R})$ is the set of functions $f:X(0)\to R$
which are constant on each connected component of $X$ and $\HH^0(X,\aug{R})$
is isomorphic to $R^{|\pi_0(X)|-1}$
and
coincides with the reduced singular cohomology group $\tilde{\HH}^0(X,R)$.

\begin{remark}
	The groups $B^0(X,\calF)$ and $Z^0(X,\calF)$
	are independent of the linear ordering $L$ on $V(X)$.
	Indeed, it is straightforward to see
	that $Z^0(X,\calF)$ consists of those $f\in C^0$
	such that $\res_{e\from u} f(u)=\res_{e\from v}f(v)$
	for every edge $e=\{u,v\}\in X(1)$, and
	$B^0(X,\calF)$ consists of the $f\in C^0$
	for which there is $g\in \calF(\emptyset)$
	such that $f(v)=\res_{v\from\emptyset} g$ for all $v\in X(0)$.
	
	For $i\geq 1$,  changing $L$ may change $B^i$ and $Z^i$,
	but not their isomorphism type. See \cite[Proposition~5.11]{First_2024_cosyst_exp_posets_preprint} or \cite[Remark~4.5]{First_2023_sheaves_on_complexes_preprint}.
\end{remark}

\subsection{Coboundary and Cosystolic Expansion}
\label{subsec:coboundary}

Let $(X,w)$ be a weighted $d$-complex (see~\S\ref{subsec:weights})
and let $\calF$ be a    sheaf on $X$.
For $i\in \{-1,0,\dots,d\}$, the \emph{$w$-support norm}
on $ C^i=C^i(X,\calF)$ is the function $\|\cdot\|_w:C^i\to \R$ defined by
\[
\|f\|_w=w(\supp f),
\]
where $\supp f=\{x\in X(i)\suchthat f(x)\neq 0\}$.\footnote{
	Caution: When $\calF$ is the constant sheaf $\R_X$, the norm $\|\cdot\|_w$
	is \emph{not} the norm induced from the inner products we defined
	on $C^0(X,\R)$ and $C^1(X,\R)$ in \S\ref{subsec:expansion}.
	Moreover, $(C^i(X,\R),\|\cdot\|_w)$ is not  a normed $\R$-vector space.}
The corresponding metric on $C^i $ is
\[
\dist_w(f,g):=\|f-g\|_w.
\]
The subscript $w$ will   be dropped from $\|\cdot\|_w$ and $\dist_w$
when there is no risk of confusion.

Let   $i\in\{-1,0,\dots,d-1\}$.
As in \cite[\S6.2, \S6.3]{First_2024_cosyst_exp_posets_preprint}, we
define the \emph{$i$-coboundary expansion} of $(X,w,\calF)$,
denoted 
$\cb_i(X,w,\calF)$
to be the supremum of the set of $\veps\in[0,\infty)$
for which
\[
\|d_i f\|_w\geq \veps \dist_w(f,B^i(X,\calF))\qquad\forall f\in C^i(X,\calF).
\]
We further say that $(X,w,\calF)$ is an \emph{$\veps$-coboundary expander
in dimension $i$} if 
$\cb_i(X,w,\calF)\geq i$.

Similarly, the \emph{$i$-cosystolic expansion} of
$(X,w,\calF)$, denoted $\cse_i(X,w,\calF)$, 
is the supremum of   $\veps\in [0,\infty)$ for which
\[
\|d_i f\|_w\geq \veps \dist_w(f,Z^i(X,\calF))\qquad\forall f\in C^i(X,\calF),
\]
and the \emph{$i$-cocycle distance} of $(X,w,\calF)$ is
\[
\ccd_i(X,w,\calF):=\sup\{\|f\|_w\where 0\neq w\in Z^i(X,\calF)\}.
\]
These definitions are related to and motivated by locally testable codes;
see \cite[\S6]{First_2024_cosyst_exp_posets_preprint} for a detailed discussion
of this. 

Observe that if $\cb_i(X,w,\calF)>0$, then we must have $B^i=Z^i$ (because 
any $f\in Z^i-B^i$ satisfies $\|d_if\|_w=0$ and $\dist_w(f,B^i)>0$)
and therefore $\cb_i(X,w,\calF)=\cse_i(X,w,\calF)$.
It is possible, however, for $\cse_i(X,w,\calF)$ to be positive when
$\cb_i(X,w,\calF)=0$.
We further note that
$\cb_i(X,w,\calF)$, $\cse_i(X,w,\calF)$ and $\ccd_i(X,w,\calF)$ 
do not depend on the implicit linear ordering on $V(X)$
\cite[Proposition~6.6]{First_2024_cosyst_exp_posets_preprint}.

\begin{example}\label{EX:Cheeger-constant}
	Let $(X,w)$ be a weighted graph.
	We saw in the introduction that
	$\cb_0(X,w, \F_2 )=h(X,w)$
	(here $\F_2$ stands for the constant sheaf it induces on $X$).
	If   $X$ is moreover  $k$-regular, $w$ is its canonical weight function 
	(Example~\ref{EX:regular-graph}),
	and we consider the non-weighted
	Cheeger constant \[\tilde{h}(X)=\min_{\emptyset\neq S\subsetneq X(0)}\frac{|E(S,X(0)-S)|}{\min\{|S|,|X(0)-S|\}} ,\]
	then   $\tilde{h}(X)=\frac{k }{2}\cb_0(X,w, \F_2 ) $.
	
	For comparison, it is not difficult to check
	that $\cse_0(X,w,\F_2)=\max_{Y\in\calC} h(Y,w|_Y)$, where $\calC$
	is the set of  connected components of $X$,
	and $\ccd_0(X,w,\F_2)=\max_{Y\in\calC} w(Y(0))$.
\end{example}

The main results of the following sections --- 
Theorems~\ref{TH:expanders-are-cb-exps},
\ref{TH:cbe-for-quotient-sheaves},
\ref{TH:cbe-for-partite-quotient-sheaves} and Corollary~\ref{CR:expansion-of-quotient-sheaves-on-buildings} ---
will only concern with coboundary expansion in dimension $0$.
However, they can be converted
to statements about cosystolic expansion by means of the following remark.

\begin{remark}\label{RM:coboundary-to-cosystolic}
	Let $(X,w)$ be a weighted graph and let $\calF$ be a   
	sheaf on $X$ such that $\cb_0(X,w,\calF)>0$.
	Denote by $\calF_0$   the subsheaf of $\calF$
	determined by $\calF_0(\emptyset)=0$ and
	$\calF_0(x)=\calF(x)$ for $x\neq \emptyset$.
	Then $\cse_0(X,w,\calF)=\cb_0(X,w,\calF)$
	and $\ccd_0(X,w,\calF)=\max\{\|d_{-1}f\|_w\where f\in \calF(\emptyset)\}$.
	Indeed, this follows directly from the definitions
	because the assumption $\cb_0(X,w,\calF)>0$ implies  $B^0=Z^0$.
\end{remark}

\section{Coboundary Expansion of Constant   Sheaves on Graphs}
\label{sec:cbe-constant}

We now show that   a weighted graph $(X,w)$ is  
a good comobinatorial expander, i.e., $h(X,w)$ is large,
if and only if
$(X,w,\calF)$ is a good coboundary expander in dimension $0$ for every  
constant  sheaf $\calF$. 
As explained in the introduction, this is evident in the case $\calF=(\F_2)_X$
(notation as in  Example~\ref{EX:constant-sheaf}).
As a consequence, it will follow that $(X,w,R_X)$ is a good coboundary expander
for every spectral expander $(X,w)$ (see~\S\ref{subsec:expansion}) and abelian group
$R$.

By contrast, we also show that under mild assumptions, $X$ admits
\emph{locally constant} sheaves $\calF$
such that $(X,w,\calF)$ has poor coboundary expansion in dimension $0$,
regardless of how large $h(X,w)$ is.

\begin{lem}\label{LM:strange-inequality}
	Suppose that $\alpha_0,\dots,\alpha_t\in [0,1]$
	satisfy 
	$\alpha_0\geq \max\{\alpha_1,\dots,\alpha_t\}$
	and  $\sum_{i=0}^t\alpha_i\leq 1$.  
	Then $\sum_{i=0}^t\alpha_i(1-\alpha_i)\geq \sum_{i=1}^t\alpha_i$.
\end{lem}

\begin{proof}
	After rearranging, the inequality
	becomes
	$\sum_{i=1}^t\alpha_i^2\leq \alpha_0(1-\alpha_0) $.
	Since $\alpha_0\geq \max\{\alpha_1,\dots,\alpha_t\}$, we have
	$
	\sum_{i=1}^t\alpha_i^2\leq \sum_{i=1}^t\alpha_0\alpha_i\leq \alpha_0(1-\alpha_0)
	$.
\end{proof}

Recall from \S\ref{subsec:expansion} that $h'(X,w)=\min_{\emptyset \neq A\subsetneq X(0)}
\frac{w(E(A,X(0)-A))}{2w(A)w(X(0)-A)}$.

\begin{thm}\label{TH:expanders-are-cb-exps}
	Let $(X,w)$ be a weighted graph and let $R$ be a nontrivial
	(additive) abelian group. Then:
	\[
	\frac{1}{2}h(X,w)\leq h'(X,w)\leq \cb_0(X,w,R_X)\leq h(X,w).
	\]
\end{thm}

\begin{proof}
	The left inequality is just \eqref{EQ:h-h-prime-relation}.
	
	We turn to prove the middle inequality.
	Given $a\in R$, write $f_a$ for the element of $C^0(X,\aug{R})$
	defined by $f_a(x)=a$ for all $x\in X(0)$.	
	We 
	abbreviate 
	$\|\cdot\|_{w}$ to $\|\cdot\|$ and $h'(X,w)$ to $h'$.
	
	Let $f\in C^0(X,\aug{R})$.
	We need to show that $\|d_0 f\| \geq   h' \dist(f,B^k(X,\aug{R})) $,
	or equivalently, that $\|d_0 f\|\geq  h' \|f-f_a\|$ for some $a\in R$.
	Let	 $a_0=0,a_1,\dots,a_t\in R$ be the values
	attained by $f$ together with $0 $,  and write
	$A_i=\{x\in X(0)\suchthat f(x)=a_i\}$ and $\alpha_i=w(A_i)$.
	Then $\sum_{i=0}^t\alpha_i=1$ and $\sum_{i=1}^t\alpha_i=\|f\|$.
	Choose $j\in\{0,\dots,t\}$ such that $\alpha_j=\max\{\alpha_0,\dots,\alpha_t\}$.
	By replacing $f$ with $f-f_{a_j}$, we may assume that $j=0$.
	
	Let $e=\{u,v\}\in X(1)$. If $u\in A_i$ and $v\in A_j$ for distinct
	$i$ and $j$, then $f(x)\in \{a_i-a_j,a_j-a_i\}$,
	and  
	otherwise  $f(x)=0$. 
	This means that
	$
	\|d_0 f\| =\sum_{0\leq i<j\leq t}w(E(A_i,A_j))$.
	
	Writing $A_i^c:=X(0)-A_i$,
	we have
	$\sum_{0\leq i<j\leq t}w(E(A_i,A_j)) =\sum_{i=0}^t \frac{1}{2} w(E(A_i,A_i^c)) $.
	By the definition of $h'=h'(X,w)$,
	\begin{align*}
	 w(E(A_i, A_i^c))  
	& \geq
	h' \cdot 2 \alpha_i(1-\alpha_i)
	\end{align*}
	Thus, $\|d_0f\| \geq h' \sum_{i=0}^t\alpha_i(1-\alpha_i)\geq 
	h' \sum_{i=1}^t\alpha_i=h' \|f\|$,
	where the second inequality is Lemma~\ref{LM:strange-inequality}. 
	
	Finally, we prove the right inequality. Fix a nontrivial element $a\in R$. For every $A\subseteq X(0)$,
	let $f_A\in C^0(X,\aug{R})$ denote the function from $X$ to $R$
	taking the value  $a$ on $A$ and $0_R$ elsewhere.
	Writing $A^c:=X(0)-A$, we have
	$\supp d_0 f_A= E(A,A^c)$
	and $\dist(f_A, B^0(X,\aug{R}))=\min\{w(A),w(A^c))\}$.
	By the definition of $\cb_0(X,w,\aug{R})$, this means that
	$w(E(A,A^c))\geq \cb_0(X,w,\aug{R}) \min\{w(A),w(A^c)\}$.
	As $A\subseteq X(0)$ was arbitrary,
	it follows $h(X,w)\geq \cb_0(X,w,\aug{R})$.
\end{proof}

\begin{cor}\label{CR:expanders-are-cb-exps}
	Let $(X,w)$ be a weighted graph which is a $[-1,\lambda]$-expander
	(see \S\ref{subsec:expansion}),
	and  
	let $R$ be a nontrival abelian group.
	Then
	$\cb_0(X,w,\aug{R})\geq  1-\lambda$. 
\end{cor}

\begin{proof}
	This follows from Theorems~\ref{TH:Cheeger} and~\ref{TH:expanders-are-cb-exps}.
\end{proof}

\begin{cor}\label{CR:different-groups}
	Let $(X,w)$ be a weighted graph and let $R$, $S$
	be nontrivial abelian groups.
	Then
	\[
	\frac{1}{2}\cb_0(X,w,R)\leq \cb_0(X,w,S)\leq 2\cb_0(X,w,R)\]	
\end{cor}

\begin{proof}
	By  Theorem~\ref{TH:expanders-are-cb-exps}, $\frac{1}{2}\cb_0(X,w,R)\leq \frac{1}{2}h(X,w)\leq\cb_0(X,w,S)
	\leq h(X,w)\leq 2\cb_0(X,w,R)$.
\end{proof}

\begin{example}
	Let $X$ be a complete graph on $3$ vertices
	and let $w$ be its canonical weight function.
	It is routine to check that 
	$\cb_0(X,w, \F_2 )=2$  
	while  $\cb_0(X,w,\aug{R})=\frac{3}{2}$
	for every abelian group $R$ admitting at least $3$ elements.
	(The latter can be checked either directly, or using
	Corollary~\ref{CR:expanders-are-cb-exps}
	and the fact that $(X,w)$ is a $[-1,-\frac{1}{2}]$-expander.)
\end{example}

In \cite[\S4.5]{First_2023_sheaves_on_complexes_preprint},
a sheaf $\calF$ on a simplicial complex $X$
is called \emph{locally constant} if
$\calF(\emptyset)=0$ and for every $z\in X-\{\emptyset\}$,
the restriction of $\calF$
to the link $X_z$ is a constant sheaf on $X_z$.
This is equivalent to saying that $\res_{y\from x}:\calF(x)\to \calF(y)$
is an isomorphism for all $\emptyset\neq x\subsetneq y \in X$
and $\calF(\emptyset)=0$. 
For example, if $\calF$ is a   constant sheaf on $X$, then the subsheaf
$\calF_0$ of $\calF$ obtained by setting $\calF_0(\emptyset)=0$
and $\calF_0(x)=\calF(x)$ for all $x\in X-\{\emptyset\}$
is locally constant.
See \cite[Example~4.15]{First_2023_sheaves_on_complexes_preprint}
for more examples.

We now show that unlike to constant sheaves, locally constant sheaves
on good expander graphs may have poor coboundary expansion. 
In particular,  Theorem~\ref{TH:expanders-are-cb-exps} and
		Corollary~\ref{CR:expanders-are-cb-exps}
		fail   for locally constant sheaves $\calF$.

	\begin{example}\label{EX:locally-constant-bad-cbe}
		Let $(X,w)$ be a connected   weighted graph such
		that every vertex in $X$ is contained in at least $2$ edges. 
		Let $e_0$ be an edge of minimal weight in $X(1)$ 
		and
		let $v_0$ be one of the vertices of $e_0$. 
		Let $\F$ be a field with more than $2$ elements 
		and let $\alpha \in\F-\{0,1\}$.
		We define a locally constant sheaf $\calF$ on $X$ as follows:
		put $\calF(x)=\F$ for all $x\in X-\{\emptyset\}$, and for all 
		$\emptyset\neq v\subsetneq e\in X(1)$, let
		\[
		\res^\calF_{e\from v} =\left\{\begin{array}{ll}
		\id_\F & (e,v)\neq (e_0,v_0) \\
		\alpha\id_\F & (e,v)=(e_0,v_0).	
		\end{array}	
		\right.
		\]	
		We claim that $\HH^0(X,\calF)=0$,
		but $\cb_0(X,w,\calF)\leq w(e_0)\leq \frac{1}{|X(1)|}$, regardless of how
		large $\lambda(X,w)$ or $h(X,w)$ are. 
		To see this, note first 
		that $Z^0(X,\calF)=0$. Indeed, $C^0(X,\calF)=\F^{X(0)}$, and
		any $f\in Z^0(X,\calF)$ must satisfy
		$f(u)=\res_{e\from u}f(u)=\res_{e\from v}f(v)=f(v)$ for every
		edge $e=\{u,v\}\in X(1)-\{e_0\}$. Our assumptions on $X$ imply that $X-\{e_0\}$
		is connected, so $f$ is constant. However, writing 
		$e_0=\{u_0,v_0\}$, we also
		have $\alpha f(u_0)= \alpha \res_{e_0\from u_0}f(u_0)=\alpha \res_{e_0\from v_0}f(v_0)=
		f(v_0)=f(u_0)$. Since $\alpha\neq 1$, it follows that $f(u_0)=0$, and  $f=0$. 
		Now that $Z^0=0$, we also have $B^0=0$ and $\HH^0(X,\calF)=Z^0/B^0=0$.
		Moreover, for every $f\in C^0$, we have $\dist_w(f,B^0)=\dist_w(f,0)=\|f\|_w$. 
		Taking  $f=(1_{\F})_{x\in X(0)}$,
		we get $\supp(d_0 f)=\{e_0\}$, so $\|d_0 f\|_w=w(e_0)$
		while $\dist_w(f,B^0(X,\calF))=1$. As a result,  $\cb_0(X,w,\calF)\leq w(e_0)$.
	\end{example}

\section{Coboundary Expansion of Quotients of Constant Sheaves on Graphs}
\label{sec:cbe-quotients}

	In this section, we show that taking the quotient
	of a constant   sheaf on a weighted graph  by a ``small'' subsheaf still 
	results in a good coboundary expander, provided that
	the weighted graph is a good enough spectral expander.
	
	Recall that given an abelian group $R$, a collection of   subgroups $\{R_i\}_{i\in I}$,
	is said to be   \emph{linearly disjoint} (in $R$)
	if the summation map $(r_i)_{i\in I}\mapsto \sum_{i\in I} r_i:\bigoplus_{i\in I}R_i\to R$
	is injective. For example, if $R$
	is a  vector space over a field $\F$ and $R_i=\F v_i$ 
	for some $v_i\in R$, then    $\{R_i\}_{i\in I}$
	are linearly disjoint if and only 
	if the vectors $\{v_i\}_{i\in I}$
	(including repetitions) are linearly independent in $R$.
	
\begin{thm}\label{TH:cbe-for-quotient-sheaves}
	Let $(X,w)$ be a weighted graph with $n$ vertices  and let $R$ be an
	(additive) abelian group. Suppose that
	we are given subgroups $R_x$ of $R$ for every nonempty $x\in X$, and set $R_\emptyset = \{0_R\}$.
	Define a subsheaf $\calG$ of $R_X$ by
	setting $\calG(x)=\sum_{y\subseteq x}R_y$ ($x\in X$),
	and put 
	\[
	t=\max\{\textstyle\frac{w(e)}{w(x)}\where x\in X(0), e\in X(1)_{\supseteq x}\}
	\qquad
	\text{and}
	\qquad
	s=\max\{w(e)\where e\in X(1)\}.
	\]
	Suppose that 
	\begin{enumerate}[label=(\arabic*)]
		\item  for every subgraph $Y$ of $X$ which is either a cycle
		of length $\leq \ceil{\frac{2}{3}n}$ or a path of a length $\leq 2$
		(in the sense of \S\ref{subsec:complexes}),
		the subgroups $\{R_y\}_{y\in Y}$ are linearly disjoint in $R$, and
		\item for every distinct $u,v\in X(0)$, the subgroups $R_u$ and $R_v$
		are linearly disjoint.
	\end{enumerate}
	If $X$ is a $[\mu,\lambda]$-spectral expander ($\mu,\lambda\in \R$), 
	then  
	\[
	\cb_0(X,w,\aug{R}_X/\calG)\geq \frac{2-4\lambda-4\max\{|\lambda|,|\mu|\}-5t-2s}{5-2\lambda}.
	\]	
\end{thm}

\begin{thm}\label{TH:cbe-for-partite-quotient-sheaves}
	Let $(X,w)$ be an $(r+1)$-partite weighted simplicial  complex.
	Let $R$, $\{R_x\}_{x\in X}$, $\calG$, $s$, $t$  be as in
	Theorem~\ref{TH:cbe-for-quotient-sheaves} and assume that
	conditions (1) and (2) of that theorem are fulfilled.
	If $(X,w)$ is an $(r+1)$-partite $[\mu,\lambda]$-expander
	(see \S\ref{subsec:partite-complexes}) with $\lambda\geq -\frac{1}{r}$,
	then  
	\[
	\cb_0(X ,w,\aug{R}_X/\calG)\geq \frac{2r-4r\lambda-4r^2\max\{|\lambda|,|\mu|\}-(5r+2)t-2r s}{ 5r+2 -2r\lambda}.
	\]	
\end{thm}

In both theorems, if all edges in $X$ have the same weight, or if
all the subgroups $\{R_v\}_{v\in X(0)}$ are linearly disjoint,
then we can eliminate the term  $-2s$, resp. $-2rs$, in
the numerator; see Remark~\ref{RM:eliminating-s} below.
We also note that elementary analysis shows that artificially increasing
$\lambda$   will only decrease the resulting coboundary expansion.
For a statement regarding the cosystolic expansion of the sheaf $R_X/\calG$
(more precisely, its subsheaf $(R_X/\calG)_0$), 
see Corollary~\ref{CR:cse-quotient-sheaves} below.

\begin{example}
	Suppose that $X$ is a connected $k$-regular graph on $n$ vertices and $w$
	is its canonical weight function (Example~\ref{EX:regular-graph}). 
	Let $R$, $\{R_x\}_{x\in X}$,
	$\calG$,
	$t$ and $s$ be as in   Theorem~\ref{TH:cbe-for-quotient-sheaves}.
	Then
	$t=\frac{2/kn}{1/n}=\frac{2}{k}$ and $s=\frac{2}{kn}$; in fact, we can ignore $s$
	because all edges have the same weight.
	Let $\rho\in [0,1]$ be a number with $|\lambda|\leq \rho$ for any eigenvalue $\lambda\neq \pm 1$  
	of $\calA_{X,w}$. Recall that $X$ is called a \emph{Ramanujan graph} if
	we can take $\rho= \frac{2\sqrt{k-1}}{k}$; see \cite{Davidoff_2003_Ramanujan_graphs}
	(for instance)
	for details and motivation for this definition.

	If $X$ is not bipartite, then $(X,w)$ is a $[-\rho,\rho]$-expander (see~\S\ref{subsec:expansion}), 
	and Theorem~\ref{TH:cbe-for-quotient-sheaves} implies that
	$\cb_0(X,w,\aug{R}/\calG)\geq \frac{2-8\rho-10/k}{5}=\frac{2}{5}-\frac{8}{5}\rho-\frac{2}{k}$.
	If $X$ is a Ramanujan graph, for instance, then $\cb_0(X,w,\aug{R}/\calG)\geq
	\frac{2}{5}- \frac{16\sqrt{ k-1 }}{5k}-\frac{2}{k}=\frac{2}{5}-O(\frac{1}{\sqrt{k}})$
	
	If $X$ is   bipartite, then $(X,w)$ is a $2$-partite $[-\rho,\rho]$-expander
	and Theorem~\ref{TH:cbe-for-partite-quotient-sheaves}	 says that
	$\cb_0(X,w,\aug{R}/\calG)\geq \frac{2-8\rho-14/k}{7}=\frac{2}{7}-\frac{8}{7}\rho -\frac{2}{k}$.
	Again, taking $X$ to be a bipartite Ramanujan graph gives
	$\cb_0(X,w,\aug{R}/\calG) \geq 	\frac{2}{7}- \frac{16\sqrt{ k-1 }}{7k} -\frac{2}{k}
	=\frac{2}{7}-O(\frac{1}{\sqrt{k}})$.
\end{example}

The rest of this section is dedicated to proving Theorems~\ref{TH:cbe-for-quotient-sheaves}
and~\ref{TH:cbe-for-partite-quotient-sheaves}.
We prove both theorems together in  a series of lemmas.

\begin{lem}\label{LM:small-set-expansion}
	Let $(X,w)$ be a weighted   graph and let $\calF$
	be a  sheaf on $X$.
	Let $\veps\in \R_+$, let $\lambda\in \R$
	and 
	suppose that $(X,w)$ is a $[-1,\lambda]$-expander,
	and   for every $v\in X(0)$ and $h\in \calF(v)-\{0\}$,
	we have 
	\[w(\{e\in X(1)_{\supseteq v}\suchthat \res_{e\from v}h(v)\neq 0\})\geq \veps w(v) .\]
	Let $f\in C^0(X,\calF)$ and $\alpha \in [0,\infty)$. If
	$\|f\|_w\leq \alpha$, then
	\[(\veps - 2\lambda -(2-2\lambda)\alpha)\|f\|_w\leq \|d_0 f\|_w . \]
\end{lem}

\begin{proof}
	Write $A=\supp f$ and $A^c=X(0)-A$. Since decreasing $\alpha$ increases
	the left hand side of the desired inequality, it is enough
	to prove the lemma for $\alpha=\|f\|_w$.  

	Fix some $v\in X(0)$ with $f(v)\neq 0$. 
	We claim that for every $e\in X(1)_{\supseteq v}$ with $\res_{e\from v} f(v)\neq 0$,
	at least one of the following hold:
	\begin{enumerate}[label=(\roman*)]
		\item $e\in E(A)$,
		\item $e\in \supp(d_0 f)\cap E(A,A^c)$.
	\end{enumerate}
	Indeed, we have $(d_0f)(e) = \pm  \res_{e\from v} f(v)\mp \res_{e\from e-v}f(e-v)  $. 
	If $ f(e-v)\neq 0$, then $v$ and $e-v$ are in $\supp f$,
	and $e\in E(A)$. Otherwise,  $ f(e-v)= 0$,
	so $(d_0f)(e)=\pm \res_{e\from v} f(v)\neq 0$ and $e\in\supp(d_0f)\cap E(A,A^c)$.
	This means that
	\begin{align*}
	w(\{e\in X(1)_{\supseteq v}\suchthat \res_{e\from v}f(v)\neq 0\})
	\leq
	w(E(A)_{\supseteq v})+w([\supp(d_0f)\cap E(A,A^c)]_{\supseteq v}).
	\end{align*}
	
	By assumption, the left hand side of the last inequality
	is at least $\veps w(v)$. 
	Summing  over all $v\in \supp f$, we get
	\begin{align*}
		\veps \|f\|_w & \leq \sum_{v\in A}
		w(E(A)_{\supseteq v}) +
		\sum_{v\in A} w([\supp(d_0f)\cap E(A,A^c)]_{\supseteq v})\\
		& \leq 2w(E(A))+w(\supp d_0 f)
		\leq 2(\alpha^2+\lambda\alpha(1-\alpha))+\|d_0 f\|_w .
	\end{align*}
	Here, the second inequality holds because every edge in $E(A)$ is counted exactly twice 
	and every edge in $\supp (d_0 f)$ is counted at most once, whereas the third inequality
	follows from Theorem~\ref{TH:EML}(ii). By rearranging, we find that
	\[
	\|d_0 f\|_w \geq \veps \|f\|_w-2 \alpha^2 - 2 \lambda\alpha(1-\alpha) =
	(\veps - 2\lambda -2\alpha+2\lambda \alpha)\|f\|_w. \qedhere
	\]
\end{proof}

\begin{lem}\label{LM:expansion-at-links}
	Let $(X,w),R,\{R_x\}_{x\in X},\calG, t$
	be as in  
	Theorem~\ref{TH:cbe-for-quotient-sheaves} or Theorem~\ref{TH:cbe-for-partite-quotient-sheaves},
	and let $\calF=\aug{R}/\calG$.
	If $\{R_y\}_{y\in Y}$ are linearly disjoint
	in $R$ for every path $Y\subseteq X$ of length $2$, 
	then
	for every $v\in X(0)$ and $h\in \calF(v)-\{0\}$,
	we have 
	\[w(\{e\in X(1)_{\supseteq v}\suchthat \res_{e\from v}h \neq 0\})\geq (2-t) w(v) .\]
\end{lem}

\begin{proof}
	Let $v \in X(0)$ and $h\in \calF(v)-\{0\}$. 
	If $\res_{e\from v} h\neq 0$ for all $e\in X(1)_{\supseteq v}$,
	then $w(\{e\in X(1)_{\supseteq v}\suchthat \res_{e\from v}h \neq 0\})= w(X(1)_{\supseteq v})=2w(v)$
	because $w$ is a weight function (see~\S\ref{subsec:weights}), and the lemma holds. 
	
	Suppose now that there exists $y \in X(1)_{\supseteq v}$ such that $\res_{y\from v} h=0$.
	We claim that  $\res_{z\from v} h\neq 0$ for all $z\in X(1)_{\supseteq v}$ different from $y$.
	Fix such $z$ and let $v'=y-v$ and $v''=z-v$. Then $Y=\{\emptyset,v,v',v'',y,z\}$
	is a path of length $2$ in $X$, meaning  that $R_\emptyset=0,R_v,R_{v'},R_{v''},R_y,R_z$
	are linearly disjoint in $R$.
	Choose $g\in R$ with $h(v)=g+R_v $ (note that $\calG(v)=R_v+R_\emptyset=R_v$).
	Since $\res_{y\from v} h=0$, we have $g\in R_y+R_v+R_{v'} $.
	Likewise, if $\res_{z\from v} h=0$,
	then $g\in R_z+R_v+R_{v''} $.
	Consequently,
	$g\in (	R_y+R_v+R_{v'})\cap (R_z+R_v+R_{v''})=R_v  =\calG(v)$,
	but this contradicts our assumption that $h=g+\calG(v)\neq 0$ in $\calF(v)$.
	Thus, we must have $\res_{z\from v}h\neq 0$.
	As this holds for all $z\neq y$, we conclude that
	$\{e\in X(1)_{\supseteq v}\suchthat \res_{e\from v}h \neq 0\}=X(1)_{\supseteq v}-\{y\}$.
	Since $w(X(1)_{\supseteq v})-w(y)=2w(v)- w(y)\geq 2w(v)- t w(v)=
	(2- t)w(v)$, we are done.
\end{proof}

\begin{lem}\label{LM:weight-of-subforest}
	Let $(X,w)$ be a weighted graph with $n$ vertices. 
	Define $t$ and $s$ as in Theorem~\ref{TH:cbe-for-quotient-sheaves}
	and let $T$ be a subgraph of $X$.
	\begin{enumerate}[label=(\roman*)]
		\item If $w(T(1))\geq t$, then $T$ contains a cycle.
		\item If $w(T(1))\geq t+s$, then $T$ contains a cycle   of length
		$\leq \ceil{\frac{2}{3} n}$.
	\end{enumerate}	 
\end{lem}

\begin{proof}
	(i) It is enough to show that if $T$ contains no cycles, then $w(T(1))< t$.
	In this case, $T$  is a forest, i.e.,
	every connected component of $T$ is a tree. 
	Choose roots for the trees in $T$ and denote the set of roots by $R\subseteq V(X)$.
	For $v\in V(X)-R$, we let $p(v)\in V(X)$ denote the parent of $v$ in $T$.
	Then $w(T(1))=\sum_{v\in V(X)-R} w(\{v,p(v)\})\leq
	\sum_{v\in V(X)-R} t w(\{v\})<\sum_{x\in X(0)} t w(x)=t$. 

	(ii) By (i), $T$ contains a cycle $C_1$. Choose an edge $e_1$ in $C_1$.
	Then $w(e_1)\leq  s$. Let $T'=T-\{e_1\}$, i.e.,   the graph obtained
	from $T$ by removing the edge $e_1$. Then $w(T'(1))= w(T(1))-w(e)\geq t+s-s=t$,
	so by (i), $T'$ contains another cycle, $C_2$,
	and $e_1\notin C_2$.
	If $C_1(0)\cap C_2(0)=\emptyset$, then one of $C_1$, $C_2$
	has less than $\frac{1}{2}n$ vertices, and is therefore the required
	cycle.
	
	Suppose now that $C_1(0)\cap C_2(0)\neq \emptyset$.	
	By Lemma~\ref{LM:cycle-minus-subgraph},   $C_1-C_2$ is a nonempty disjoint union of open paths.
	Let $Z$ be one these paths, and let $x,y$ denote its end points.
	If $x=y$, then $C_1$ and $C_2$ share exactly one vertex.
	Thus,   $|C_1(0)|+|C_2(0)|\leq n+1$, and again,
	one of   $C_1$, $C_2$   is the required
	cycle.
	Assume $x\neq y$. Then $x,y\in  C_2(0)$.
	Let $P$ and $Q$ denote the two paths
	from $x$ to $y$ contained in  $C_2$ and put
	$p=|P(0)|$, $q=|Q(0)|$, $z=|Z(0)|$.
	Since $P(0)-\{x,y\}$, $Q(0)$ and $Z(0)$
	are pairwise disjoint, we have
	$p+q+z\leq n+2$.
	Let
	$R_1:=P\cup Q$, $R_2:=P\cup Z$ and $R_3:=Q\cup Z$.
	Then $R_1,R_2,R_3$ are cycles, and   
	$|R_1(0)|+|R_2(0)|+|R_3(0)|=(p+q-2)+(p+z)+(q+z)=2(p+q+z)-2\leq
	2(n+1)$.
	This means that there is $i\in\{1,2,3\}$
	such that $|R_i(0)|\leq \floor{\frac{2}{3}(n+1)}=\ceil{\frac{2}{3}n}$, so we are done.
\end{proof}

\begin{remark}
	The proof of Lemma~\ref{LM:weight-of-subforest}(ii) also shows that the girth of a graph 
	with
	$n$ vertices and $n+1$ edges is at most $\ceil{\frac{2}{3}n}$. This bound is tight, e.g., consider
	a graph $X$ obtained by gluing $3$ closed paths of length $k$ or $k-1$ at their endpoints. 
\end{remark}

\begin{lem}\label{LM:diameter-bound}
	Let $X$ be a connected graph which is a union of its cycle subgraphs.
	Then every two vertices in $X$ can be connected by a path
	of length $\leq \frac{2}{3}(|X(0)|-1)$.
\end{lem}

\begin{proof}
	Write $n=|X(0)|$, and
	let $x,y\in X(0)$ be distinct $0$-faces. 
	Since $X$ is connected, there is path from $x$ to $y$
	and our assumption on $X$
	implies that this path is contained in a union of
	cycles.
	This means that  there
	are cycles $R_1,\dots,R_t$  such that $x\in R_1(0)$, $y\in R_t(0)$
	and there exists $x_i\in R_i(0)\cap R_{i+1}(0) $ for all 
	$i\in \{1,\dots,t-1\}$. Set $x_0=x$ and $x_t=y$.
	We prove the lemma by induction on $t$.	
	
	If $t=1$, then $x,y\in R_1$, so there is a path from $x$ to $y$
	of length at most $\floor{\frac{1}{2} |R_1(0)|} \leq \floor{\frac{1}{2} n} \leq \frac{2}{3}(n-1)$
	(because $R_1$, and hence $X$, has at least $3$ vertices). Suppose henceforth that
	$t>1$.
	
	If $R_i \cap R_j \neq \emptyset$ for some $i,j\in \{1,\dots,t\}$ with $i+2\leq j$,
	then we can remove $R_{i+1},\dots,R_{j-1}$ from
	$R_1,\dots,R_t$ and finish by the induction hypothesis.
	We  may therefore assume that   $R_i\cap R_j=\emptyset$ whenever $i+2\leq j$.

	Next, if $|R_{i}(0)\cap R_{i+1}(0)|=1$ for all
	$i\in\{1,\dots,t-1 \}$, 
	Then $R_{i }(0)\cap R_{i+1}(0)=\{x_i\}$ for all  $i\in\{1,\dots,t-1\}$.
	Since $R_i\cap R_j=\emptyset$ when $i+2\leq j$,
	this means that the sets $R_1(0)-\{x_1\},\dots,R_t(0)-\{x_t\}$ are pairwise disjoint.
	Thus, writing $r_i=|R_i(0)|-1$, we have $\sum_{i=1}^tr_i\leq |X(0)-\{x_t\}|=n-1$.
	Let $i\in \{1,\dots,t\}$.
	Since $R_i$ is a cycle with $r_i+1$ vertices,  there is a closed path $P_i$  
	from $x_{i-1}$ to $x_i$ of length at most $\floor{\frac{1}{2} (r_i+1)} \leq \frac{2}{3}r_i$
	(because $r_i\geq 2$). The union $P_1\cup \dots\cup P_t$
	is a path from $x_0=x$ to $x_t=y$
	of length at most $\sum_{i=1}^t\frac{2}{3}r_i\leq  \frac{2}{3}(n-1)$.
	
	Finally, if  there is $i\in\{1,\dots,t-1\}$
	such that $R_{i}$ and $R_{i+1}$ share at least $2$ vertices,
	then $R_i-R_{i+1}$ is a disjoint union of open paths and each of these paths has
	distinct end points (cf.\ Lemma~\ref{LM:cycle-minus-subgraph}). 
	Of these open paths, let $P$ be denote the one containing $x_{i-1}$ and let $z,w$
	be its endpoints. Then $z,w\in R_{i+1}(0)$.
	Let $Q$ denote a closed path from $z$ to $w$ in $R_{i+1}$ 
	which also includes $x_{i+1}$. Then $R':=P\cup Q$ is a cycle
	containing both $x_{i-1}$ and $x_{i+1}$. We replace $R_i,R_{i+1}$ with $R'$
	and proceed by induction on $t$.
\end{proof}

\begin{lem}\label{LM:cycle-graph-cycles}
	Let $X$ be a cycle graph, let $R$, $\{R_x\}_{x\in X}$ and $\calG$
	be as in Theorem~\ref{TH:cbe-for-quotient-sheaves}
	and suppose that 
	all the $\{R_x\}_{x\in X}$ are linearly disjoint in $R$.	
	Then 
	for every $f\in Z^0(X,R_X/\calG)$, there
	exists 
	$h\in R$ such that $f(v)=h+R_v$ for all $v\in X(0)$.
\end{lem}

\begin{proof}
	Let $v_0,\dots,v_{\ell-1}$ be the vertices
	of $X$ and let $e_0,\dots,e_{\ell-1}$ be the edges of $X$.
	We  choose the numbering such that $e_i=\{v_i,v_{(i+1)\bmod \ell}\}$
	for all $i$ and
	write $v_\ell=v_0$  for convenience.

	Let $f\in Z^0(X,R_X/\calG)$, and choose $g\in C^0(X,R)$
	projecting onto $f$, i.e., $f(v)=g(v)+R_v$ for all $v\in X(0)$.
	For every $i\in\{0,\dots,\ell-1\}$,
	write $g_i=g(v_i)$, and let $g'_i=g_i-g_0$.
	Unfolding the definitions,
	the assumption $d_0 f=0$ is equivalent to having
	$g'_{i+1}-g'_i=g_{i+1}-g_i\in R_{v_i}+R_{v_{i+1}}+R_{e_i}$ for all $i\in\{0,\dots,\ell-1\}$.

	We claim that there exist 
	$c_0\in R_{v_0}$, $c_1,\tilde{c}_1\in R_{v_1}$, \dots,
	$c_{\ell -2},\tilde{c}_{\ell -2}\in R_{v_{\ell -2}}$,
	$\tilde{c}_{\ell-1}\in R_{v_{\ell-1}}$
	and $u_i\in R_{e_i}$ ($i\in\{0,\dots,\ell-2\}$)
	such that $g'_i=c_0+u_0+c_1+u_1+\dots+c_{i-1}+u_{i-1}+\tilde{c}_i$
	for all $i\in\{1,\dots,\ell-1\}$.
	The proof is by induction on $i$. For the case $i=1$, note
	that $g'_1 = g_1-g_0\in R_{v_0}+R_{v_1}+R_{e_0}$, so we can choose
	$c_0\in R_{v_0}$, $\tilde{c}_1\in R_{v_1}$ and $u_0\in R_{e_0}$
	such that $g'_1=c_0+u_0+\tilde{c}_1$.
	Suppose now that $i\in\{1,\dots,\ell-2\}$
	and $c_0,\tilde{c}_1,c_1,\dots,\tilde{c}_{i-1},c_{i-1},\tilde{c}_i$
	and $u_0,\dots,u_{i-1}$
	were chosen so that 
	$g'_{j}=c_0+u_0+\dots+c_{j-1}+u_{j-1}+\tilde{c}_j$
	for all $j\in\{1,\dots,i\}$. 
	Then
	\[
	g'_{i+1}-c_0-u_0-\dots-c_{i-1}-u_{i-1}-\tilde{c}_i = g'_{i+1}-g'_i \in R_{v_i}+R_{v_{i+1}}+R_{e_i}.
	\]
	Thus, there are
	$s\in R_{v_i}$, $\tilde{c}_{i+1}\in R_{v_{i+1}}$
	and $u_i\in R_{e_i}$
	such that $g'_{i+1}-\sum_{j=0}^{i-1}(c_j+u_j)-\tilde{c}_i=s+\tilde{c}_{i+1}+u_i$,
	or rather
	$g'_{i+1}=\sum_{j=0}^{i-1}(c_j+u_j)+(\tilde{c}_i+s)+u_i+\tilde{c}_{i+1}$.
	Setting $c_i=\tilde{c}_i+s$ then proves our claim.

	Now, we have $c_0+u_0+\dots+c_{\ell-2}+u_{\ell-2}+\tilde{c}_{\ell-1}=g'_{\ell-1}=g_{\ell-1}-g_0\in R_{v_{\ell-1}}+R_{v_0}+R_{e_{\ell-1}}$. Since   $\{R_x\}_{x\in X(0)\cup X(1)}$
	are linearly disjoint, we must have $c_1=\dots=c_{\ell-2}=0$
	and $u_0=\dots=u_{\ell-2}=0$.
	This means that $g'_i=c_0+\tilde{c}_i$ for all $i\in\{1,\dots,\ell-1\}$.
	Define $h = g_0+c_0$. Then for all $i\in \{1,\dots,\ell-1\}$,
	we have $g_i = g'_i +g_0=h+\tilde{c}_i\in h+R_{v_i}$,
	and $g_0 = h - c_0\in h+R_{v_0}$.
	This means that $f( v_i )=h+R_{v_i}$
	for all $i$, so we proved the existence of $h$.
\end{proof}

We are now ready to prove the following key lemma.

\begin{lem}\label{LM:key-lemma}
	Under  the assumptions of 
	Theorem~\ref{TH:cbe-for-quotient-sheaves},
	put $\calF=R_X/\calG$ 
	and let $f\in C^0(X,\calF)$
	and $\beta\in [0,1]$.
	If $\|d_0 f\|_w\leq \beta$, then 
	\[\dist_w(f,B^0(X,\calF)) <
	\frac{2}{3}+
	\frac{1}{3}[\beta+t+s+\lambda+2\max\{|\lambda|,|\mu|\}].
	\]
	If, instead, we use the  assumptions of 
	Theorem~\ref{TH:cbe-for-partite-quotient-sheaves}, then
	\[
	\dist_w(f,B^0(X,\calF)) <
	\frac{2r+2}{3r+2}+\frac{r}{3r+2}[\beta+t+s+\lambda+2r \max\{|\lambda|,|\mu|\}] .
	\]
\end{lem}

\begin{proof}
	\Step{1} 
	Let $n=|X(0)|$.	
	We call a subgraph $Y$ of $X$ an \emph{$f$-blob}, or just a blob for short,
	if:
	\begin{enumerate}[label=(b\arabic*)]
		\item $Y$ is connected and equals to the union of its cycle subgraphs, and
		\item there exists $g\in R$   
		such that $f(v)=g +R_v$ for all $v\in Y(0)$.
	\end{enumerate} 
	Note that   condition (b1)
	implies that   $|Y(0)|>2$, because a cycle
	has at least $3$ vertices.
	We denote an element $g\in G$ as in (b2)
	by $g_Y$; we will  see below that $g_Y$ is uniquely determined by $Y$.
	
	Denote the set of $f$-blobs by $\calB$.	By Lemma~\ref{LM:cycle-graph-cycles}
	and assumption (1) of Theorem~\ref{TH:cbe-for-quotient-sheaves},
	every cycle $Y\subseteq X-\supp(d_0f)$ of length $ \ceil{\frac{2}{3}n}$ or less
	is a blob, because the restriction of $f$ to $Y$ is in $Z^0(Y,\calF|_Y)$.

	Note that if there exists  a blob $Y$ with $w(Y(0))\geq \alpha$,
	then $\dist_w(f,B^0(X,\calF))\leq 1-\alpha$.
	Indeed, writing $g=g_Y$ and $f'=f-d_{-1}g$,
	we have $\dist_w(f,B^0(X,\calF)) \leq \|f'\|_w\leq 1-\alpha$,
	because $f'$ vanishes on $Y(0)$.
	We will prove the lemma by showing that there exists a sufficiently large blob.

\medskip	
	
	\Step{2}
	Observe   that if $Y$ and $Z$ are blobs such that $|Y(0)\cap Z(0)|>1$,
	then $Y\cup Z$ is also a blob.
	Indeed, (b1) holds for $Y\cup Z$ because it holds for $Y$
	and $Z$, and  $Y\cap Z\neq\emptyset$.
	To see that (b2) holds, fix  a choice of $g_Y$ and $g_Z$.
	Then
	for every $v\in Y(0)\cap Z(0)$, we have $g_Y+R_v=f(v)=g_{Z}+R_v$,
	so $g_Y-g_Z\in R_v$. Choosing distinct $u,v\in Y(0)\cap Z(0)$,
	we get $g_Y-g_Z\in R_u\cap R_v=0$, because $R_u$ and $R_v$
	are linearly disjoint (assumption (2) of Theorem~\ref{TH:cbe-for-quotient-sheaves}). 
	This means that $g_Y=g_Z$ and (b2)
	holds for $Y\cup Z$ by taking $g:=g_Y=g_Z$.
	
	Applying the previous paragraph with $Y=Z$ shows that 
	$g_Y$ is uniquely determined by $Y$.

\medskip
	
	\Step{3}
	Write $M= 	\bigcup_{Y\in \calB} Y$. Then $M$ is a subgraph of $X$.
	
	We claim that $M(1)\subseteq X(1)-\supp (d_0 f)$. 
	To show this, it is enough to prove that $Y(1)\subseteq X(1)-\supp (d_0 f)$
	for any blob $Y$. Let $e$ be an edge in $Y$
	and let $u$ and $v$ be its vertices. Then $f(u)=g_Y+R_u$
	and $f(v)=g_Y+R_v$. As a result, $(d_0 f)(e) = g_Y-g_Y+(R_u+R_v+R_e)=0$
	in $\calF(e)$, meaning that $e\notin \supp(d_0 f)$,
	hence our claim.
	
	We observed in Step~1 that every
	cycle of length $\leq \ceil{\frac{2}{3} n}$ in $X-\supp(d_0 f)$
	is a blob, and thus contained in $M$.
	It follows that the graph underlying $X(1)-\supp(d_0 f)-M(1)$ contains no
	cycles of length $\leq \ceil{\frac{2}{3} n}$.
	Thus, by Lemma~\ref{LM:weight-of-subforest}, 
	$w( X(1)-\supp(d_0f)-M(1))< t+s$.
	Since $\|d_0 f\|_w\leq \beta$,
	it follows that
	\[
	w(M(1))> 1-\beta-t-s.
	\]

\medskip

	\Step{4} 
	A blob is called maximal if it is not properly contained in any
	other blob. Write $\calM$ for the set of maximal blobs.
	Since every blob is contained in a maximal blob,
	$M=\bigcup_{Y\in \calM} Y$.   Step~2 tells us that for every   $Y \in\calM$ 
	and $Z\in\calB$,
	either $Z\subseteq Y$, or  $|Y(0)\cap Z(0)|\leq 1$.

	Let $N$ denote the set of $0$-faces of $X$
	belonging to more than one blob in $\calM$.
	We define a graph $\Gamma$ as follows:
	The vertices of $\Gamma$ are $\calM\cup N$ and
	the edges of $\Gamma$ are pairs $\{x,Y\}$
	such that $x\in N$, $Y\in\calM$ and $x\in Y(0)$.
	See Figure~\ref{FG:blobs} for an illustration.

\begin{figure}[htb]
\begin{center}
\includegraphics[width=\textwidth]{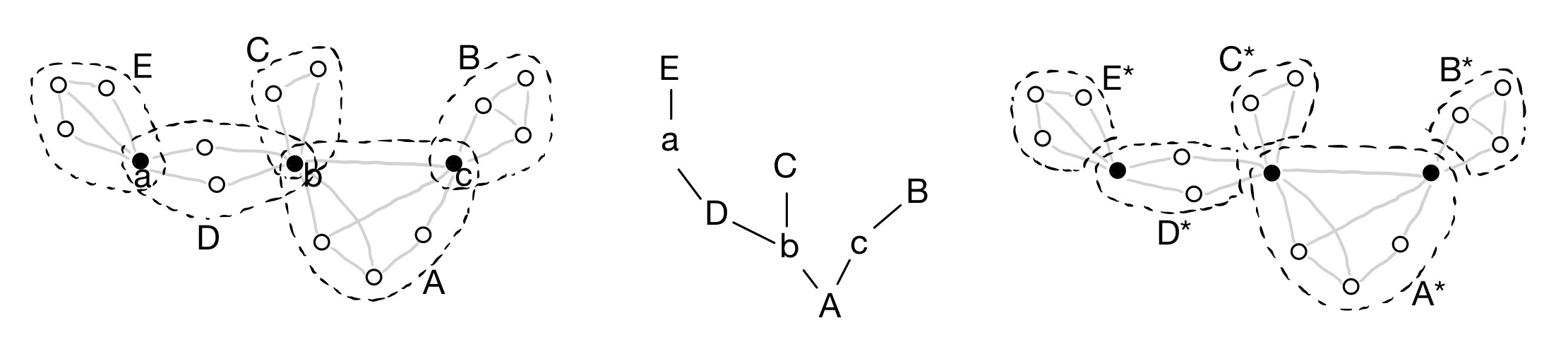}
\end{center}
\captionsetup{margin=1cm}
\caption{An illustration of the collection of blobs $\calM$ (left), the associated
graph $\Gamma$ (middle), and the partition $\{Y^*\where Y\in\calM\}$ (right).
The blobs are labelled $A$--$E$. The black vertices are
those in living in $N$.
The  set of roots taken on the right is $\calR=\{A\}$.}
\label{FG:blobs}
\end{figure}
		
	We claim that the graph $\Gamma$ has no cycles.
	For the sake of contradiction, 
	suppose otherwise. Then there exist $\ell\geq 2$ 
	and distinct
	$x_0,\dots,x_{\ell-1}\in N$, $Y_0,\dots,Y_{\ell-1}\in\calM$
	such that $\{x_i\}= Y_i\cap Y_{i+1}$ for all $i\in\{0,\dots,\ell-1\}$ (with the convention
	$Y_\ell =Y_0$) and $Y_i\cap Y_j=\emptyset$ whenever $|i-j|>1$. 
	By applying Lemma~\ref{LM:diameter-bound} to $Y_i$,
	we see that there is a closed path $P_i\subseteq Y_i$ 
	of length $\leq \frac{2}{3}(|Y_i(0)|-1)$ from $x_{i-1}$ to $x_i$.
	Since the sets $Y_0(0)-\{x_0\},\dots,Y_{\ell-1}(0)-\{x_{\ell-1}\}$
	are pairwise disjoint,
	the union $P =\bigcup_{i=0}^{\ell-1}Y_i$ is   a cycle
	of length
	at most $ \frac{2}{3}\sum_{i=0}^{\ell-1}(|Y_i(0)|-1)\leq \frac{2}{3} n$
	that is contained in $M\subseteq X-\supp(d_0f)$.
	This means that $P$ is a blob --- see Step~1.
	By construction, $P$ and each of the $Y_i$ share  at least two $0$-faces, namely, $x_{i}$
	and $x_{i+1}$, so we must have $P\subseteq Y_i$.
	But then $|Y_0(0)\cap Y_1(0)|\geq |P(0)|\geq 2$,
	which contradicts our assumption $Y_0\cap Y_1=\{x_0\}$.

\medskip	
	
	\Step{5}
	By the previous step, $\Gamma$ is a forest. 
	Let $\calR$ be a set of roots for the trees in $\Gamma$.
	By the definition of $N$, every connected component of $\Gamma$ contains a vertex in  $\calM$,
	so we may choose $\calR$ to be contained in $\calM$.
	We denote the parent of every every $Y\in \calM-\calR$ by $x(Y)$.
	Since $x(Y)$ is not a root, it has a parent, which we denote by $g(Y)$
	(the grandparent of $Y$).
	For every  blob $Y\in \calM $, define:
	\[
	Y^* = \left\{
	\begin{array}{ll}
	Y(0) & Y\in \calR \\
	Y(0) - \{x(Y)\} & Y\in\calM-\calR.
	\end{array}
	\right.
	\]

	We claim that $\{Y^*\where Y\in \calM\}$ is a partition of $M(0)$ (see Figure~\ref{FG:blobs}
	for an illustration).
	To that end, let us first show that	
	$M(0)\subseteq  \bigcup_{Y\in\calM} Y^*$. Observe that
	every   $x\in M(0)$ belongs to $Y(0)$ for some $Y\in \calM$.
	If $Y\in \calR$, then $Y^*=Y(0)$ and $x\in Y^*$.
	Otherwise, $Y^*=Y(0)-\{x(Y)\}$, so 
	$x\in Y^*$ provided $x\neq x(Y)$. If
	$x=x(Y)$, then $x$ is a vertex of $g(Y)$ by the definition of $\Gamma$. 
	Since $g(Y)$ is the parent of $x$ (in $\Gamma$),
	$x$ is not the parent of $g(Y)$, which means
	that $x\in g(Y)^*$.
	Thus, in any case,  $x\in  \bigcup_{Y\in\calM} Y^*$.
	Next, we need to show that   $Y^*\cap Z^*= \emptyset$ for any distinct
	$Y,Z\in \calM$. For the sake of contradiction, suppose there is $x\in Y^*\cap Z^*$. Then
	$x\in N$ and $x$ is adjacent to $Y$ and $Z$ in $\Gamma$.
	This means that at least one of $Y$, $Z$ is not a parent of $x$; say it is $Y$.
	Then $x$ must be the parent of $Y$, which means that $x\notin Y-\{x(Y)\}=Y^*$, a contradiction.

	Now
	let $e\in M(1)$. We claim that  one of the following holds:
	\begin{enumerate}[label=(\roman*)]
		\item $e\in E(Y^*)$ for some $Y\in \calM$,
		\item $e\in E(Y^*,g(Y)^*)$ for some $Y \in \calM-\calR$.
	\end{enumerate}	 
	Indeed, let $y,z$ denote the vertices of $e$.
	Since $M=\bigcup_{Y\in\calM} Y$, there is $Y\in \calM$
	such that $e\in Y(1)$, and thus $y,z\in Y(0)$.
	If both $y$ and $z$ are different from $x(Y)$,
	then $y,z\in Y^*$ and (i) holds.
	Otherwise, exactly one of $y,z$ equals $x(Y)$, say $y=x(Y)$ and
	$z\in Y^*$.
	Then $y$ is the parent of $Y$ in $\Gamma$.
	As explained in the previous paragraph, $y\in g(Y)$
	and $y$ is not a parent of $g(Y)$, so $y\in  g(Y)^*$ and (ii) holds.

\medskip

\Step{6}
	At this point, we assume that we are in the setting of Theorem~\ref{TH:cbe-for-quotient-sheaves}.
	For every $Y\in \calM$, put $\alpha_Y=w(Y^*)$
	and set 
	\[\alpha = \max\{\alpha_Y\where Y\in \calM\} .\]
	Since $\{Y^*\where Y\in \calM\}$ is a partition of $M(0)$
	(Step~5),
	we have
	$
	\sum_{Y\in\calM} \alpha_Y = w(M(0)) \leq 1$.
	
	Let $g^{-1}(Y)$ denote the set of blobs $Z\in\calM$ with $g(Z)=Y$
	(i.e.\ the grandchildren of $Y$).
	By Step 3 and the last paragraph of Step 5, we  have 
	\[
	1-\beta-t-s < w(M(1))\leq \sum_{Y\in\calM} w(E(Y^*))+\sum_{Y\in \calM} w(E(Y^*,
	{\textstyle\bigcup_{Z\in g^{-1}(Y)}} Z^*)).
	\]
	Put $\theta = \max\{|\mu|,|\lambda|\}$.
	By Theorem~\ref{TH:EML}, the right hand side is at most
	\begin{align*}
	&\sum_{Y\in\calM} (\alpha_Y^2+\lambda \alpha_Y )
	+
	\sum_{Y\in\calM } (2\alpha_Y{\textstyle\sum_{Z\in g^{-1}(Y)}\alpha_{Z}}+
	2\theta\sqrt{\alpha_Y {\textstyle\sum_{Z\in g^{-1}(Y)} \alpha_{Z}} })
	\\
	\leq &\sum_{Y\in\calM}  \alpha_Y(\alpha+\lambda)+
	\sum_{Y\in\calM } (2\alpha_Y{\textstyle\sum_{Z\in g^{-1}(Y)}\alpha_{Z}} +
	\theta (\alpha_Y+ {\textstyle\sum_{Z\in g^{-1}(Y)} \alpha_{Z}}) )\\
	=&\sum_{Y\in\calM }\alpha_Y(\alpha+\lambda)
	+\theta\sum_{Y\in \calM}\alpha_Y+
	\sum_{Y\in \calM}\sum_{Z\in g^{-1}(Y)}(2\alpha_Y+\theta)\alpha_Z\\
	\leq&
	\alpha+\lambda+\theta+\sum_{Z\in\calM-\calR}\alpha_Z(2\alpha_{g(Z)}+\theta)\\
	\leq&
	\alpha+\lambda+\theta+(2\alpha+\theta)=3\alpha+2\theta+\lambda.
	\end{align*}
	Thus,
	$
	1-\beta-t-s < 3\alpha+2\theta+\lambda
	$,
	and by rearranging, we get
	\[
	\alpha > \frac{1}{3}-\frac{\beta}{3}-\frac{t}{3}-\frac{s}{3}
	-\frac{\lambda+2\max\{|\lambda|,|\mu|\}}{3} .
	\]
	By the definition of $\alpha$,  there exists
	a blob $Y$ with $w(Y(0))\geq w(Y^*)=\alpha$.
	As explained in Step~1, it follows that
	$\dist_w(f,B^0(X,\calF))\leq 1-\alpha$, so this completes
	the proof  of the lemma under the assumptions
	of Theorem~\ref{TH:cbe-for-quotient-sheaves}.
	
\medskip

	\Step{7}
	Finally,
	suppose   that we are in the setting of Theorem~\ref{TH:cbe-for-partite-quotient-sheaves}.
	As in Step~6, we find that
	\[
	1-\beta-t-s <   \sum_{Y\in\calM} w(E(Y^*))+\sum_{Y\in \calM} w(E(Y^*,
	{\textstyle\bigcup_{Z\in g^{-1}(Y)}} Z^*)).
	\]
	Using Theorem~\ref{TH:EML-bipartite}(ii)
	and Theorem~\ref{TH:EML}(ii) (the assumption $\lambda\geq -\frac{1}{r}$ guarantees
	that $(X,w)$ is a $[-1,\lambda]$-expander  if we forget the $(r+1)$-partite structure of $X$), 
	we see that the right hand side is at most
	\begin{align*}
	&\sum_{Y\in\calM} (\alpha_Y^2+\lambda \alpha_Y )
	+
	\sum_{Y\in\calM } (\textstyle{\frac{2r+2}{r}}\alpha_Y{\textstyle\sum_{Z\in g^{-1}(Y)}\alpha_{Z}}+
	2r \theta \sqrt{\alpha_Y {\textstyle\sum_{Z\in g^{-1}(Y)} \alpha_{Z}} })
	\end{align*}
	and a computation similar to the one in Step~6 shows that
	this expression is bounded by 
	\[
	\alpha+\lambda+r\theta+(\textstyle{\frac{2r+2}{r}}\alpha+r\theta)=
	\textstyle{\frac{3r+2}{r}}\alpha+2r\theta+\lambda
	.\]
	As a result,
	\[
	\alpha > \textstyle{\frac{r }{3r+2}}[1-\beta -t-s-\lambda-2r\max\{|\lambda|,|\mu|\}]
	\]
	and we conclude the proof as in Step~6.
\end{proof}

We are now ready to prove Theorems~\ref{TH:cbe-for-quotient-sheaves}
and~\ref{TH:cbe-for-partite-quotient-sheaves}.

\begin{proof}[Proof of Theorem~\ref{TH:cbe-for-quotient-sheaves}]
	Let $f\in C^0(X,R_X/\calG)$.
	We need to show that $\|df\|_w\geq \veps\dist_w(f,B^0(X,R_X/\calG)$
	for $\veps:=\frac{2-4\lambda-4\max\{|\lambda|,|\mu|\}-5t-2s}{5-2\lambda}$.
	By replacing $f$ with a member of $f+B^0$ minimizing $\|\cdot\|_w$,
	we may assume that $\|f\|_w=\dist(f,B^0)$.
	
	Fix some $\beta\in [0,1]$, to be chosen later.
	If $\|d_0 f\|_w\geq \beta$, then we have
	$\|d_0 f\|_w\geq \beta\|f\|_w $.
	On the other hand, if $\|d_0 f\|<\beta$, then by Lemma~\ref{LM:key-lemma}, we have
	$\|f\|_w= \dist_w(f,B^0(X,\aug{R}/\calG))\leq 
	\frac{2}{3}+
	\frac{1}{3}[\beta+t+s+\lambda+2\max\{|\lambda|,|\mu|\}]$. By Lemmas~\ref{LM:small-set-expansion}
	and~\ref{LM:expansion-at-links} this means that
	\begin{align*}
	\|d_0f\|_w& \geq \left[2-t-2\lambda-(2-2\lambda)[\textstyle{\frac{2}{3}+\frac{\beta}{3}+\frac{t}{3}+\frac{s}{3}
	+\frac{\lambda+2\max\{|\lambda|,|\mu|\}}{3}}]\right]\|f\|_w\\
	& \geq  
	\left[\textstyle{[
	\frac{2}{3}
	-\frac{4}{3}\lambda -\frac{4}{3}\max\{|\lambda|,|\mu|\} -\frac{5}{3}t-\frac{2}{3}s] 
	-
	(\frac{2-2\lambda}{3})
	 \beta }\right]\|f\|_w
	\end{align*}
	As a result, $(X,w,\aug{R}/\calG)$ is an $\veps$-coboundary expander in dimension $0$ for
	\[
	\veps=\min\{\beta, \textstyle{[
	\frac{2}{3}
	-\frac{4}{3}\lambda -\frac{4}{3}\max\{|\lambda|,|\mu|\} -\frac{5}{3}t-\frac{2}{3}s] 
	-
	(\frac{2-2\lambda}{3})
	 \beta}\}
	\]
	The right hand side is maximized when
	$\veps=\beta=\frac{2-4\lambda-4\max\{|\lambda|,|\mu|\}-5t-2s}{5-2\lambda}$, and theorem follows.
\end{proof}

\begin{proof}[Proof of Theorem~\ref{TH:cbe-for-partite-quotient-sheaves}]
	Similarly to the proof of Theorem~\ref{TH:cbe-for-quotient-sheaves},
	we find that for every $\beta\in [0,1]$, 
	\[
	\cb_0(X,w,R_X/\calG)\geq \min\{\beta, \textstyle{[
	\frac{2r }{3r+2}
	-\frac{4r }{3r+2}\lambda -\frac{4r^2}{3r+2}\theta -\frac{5r+2}{3r+2}t-\frac{2r}{3r+2}s] 
	-
	 \frac{ (2-2\lambda)r}{3r+2} 
	 \beta}\},
	\] 
	where $\theta=\max\{|\lambda|,|\mu|\}$. The maximum of the right hand
	size  is attained 
	for $\beta=\frac{2r-4r\lambda-4r^2\theta-(5r+2)t-2r s}{ 5r+2 -2r\lambda}$,
	hence the theorem.
\end{proof}

\begin{remark}\label{RM:eliminating-s}
	In Theorems~\ref{TH:cbe-for-quotient-sheaves}
	and~\ref{TH:cbe-for-partite-quotient-sheaves},
	if all edges in $X$ have the same weight, or if \emph{all} the subgroups
	$\{R_x\}_{x\in X}$ are linearly disjoint, then we can eliminate
	the terms $-2s$, resp.\ $-2rs$, in the lower bound for $\cb_0(X,w,R_X/\calG)$,
	or equivalently take $s=0$.
	
	Indeed, if all edges in $X$ have the same weight,
	then the assertion $w(M(1))>1-\beta-t-s$ in Step~3 of the proof of Lemma~\ref{LM:key-lemma}
	implies that $w(M(1))\geq 1-\beta-t$.
	Likewise, if all the subgroups
	$\{R_x\}_{x\in X}$ of $R$ are linearly disjoint, then
	in the same place in the proof,	
	we can apply Lemma~\ref{LM:weight-of-subforest}(i) instead of Lemma~\ref{LM:weight-of-subforest}(ii)
	and get that $w(M(1))> 1-\beta-t$.
	Carrying the entire proof of
	Lemma~\ref{LM:key-lemma} using the inequality
	$w(M(1))\geq 1-\beta-t$ in place of $w(M(1))>1-\beta-t-s$
	allows
	us to eliminate $s$ at the cost of replacing the strict inequalities in the proof with non-strict inequalities.
	This, in turn, eliminates $s$ from Theorems~\ref{TH:cbe-for-quotient-sheaves}
	and~\ref{TH:cbe-for-partite-quotient-sheaves}.
\end{remark}

\begin{cor}\label{CR:cse-quotient-sheaves}
	With notation and assumptions
	as in Theorem~\ref{TH:cbe-for-quotient-sheaves}
	(resp.\ Theorem~\ref{TH:cbe-for-partite-quotient-sheaves}),
	let $\calF$ be the subsheaf of $R_X/\calG$ determined
	by $\calF(\emptyset)=0$ and $\calF(x)=(R_X/\calG)(x)$ for all nonempty $x\in X$,
	and put $\eta = \max\{w(v)\where v\in X(0)\}$ (we always have $\eta\leq \frac{s}{t}$).
	Then 
	$\cse_0(X,w,\calF)\geq \frac{2-4\lambda-4\max\{|\lambda|,|\mu|\}-5t-2s}{5-2\lambda}$
	(resp.\
	$\cse_0(X,w,\calF)\geq \frac{2r-4r\lambda-4r^2\max\{|\lambda|,|\mu|\}-(5r+2)t-2r s}{ 5r+2 -2r\lambda}$)
	and 
	$\ccd_0(X,w,\calF)\geq 1-\eta$.	
\end{cor}

\begin{proof}
	The statement about $\cse_0(X,w,\calF)$
	is a consequence of
	Theorem~\ref{TH:cbe-for-quotient-sheaves}
	(resp.~\ref{TH:cbe-for-partite-quotient-sheaves}) and
	Remark~\ref{RM:coboundary-to-cosystolic}.
	The latter also tells us that
	in order to prove the lower bound on $\ccd_0(X,w,\calF)$,
	it is enough to show that for all $f\in R=C_{-1}(X,R_X/\calG)$,
	we have $\|d_{-1} f\|\geq 1-\eta$.
	Observe that $(d_{-1} f)(v)=f+R_v\in R/R_v$ for all $v\in X(0)$.
	Thus, $\|d_{-1} f\|=1-w(A)$,
	where $A$ is the set of $0$-faces $v$ such that $f\in R_v$.
	Since every two of the groups $\{R_v\}_{v\in X(0)}$
	are linearly disjoint, $A$ is either empty or a singleton,
	so $w(A)\leq \eta$ and the corollary follows.
	(It is worth noting that if $R_v\neq 0$ for all $v\in X(0)$, then
	by choosing $f\in R_v$ with $w(v)$ maximal, we  
	get $\|d_{-1}f\|=1-\eta$, so the lower bound $1-\eta$ cannot be improved in general.)
\end{proof}

\section{The Case of Finite Buildings}
\label{sec:buildings}

We now apply Corollary~\ref{CR:expanders-are-cb-exps}
and Theorem~\ref{TH:cbe-for-partite-quotient-sheaves} 
to the $0$-skeleton of finite buildings admitting a strongly transitive group action, giving
upper bounds on the coboundary expansion in terms of the \emph{thickness} and the \emph{type} of the building. 
This is the main reason why we have taken care
to treat the case of $(r+1)$-partite weighted simplicial complexes.
Applications of the results of this section to locally testable codes
appear in \cite[\S9]{First_2023_sheaves_on_complexes_preprint}.

\medskip

We refer the reader to \cite{Abramenko_2008_Buildings}  for 
an extensive discussion of buildings.
Here we satisfy with saying that a building is a  possibly-infinite  simplicial complex
$X$ equipped with a collection $\calE$ of subcomplexes called \emph{apartments}
satisfying certain axioms.
All the apartments of $X$
are isomorphic to each other and to a \emph{Coxeter complex} $\Sigma$.
The data of $\Sigma$ (up to isomorphism) is encoded by a \emph{Coxeter diagram}
$T$, which is a finite graph with edges labelled by elements from the set
$\{3,4,5,\dots\}\cup\{\infty\}$ (unlabelled edges are given the label $3$ by default).
We write $T=T(X)$ and say that $T$ is the \emph{type} of $X$.
There is a labelling of the vertices
of $X$ by the vertices of $T$ making $X$ into a pure $(r+1)$-partite simplicial complex,
where $r=\dim X= |V(T)|-1$. The type of a face $z\in X$  is the set $t(z)\subseteq V(T)$ consisting
of the types of the vertices of $z$.
If $\dim z\leq \dim X-1$, then the link $X_z$ is also a building and its type $T(X_z)$ is the Coxeter diagram
obtained  from $T=T(X)$ by removing the vertices in $t(z)$.

Let $B$ denote the $V(T)\times V(T)$ real matrix determined by
\[
B_{u,v}=\left\{\begin{array}{ll}
1 & u=v\\
-\cos\frac{\pi}{m} & \text{$\{u,v\}\in T(1)$ and $m$ is the label of $\{u,v\}$}\\
0 & \{u,v\}\notin T(1).
\end{array}\right.
\]
The building $X$ is called \emph{spherical} if $B$ is positive
definite, or
equivalently, if its apartments are finite. 
Following \cite[Chapter~10]{Abramenko_2008_Buildings},
we call  $X$ \emph{affine} if $B$ is positive semidefinite
and $\rank B=|V(T)|-1$ (consult \cite[Proposition~10.44]{Abramenko_2008_Buildings}).
See \cite[p.~50, Remark~10.33(b)]{Abramenko_2008_Buildings} for a complete list of the possible
Coxeter diagrams of spherical and affine buildings.
If $X$ is an \emph{affine or spherical} building and $z$ is a face with $0\leq \dim z<\dim X$, then the link $X_z$
is a \emph{spherical} building.

A building $X$ of dimension $r$ is called $q$-\emph{thick} if every $(r-1)$-face of $X$ is contained in at least $q$ $r$-faces. We say that $X$ is \emph{thick} if it
is $3$-thick.  

Recall from \cite[\S6.1.1]{Abramenko_2008_Buildings}
that an $r$-dimensional building $X$ is said to posses a \emph{strongly transitive action}
if there is a group $G$ acting on $X$ via type-preserving simplicial automorphisms
such that $G$  takes apartments to apartments 
and acts transitively on the set of pairs $(A,x)$ consisting of an apartment
$A\in \calE$ and an $r$-face $x$ in $A$. Since $G$ is type preserving,
this means that $G$ acts transitively on the set of faces of a fixed type
$t\subseteq T(X)(0)$.
Moreover, for every $z\in X$ of dimension $ \dim X-2$ or less, the building $X_z$ 
also possesses a strongly transitive action.
It follows from Tits' classification of thick spherical and affine buildings, 
see \cite[Chapter~9, \S11.9]{Abramenko_2008_Buildings} for a survey,
that all thick finite buildings of dimension $\geq 2$ and all locally-finite thick affine buildings
of dimension $ \geq 3$ admit a strongly transitive action.

\begin{example}\label{EX:An-building}
	Let $\F$ be a field and let $n\in\N$.
	The incidence complex  of nontrivial subspaces of $\F^{n+1}$,
	denoted $A_n(\F )$, is an $(n-1)$-dimensional building of type $A_n$,
	where $A_n$ is the Coxeter diagram consisting of a single path
	of length $n-1$ with all edges labeled $3$. In more detail,
	the vertices of $A_n(\F )$ are the nontrivial subspaces of $\F^{n+1}$
	and its faces are the sets of vertices which are totally ordered by inclusion.
	The apartments of $A_n(\F )$ are induced from bases of $\F^{n+1}$
	as follows: If $E$ is a basis of $\F^{n+1}$, then the collection of
	faces $x=\{V_0,\dots,V_i\}\in A_n(\F )$ for which each $V_j$ is spanned
	by a subset of $E$ is an apartment.
	It is possible to name the vertices of the Coxeter diagram of $A_n(\F)$
	by $1,\dots,n-1$ in such a way that the type of
	a vertex is its dimension as an $\F$-vector space.
	
	The group $\nGL{\F}{n+1}$ acts on $A_n(\F)$ via its standard 
	action on $\F^{n+1}$. This action is type-preserving,
	and it is strongly transitive because
	$\nGL{\F}{n+1}$ acts transitively on the set of bases of $\F^{n+1}$.
	
	When $n=2$, the graph $A_2(\F)$ is nothing but the incidence graph
	of points and lines in the $2$-dimensional projective plane over $\F$.
\end{example}

Let $X$ be a building of type $T=T(X)$.
We  define
\[
m(X)=
m(T)=\max\Circs{\left\{2\right\}\cup\left\{n\where \text{$T$ has an edge labelled $n$}\right\}}.
\]
Observe   that $m(X)\geq m(X_z)$ for every face $z\in X$ of dimension $<\dim X-1$.
This definition is motivated by the following theorem, which we derive
from results of Evra--Kaufman  \cite{Evra_2016_cosystolic_expanders_arxiv_version}
(see also \cite{Evra_2016_cosystolic_expanders})
and Oppenheim \cite{Oppenheim_2015_vanishing_of_cohomology}.

\begin{thm}\label{TH:good-expansion-of-buildings}
	Let $X$ be a (finite) $r$-dimensional simiplicial complex such that one of the following hold:
	\begin{enumerate}[label=(\arabic*)]
		\item $X$ is a  $q$-thick spherical building admitting a strongly transitive action;
		\item the universal covering of $X$ is a  $q$-thick affine building of dimension $\geq 2$ 
		admitting a strongly
		transitive action, and the labeling of its vertices descends to $X$.
	\end{enumerate}
	In both cases, $X$ has the structure of a pure  $(r+1)$-partite simplicial complex.
	Let $w$ denote the canonical weight function of $X$ (Example~\ref{EX:regular-graph}),
	let $T$ denote the Coxeter diagram of the building   mentioned in (1) or (2)
	and put $m=m(T)$. Then $(X ,w)$ is an  $(r+1)$-partite $[-r\lambda,\lambda]$-expander
	for
	\[
	\lambda = \frac{\sqrt{m-2}}{ \sqrt{q}-(r-1)\sqrt{m-2}},
	\]
	provided $q \geq r^2(m-2)$. Furthermore, $m\leq 8$ when (1) holds, and $m\leq 6$
	when (2) holds.
\end{thm}

\begin{proof}
	Suppose first that $r=1$. Then only case (1) is possible. 	
	Thus,
	$X$ is a $1$-dimensional building admitting a strongly transitive action
	by a group $G$. 
	The Coxeter graph of $X$ consists of two vertices
	and one edge labelled $m$,
	so    each apartment in $X$ is a cycle
	graph of length $2m$. 
	We label the  vertices of $T=T(X)$ by $0$ and $1$
	and write $X_{\{i\}}$ ($i\in\{0,1\}$) for the $0$-faces of type
	$\{i\}$ in $X$. Since $G$ acts transitively on $X_{\{0\}}$ and $X_{\{1\}}$,
	$X$ is a biregular graph.
	Write $n_i=|X_{\{i\}}|$  and let $k_i$ denote the number of edges
	containing a $0$-face in $X_{\{i\}}$. Then $|X(1)|=n_0k_0=n_1k_1$,
	which means that $w(e)=\frac{1}{n_0k_0}=\frac{1}{n_1k_1}$
	for all $e\in X(1)$  and 
	$w(x )=\frac{1}{2n_i}$ for all $x\in X_{\{i\}}$.
	We may assume without loss of generality that $k_0\leq k_1$.
	Note also that
	$q\leq k_0$.
	We now adapt 
	the proofs of \cite[Propositions~5.21, 5.22]{Evra_2016_cosystolic_expanders_arxiv_version}
	to our  \emph{weighted graph} situation,
	and also improve the expansion constants, in order to show that
	$(X,w)$ is a bipartite $[-\sqrt\frac{m-2}{{q}},\sqrt\frac{m-2}{{q}}]$-expander.

	Let $f'\in C^0_\diamond(X,\R)$ (notation as in \S\ref{subsec:partite-complexes})
	denote a nonzero eigenfunction  of $\calA:=\calA_{X,w}$ 
	with eigenvalue $\lambda$. We need to prove that $\lambda^2\leq \frac{m-2}{q}$.
	It is enough to consider the case $\lambda\neq 0$.
	In this case,   $f'$ cannot vanish on $X_{\{0\}}$. Thus, we may choose
	$s\in X_{\{0\}}$ with $f'(s)\neq 0$ and put $K=\{g\in G\suchthat gs=s\}$. 
	By applying Lemma~\ref{LM:bipartite-sym-spec} to $f'$, we
	see that
	there exists   $f''\in C^0_{\diamond}(X,\R)$ with $\calA f''=-\lambda f''$
	and $f''(s)\neq 0$.
	
	Given $0$-faces $x,y\in X(0)$,
	we
	denote the $K$-orbit of $x\in X(0)$ by $[x]$,
	and write  $d(x,y)$ for
	the  length
	of the shortest path from  $x$ to $y$ in $X$.
	We claim that following hold (cf.\ \cite[Definition~5.20]{Evra_2016_cosystolic_expanders_arxiv_version}):
	\begin{enumerate}[label=(\roman*)]
		\item The number of $K$-orbits in $X(0)$ is exactly  $m+1$.
		\item For all $i\in\{0,1\}$, the $0$-faces in $X_{\{i\}}$ of maximal distance
		from $s$ form a $G$-orbit. This maximal distance is $m-((m+i)\bmod 2)$.
		\item If $x\in X(0)-\{s\}$ is not of maximal distance from $s$,
		then there is exactly one $0$-face $y$ adjacent to $x$ with $d(s,y)<d(s,x)$.
	\end{enumerate}
	This is similar to the proof of
	\cite[Propositions~5.22]{Evra_2016_cosystolic_expanders_arxiv_version}:
	To see (i), fix an apartment $E$ containing $s$
	and apply \cite[Lemma~5.16]{Evra_2016_cosystolic_expanders_arxiv_version}
	to conclude that every $K$-orbit in $X(0)$ meets $E$. Let $x,y\in E(0)$ be the $0$-faces adajcent to $s$.
	By the strong transitivity of the $G$-action,   there is $g\in G$ such that $g(E)=E$
	and $g\{s,x\}=g\{s,y\}$. Since $G$ preserves types,   $g(s)=s$, so $g$ is a reflection
	of the $2m$-cycle   $E$ fixing $s$. This means that all except possibly $2$ $K$-orbits
	meet $E$ in at least two vertices, so there can be at most $m+1$ $K$-orbits.
	On the other hand, since the action of $G$ preserves distances and there are $0$-faces
	of distance $m$ in $X$, the number of $K$-orbits on $X(0)$ must be at least $m+1$. This proves (i).
	The first assertion of (ii) is shown exactly as in {\it op.\ cit.},
	and the second assertion follows from the fact that $s\in X_{\{0\}}$ and each apartment
	is a cycle of length $2m$.
	As for (iii), let $D$ be the set of  $0$-faces  $y\in X(0)$
	adjacent to $x$ with $d(s,y)<d(s,x)$. It is shown in 
	the proof \cite[Propositions~5.22]{Evra_2016_cosystolic_expanders_arxiv_version} that
	there is an apartment $E$ of $X$ with $s,x\in E$ and $D\subseteq E$.
	Since $E$ is a cycle graph, $D$ must be a singleton.

	Next, let us regard $\calA=\calA_{X,w}$ as an operator from the \emph{complex} vector space $C^0(X,\C)$ to
	itself; this does not affect the spectrum of $\calA$. Put $Y=K{\setminus} X$ 
	and let $C^0(Y,\C)$ denote the set of functions from $Y$ to $\C$ (note that $Y$ is not a graph
	in general). We define a linear operator $\calB:C^0(Y,\C)\to C^0(Y,\C)$ by
	\[(\calB g)[x] = \sum_{y\in X(1)_x} \frac{w(x\cup y)}{2w(x)} g[y]\]
	for all $g\in C^0(Y,\C)$, $[x]\in K\leftmod X(0)$.
	Note that this does not depend on the representative $x$ in the orbit $[x]$. 
	
	We claim that $\Spec \calB\subseteq \Spec \calA$; in particular, $\Spec \calB\subseteq \R$.
	Indeed, given $\mu\in \C$ and $g\in C^0(Y,\C)$ with $\calB g=\mu g$,
	the function $f\in C^0(X,\C)$ defined by $f(x)=g[x]$ satisfies
	$\calA f=\mu f$ because, for all $x\in X(0)$,
	\begin{align*}
		(\calA f)(x) &=
		\sum_{y\in X(0)_x} \frac{w(x\cup y)}{w(x)} f(y)
		=
		\sum_{y\in X(0)_x} \frac{w(x\cup y)}{w(x)} g[y]
		=(\calB g)[x] =\mu g[x] = \mu f(x).
	\end{align*}
	
	Next, we claim that the $\calA$-eigenvalue $\lambda$ corresponding
	to $f'\in C_{\diamond}^0(X,\R)$ is in $\Spec \calB$. Indeed,
	define $g' \in C^0(Y,\C)$   by $g'[y]=\sum_{k\in K} f'(ky)$
	and note that $g'[s]=|K|f'(s)\neq 0$. Then, for
	every $x\in X(0)$, 
	\begin{align*}
		(\calB g')[x] & = \sum_{y\in X(1)_x} \frac{w(x\cup y)}{2w(x)} g'[y]
		=\sum_{y\in X(1)_x}\sum_{k\in K} \frac{w(x\cup y)}{2w(x)} f'(ky) \\
		&=\sum_{k\in K} \sum_{y\in X(1)_x}\frac{w(kx\cup ky)}{2w(kx)} f'(ky)
		=\sum_{k\in K} \sum_{y\in X(1)_{kx}}\frac{w(kx\cup  y)}{2w(kx)} f'( y)\\
		&=\sum_{k\in K} (\calA f)(kx) =\sum_{k\in K} \lambda f(kx)=\lambda g'[x].
	\end{align*}
	Thus,  $\calB g'=\lambda g'$ and
	$\lambda\in\Spec \calB$. Applying this argument for the $\calA$-eigenfunctions
	$f''$, $1_{X(0)}$ and $1_{X_{\{0\}}}-1_{X_{\{1\}}}$ shows that
	we also have $1,-1,-\lambda \in \Spec B$.
	Since $\lambda\neq -\lambda$ 
	(because $\lambda\neq 0$) 
	and $\Spec \calB\subseteq \R$, this means that
	\begin{align}\label{EQ:TH:good-expansion-of-buildings:eq1}
	2+2\lambda^2= (-1)^2+(-\lambda)^2+\lambda^2+1^2\leq \Tr(\calB^2).
	\end{align}
	
	We now bound $\Tr(\calB^2)$ from above. Fix a set of representatives
	$U$ for $K{\setminus} X(0)$ and write $1_{[x]}\in C^0(Y,\C)$ for the characteristic
	function of $\{[x]\}$. Then
	\begin{align*}
		\Tr(\calB^2) &=\sum_{x\in U}(\calB^2 1_{[x]})[x] 
		=\sum_{x\in U}\sum_{y\in X(1)_x} \frac{w(x\cup y)}{2w(x)}(\calB 1_{[x]})[y]\\
		&=\sum_{x\in U}\sum_{y\in X(1)_x}\sum_{z\in X(1)_y} \frac{w(x\cup y)w(y\cup z)}{4w(x)w(y)}1_{[x]}[z]
		=\frac{1}{k_0k_1}\sum_{x\in U}|L(x)|,
	\end{align*}
	where $L(x)$ is the set of triples $(x,y,z)\in X(0)^3$ with   $\{x,y\},\{y,z\}\in X(1)$ and $z\in [x]$.
	Let $s' $  be the unique $0$-face in
	$U$ with $d(s,s')=m$ and let $s''$ be the unique
	$0$-face in $U$ with   $d(s,s'')=m-1$; they exist by (ii) above.
	Put $\epsilon=m\bmod 2$ and note that $s'\in X_{\{\veps\}}$.
	We analyze the size of $L(x)$ by splitting into four cases:
	\begin{enumerate}[label=\Roman*)]
		\item $x=s' $: Let $(s' ,y,z)\in L(s' )$. There are $k_{\epsilon}$ possibilities for $y$.
		Since $d(s,y)=m-1>0$, there is some $z'\in X(1)_y$ with $d(s,z')<d(s,y)<d(s,s' )=d(s,z)$,
		so $z'\notin [z]$. This means that, for each $y$, there are at most $k_{1-\epsilon}-1$ possibilities
		for $z$. Thus, $|L(s')|\leq k_\epsilon(k_{1-\epsilon}-1)=k_0k_1-k_{1-\epsilon}$.
		\item $x=s''$: $|L(s'')|\leq k_0k_1$; in fact, this holds for any $x\in U$.
		\item $x=s$: Since $[s]=\{s\}$,
		we have $L(s)=\{(s,y,s)\where y\in X(1)_{s}\}$, so $|L(s)|=k_{0}$.
		\item $x\neq s,s',s''$: Write $t(x)=\{i\}$.
		If  $(x,y,z)\in L(x)$, then  both $x$ and $y$ are not of 
	maximum distance from $s$ in $X(0)$. Noting that $d(s,x)=d(s,z)$ (because $z\in [x]$),
	(iii) implies that $y$ is uniquely determined by $x$ if $d(s,y)<d(s,x)$ and $z$ is uniquely
	determined by $y$ in if $d(s,y)>d(s,z)$.
	Thus, $|L(x)|\leq 1\cdot k_{1-i} + k_i \cdot 1=k_0+k_1$.
	\end{enumerate}
	By (i), $|U|\leq m+1$, so we conclude that
	\[
	\Tr(\calB^2)\leq \frac{(k_0k_1-k_{1-\epsilon})+k_0k_1+k_{0}+(m-2)(k_0+k_1)}{k_0k_1} \leq 2+\frac{2(m-2)}{ q},
	\]
	where the inequality holds because $q\leq k_0\leq k_{1-\epsilon}$.
	Combining this with \eqref{EQ:TH:good-expansion-of-buildings:eq1}
	gives the desired conclusion
	\[
	\lambda^2 \leq   \frac{m-2}{ q} .
	\]
	This proves the theorem when $r=1$.
	
	Suppose now that $r>1$. Assumptions  (1) and (2) imply that 
	all the positive-dimensional links of $X$ are connected,
	and for every
	$z\in X(r-2)$, the complex $X_z$ is a $1$-dimensional spherical building admitting
	a strongly transitive action  and having $m(X_z)\leq m $. 
	Since we proved the theorem when $r=1$, 
	the weighted
	graph $(X_z,w_{X_z})$ is a $2$-partite $[-\sqrt{\frac{m-2}{ q}},\sqrt{\frac{m-2}{ q}}]$-expander.
	Moreover, our assumption $q\geq r^2(m-2)$ implies that
	$\sqrt{\frac{m-2}{q}}\leq \frac{1}{r }$.
	Thus,  by a theorem of Oppenheim \cite[Corollary~5.6]{Oppenheim_2015_vanishing_of_cohomology}\footnote{
		There is a typo in \cite[Corollary~5.6]{Oppenheim_2015_vanishing_of_cohomology}:
		 the expression
		``$1-(n-k)f^{n-k-2}(\lambda)$'' should be ``$1+(n-k-1)f^{n-k-2}(\lambda)$''.
	}
	(see also the formula at the end of \cite[Theorem~1.4]{Oppenheim_2015_vanishing_of_cohomology}), 
	$(X,w)$ is an $(r+1)$-partite $[-r\lambda,\lambda]$-expander for
	\[\lambda = 1-\frac{r(1-\sqrt{\frac{m-2}{ q}})-(r-1)}{(r-1)(1-\sqrt{\frac{m-2}{ q}})-(r-2)}
	=\frac{\sqrt{m-2}}{ \sqrt{q}-(r-1)\sqrt{m-2}}.\qedhere\]
\end{proof}

\begin{example}
	Let $\F_q$ be a finite field
	with $q$ elements and let  $X=A_2(\F_q)$ (notation as in Example~\ref{EX:An-building}). 
	Then $X$ is a $(q+1)$-thick building
	admitting a strongly transitive action, and $m(X)=m(A_2)=3$.
	Thus, by Theorem~\ref{TH:good-expansion-of-buildings},
	$(X,w_X)$ is a bipartite $[-\frac{1}{\sqrt{q+1}},\frac{1}{\sqrt{q+1}}]$-expander.
	This agrees with the well-known fact that $\Spec (\calA_{X,w})= 
	\{\pm 1,\pm\frac{\sqrt{q}}{q+1}\}$ with $1$
	and $-1$ occurring with multiplicity $1$.
	(In fact, the counting argument in cases I)--IV) in the proof can be slightly improved
	to make the bounds match.) 
\end{example}

\begin{cor}\label{CR:cbe-buildings-constant}
	Let $X,w,r,q,m$ be as in Theorem~\ref{TH:good-expansion-of-buildings}
	and assume   $q\geq r^2({m-2})$.
	Then  $(X,w,\aug{R})$
	is a $(1-\frac{\sqrt{m-2}}{\sqrt{q}-(r-1)\sqrt{m-2}})$-coboundary
	expander in dimension $0$  for every nontrivial abelian group $R$.
\end{cor}

\begin{proof}
	This follows from Theorem~\ref{TH:good-expansion-of-buildings} and
	Corollary~\ref{CR:expanders-are-cb-exps}.
\end{proof}

\begin{lem}\label{LM:edge-vertex-weight-ratio}
	Let $X$ be a pure $r$-dimensional (finite) simplicial complex
	and let $w$ be its canonical weight function.
	Suppose that $X$ is $q$-thick, i.e., every $(r-1)$-face of $X$ is contained in at least $q$ $r$-faces.
	Then, for every edge $e\in X(1)$ and $0$-face $v\subseteq e$, we
	have $\frac{w(e)}{w(v)}\leq \frac{2}{q+r-1}$.
\end{lem}

\begin{proof}
	Fix $e\in X(1)$ and let $u,v$ be the $0$-faces of $e$.
	It follows readily
	from the defining properties
	of $w$ in \S\ref{subsec:weights}
	that $\frac{w(e)}{w(v)}=\frac{2|X(r)_{\supseteq e}|}{r|X(r)_{\supseteq v}|}$.
	For $c\in X(r)_{\supseteq e}$,
	define $N(c):=X(r)_{\supseteq c-u} -\{c\}$.
	Since $X$ is $q$-thick,
	$|N(c)|\geq q-1$, and thus
	$\sum_{c\in X(r)_{\supseteq e}} |N(c)|\geq (q-1)|X(r)_{\supseteq e}|$.
	
	We now bound $\sum_{c\in X(r)_{\supseteq e}} |N(c)|$ from above.
	Let $c\in X(r)_{\supseteq e}$.
	Then every member of $c'\in  N(c)$
	lies in $X(r)_{\supseteq v}-X(r)_{\supseteq e}$.
	Fixing $c'\in X(r)_{\supseteq v}-X(r)_{\supseteq e}$,
	we claim that  $c'$ occurs as a member of $N(c)$ for 
	at most $r$ faces $c\in X(r)_{\supseteq e}$.
	Indeed, if $c\in X(r)_{\supseteq e}$ is an $r$-face
	such that $c'\in N(c)$, then $c\cup c'=c'\cup u$, which means that $e\subseteq c\subseteq c'\cup u$.
	Since  $|c'\cup u|=r+2$ and $|e|=2$, there are at most $r$ possibilities
	for $c$, proving our claim. 
	As a result of the claim, $
	r\left(|X(r)_{\supseteq v}|-|X(r)_{\supseteq e}|\right)\geq
	\sum_{c\in X(r)_{\supseteq e}}|N(c)|$.
	
	Putting everything together gives
	\[
	r\left(|X(r)_{\supseteq v}|-|X(r)_{\supseteq e}|\right)\geq
	\sum_{c\in X(r)_{\supseteq e}} |N(c)|\geq (q-1)|X(r)_{\supseteq e}|,
	\]
	and by rearranging, we find that
	$\frac{|X(r)_{\supseteq e}|}{|X(r)_{\supseteq v}|}\leq \frac{r}{q+r-1}$,
	so $\frac{w(e)}{w(v)}\leq \frac{2}{r}\cdot \frac{r}{q+r-1}=\frac{2}{q+r-1}$.
\end{proof}

\begin{cor}\label{CR:expansion-of-quotient-sheaves-on-buildings}
	Let $X,w,r,q,m$ be as in Theorem~\ref{TH:good-expansion-of-buildings}
	and assume $q\geq r^2({m-2})$.
	Let $R$ be a  abelian group and let $\{R_x\}_{x\in X-\{\emptyset\}}$ be subgroups
	of $R$ satisfying
	conditions (1) and (2) of
	Theorem~\ref{TH:cbe-for-quotient-sheaves} (after setting $R_\emptyset = \{0_R\}$), 
	e.g., such that the summation map $\bigoplus_{x\in X}R_x\to R$ is injective.
	For every $x\in X$, put $\calG(x)=\sum_{y\subseteq x} R_y$
	so that $\calG$ becomes a subsheaf of $R_X$.
	Then 
	\[
	\cb_0(X,w,R_X/\calG)\geq \frac{2r}{5r+2} -
	\frac{(4r^3+4r)\sqrt{m-2}}{(5r+2)(\sqrt{q}-(r-1)\sqrt{m-2})}
	-
	\frac{14r +4}{(5r+2)(q+r-1)}.
	\]
\end{cor}

Note that the lower-bound on $\cb_0(X,w,R_X/\calG)$ approaches
$\frac{2r}{5r+2}$ as the thickness $q$ tends to $\infty$.

\begin{proof}
	By Theorem~\ref{TH:good-expansion-of-buildings},
	$(X,w)$ is an $(r+1)$-partite $[-r\lambda, \lambda]$-expander
	for $\lambda=\frac{\sqrt{m-2}}{q-(r-1)\sqrt{m-1}}$,
	and by Lemma~\ref{LM:edge-vertex-weight-ratio},
	we have $t:=\max\{\frac{w(e)}{w(x)}\where e\in X(1), x\in X(0)_{\subseteq e}\}\leq 
	\frac{2}{q+r-1}$. In particular, $s:=\max\{w(e)\where e\in X(1)\}\leq t= \frac{2}{q+r-1}$.
	Now apply Theorem~\ref{TH:cbe-for-partite-quotient-sheaves}. (The lower bound
	guaranteed   by that theorem is slightly better than the simpler
	expression given in the corollary.)
\end{proof}

\begin{remark}\label{RM:expansion-of-quotient-sheaves-on-buildings}
	When $r=\dim X>1$,
	the  constants $t$ and $s$ of Theorem~\ref{TH:cbe-for-partite-quotient-sheaves}
	are often much smaller than the bounds used in the proof of Corollary~\ref{CR:expansion-of-quotient-sheaves-on-buildings}. 
	Using a better upper bound on  $t$ and $s$ 
	will result in decreasing
	the right fraction 
	in the lower bound for $\cb_0(X,w,R_X/\calG)$ in Corollary~\ref{CR:expansion-of-quotient-sheaves-on-buildings}.	
	Also, when $r=\dim X=1$, we may eliminate $s$ by Remark~\ref{RM:eliminating-s},
	thus changing the said  right fraction  to
	$t=\frac{2}{ (q+r-1)}$.
\end{remark}

\begin{example}
	The lower bounds on $\cb_0(X,w,\aug{R}_X/\calG)$ 
	provided by Corollary~\ref{CR:expansion-of-quotient-sheaves-on-buildings}
	and Remark~\ref{RM:expansion-of-quotient-sheaves-on-buildings}
	are detailed in the following table for some spherical $q$-thick   buildings of dimensions
	$1$ and $2$ admitting a strongly transitive action.
	\begin{center}
	\begin{tabular}{|c|c|c|l|l|}
	\hline
	$\dim X$ & $T(X)$ & $m$ & $\cb_0(X,w,\aug{R}_X/\calG)$ & $ > 0$ if \\
	\hline
	$1$ & $A_2$ & $3$ & $\frac{2}{7}-\frac{8 }{7\sqrt{q}}-\frac{2}{q}$ & $q\geq 29$ \\
	& $C_2$ & $4$ & $\frac{2}{7}-\frac{8\sqrt{2} }{7 \sqrt{q} }-\frac{2}{q}$ & $q\geq 45$ \\
	& $G_2$ & $6$ & $\frac{2}{7}-\frac{16 }{7 \sqrt{q} }-\frac{2}{q}$ & $q\geq 78$ \\
	\hline
	$2$ & $A_3$ & $3$ & $\frac{1}{3}-\frac{10 }{3(\sqrt{q}-1)}-
	\frac{8}{3(q+1)}$ & $q\geq 136$ \\
	 & $C_3$ & $4$ & $\frac{1}{3}-\frac{10\sqrt{2} }{3(\sqrt{q}-\sqrt{2})}-
	\frac{8}{3(q+1)}$ & $q\geq 257$\\
	\hline
	\end{tabular}	
	\end{center} 
\end{example}

\section{Further Questions}
\label{sec:questions}

We finish with several questions about possible
extensions of Theorems~\ref{TH:cbe-for-quotient-sheaves} and~\ref{TH:cbe-for-partite-quotient-sheaves}.

If not   indicated otherwise, 
$(X,w)$ is assumed to be a weighted graph (resp.\ $(r+1)$-partite weighted simplicial complex)
which is   a $[-\lambda,\lambda]$-expander (resp.\ $(r+1)$-partite $[-\lambda,\lambda]$-expander)
for some
$\lambda>0$. 
We let $R$ be a nontrivial abelian group, $\{R_x\}_{x\in X}$ be subgroups of $R$ with $R_\emptyset=0$,
and define the subsheaf $\calG$ of $\aug{R}_X$   as in Theorem~\ref{TH:cbe-for-quotient-sheaves},
i.e., $\calG(x)=\sum_{y\subseteq x}R_y$.

\medskip

As $\lambda\to 0^+$, the lower bound on the $0$-dimensional coboundary expansion  of
$(X,w,R_X/\calG)$ provided by 
Theorem~\ref{TH:cbe-for-quotient-sheaves}
(resp.\ Theorem~\ref{TH:cbe-for-partite-quotient-sheaves}) approaches $\frac{2}{5}$ (resp.\
$\frac{2r}{5r+2}$). We expect that this could be improved. 

\begin{que}
	Provided that all the $\{R_x\}_{x\in X}$ are linearly disjoint in $R$,
	is it the case that $\cb_0(X,w,R_X/\calG)$
	approaches $1$ as $\lambda\to 0^+$?
\end{que}

Next, we ask whether condition (1) of Theorem~\ref{TH:cbe-for-quotient-sheaves} can
be relaxed. 

\begin{que}\label{QE:order-of-magnitude}
	Let $m:\N\to \N\cup\{0\}$ be a    function satisfying $0\leq m(n)\leq n$ for all $n\in\N$.
	Assuming $\lambda$ is fixed and sufficienly small, 
	does the coboundary expansion of $(X,w,R_X/\calG)$ in dimension $0$
	remains bounded away from $0$ if (instead of condition (1) of Theorem~\ref{TH:cbe-for-quotient-sheaves}) 
	we require that every
	$m(|X(0)|)$ of the  $\{R_x\}_{x\in X}$ are linearly disjoint in $R$? More specifically:
	\begin{enumerate}[label=(\alph*)]
		\item Can we take $m = o(n)$? Can we take $m=O(1)$?
		\item What if we also require that $R_x=0$ when $x$ is not a vertex?
	\end{enumerate}
\end{que}

The motivation behind the variant (b) is  that in \cite[\S9]{First_2023_sheaves_on_complexes_preprint}, where Theorem~\ref{TH:cbe-for-partite-quotient-sheaves}
is applied,
we only need 
the special case where $R_x=0$ for all $x\in X-X(0)$. Also, if one allows $R_x$ to be nonzero when $x$ is an edge,
then it is impossible to take $m=O(\log n)$
as the following example shows.

\begin{example}
Fix an integer $k\geq 3$ such that $q:=k-1$ is an odd prime power.
It is known \cite[Thm.~4.13]{Morgenstern_1994_explicit_Ramanujan_graphs}
that there exists an infinite family $\{X_i\}_{i\in\N}$
of $k$-regular Ramanujan graphs, i.e.\ $\lambda(X_i)\leq \frac{2\sqrt{k-1}}{k}$ 
for all $i$,
such that 
the girth of $X_i$ is greater than $\frac{4}{3} \log_q |X_i(0)|$.
Let $X$ be one of these graphs, let $w$ be its natural weight function
and let 
$n=|X_i(0)|$.
Let $\F$ be a   field and let $V$ a vector space
over $\F$ of dimension at least $\floor{\frac{4}{3}\log_q n}+1$.
Provided that $|\F|$ or $\dim V$
are sufficiently large, there exist vectors $\{f_v\}_{v\in X(0)}$ in $V$
such that every $\floor{\frac{4}{3} \log_q n}+1$ of the $f_v$
are linearly independent.
Consider the collection $(f_v)_{v \in X(0)}$
as a $0$-cochain $f\in C^0(X,V)$    and form the subsheaf
$\calG$ of $V_X$ as in Example~\ref{EX:problem-in-quotient-sheaves};
briefly, we take $R=V$ and put  $R_v=0$ for every $v\in X(0)$
and $R_e=\F d_0f(e)=\F(f_{e^+}-f_{e^-})$ for every $e\in X(1)$.
As noted in that example, $\cb_0(X,w,V_X/\calG)=0$.
However, since the girth of $X$ is greater than $\frac{4}{3}\log_q n$,
and every $\floor{\frac{4}{3} \log_q n}+1$ of the $f_v$
are linearly independent, one readily checks that
every $\floor{\frac{4}{3} \log_q n}$ of the $R_x$
are linearly disjoint.
\end{example}

Finally, we ask whether Theorems~\ref{TH:cbe-for-quotient-sheaves}
and~\ref{TH:cbe-for-partite-quotient-sheaves} extend to higher dimensions.

\begin{que}
	Suppose that $(X,w)$ is a weighted simplicial complex
	(resp.\ $({r+1})$-partite weighted simplicial complex) 
	and we are given subgroups   $\{R_x\}_{x\in X}$ of $R$
	which 
	are all linearly disjoint. Form the subsheaf $\calG$ of $R_X$ as in Theorem~\ref{TH:cbe-for-quotient-sheaves}
	and suppose that the underlying graph of $X_z$ is a $[-\lambda,\lambda]$-expander
	(resp.\ $(\dim X-\dim z)$-partite $[-\lambda,\lambda]$-expander)
	for every $z\in X$ with $\dim z\leq \dim X-2$.
	Provided $\lambda$ is sufficiently small,
	is there an $\veps>0$, depending only on $\lambda$, $\dim X$  and $r$, 
	such that  $(X,w,R_X/\calG)$ is   an 
	$\veps$-coboundary expander in dimensions $0,\dots,\dim X-1$?
\end{que}

\bibliographystyle{plain}
\bibliography{MyBib_24_03}

\def\cprime{$'$} \def\cprime{$'$} \def\cprime{$'$} \def\cprime{$'$}
\begin{thebibliography}{10}

\bibitem{Abramenko_2008_Buildings}
Peter Abramenko and Kenneth~S. Brown.
\newblock {\em Buildings}, volume 248 of {\em Graduate Texts in Mathematics}.
\newblock Springer, New York, 2008.
\newblock Theory and applications.

\bibitem{Curry_2014_sheaves_cosheaves_PhD}
Justin~Michael Curry.
\newblock {\em Sheaves, cosheaves and applications}.
\newblock ProQuest LLC, Ann Arbor, MI, 2014.
\newblock Thesis (Ph.D.)--University of Pennsylvania.

\bibitem{Davidoff_2003_Ramanujan_graphs}
Giuliana Davidoff, Peter Sarnak, and Alain Valette.
\newblock {\em Elementary number theory, group theory, and {R}amanujan graphs},
  volume~55 of {\em London Mathematical Society Student Texts}.
\newblock Cambridge University Press, Cambridge, 2003.

\bibitem{Dikstein_2023_cbe_cse_without_dep_on_dim_deg}
Yotam Dikstein and Irit Dinur.
\newblock Coboundary and cosystolic expansion without dependence on dimension
  or degree, 2023.

\bibitem{Dinur_2022_ltcs_const_rate}
Irit Dinur, Shai Evra, Ron Livne, Alexander Lubotzky, and Shahar Mozes.
\newblock Locally testable codes with constant rate, distance, and locality.
\newblock In {\em S{TOC} '22---{P}roceedings of the 54th {A}nnual {ACM}
  {SIGACT} {S}ymposium on {T}heory of {C}omputing}, pages 357--374. ACM, New
  York, [2022] \copyright 2022.

\bibitem{Dinur_2022_near_coverings}
Irit Dinur and Roy Meshulam.
\newblock Near coverings and cosystolic expansion.
\newblock {\em Arch. Math. (Basel)}, 118(5):549--561, 2022.

\bibitem{Dotterrer_2012_coboundary_expanders}
Dominic Dotterrer and Matthew Kahle.
\newblock Coboundary expanders.
\newblock {\em J. Topol. Anal.}, 4(4):499--514, 2012.

\bibitem{Dotterrer_2018_topological_overlap}
Dominic Dotterrer, Tali Kaufman, and Uli Wagner.
\newblock On expansion and topological overlap.
\newblock {\em Geom. Dedicata}, 195:307--317, 2018.

\bibitem{Evra_2016_cosystolic_expanders}
Shai Evra and Tali Kaufman.
\newblock Bounded degree cosystolic expanders of every dimension.
\newblock In {\em S{TOC}'16---{P}roceedings of the 48th {A}nnual {ACM} {SIGACT}
  {S}ymposium on {T}heory of {C}omputing}, pages 36--48. ACM, New York, 2016.

\bibitem{Evra_2016_cosystolic_expanders_arxiv_version}
Shai Evra and Tali Kaufman.
\newblock Bounded degree cosystolic expanders of every dimension.
\newblock 2017.
\newblock Summary appeared in S{TOC}'16---{P}roceedings of the 48th {A}nnual
  {ACM} {SIGACT} {S}ymposium on {T}heory of {C}omputing.

\bibitem{First_2023_sheaves_on_complexes_preprint}
Uriya~A. First and Tali Kaufman.
\newblock On good $2$-query locally testable codes from sheaves on high
  dimensional expanders.
\newblock 2023.

\bibitem{First_2024_cosyst_exp_posets_preprint}
Uriya~A. First and Tali Kaufman.
\newblock Cosystolic expansion of sheaves on posets with applications to good
  2-query locally testable codes and lifted codes.
\newblock 2024.
\newblock arXiv:2403.19388.

\bibitem{First_2024_cosyst_exp_posets_stoc}
Uriya~A. First and Tali Kaufman.
\newblock Cosystolic expansion of sheaves on posets with applications to good
  2-query locally testable codes and lifted codes.
\newblock In {\em S{TOC}'24---{P}roceedings of the 56th {A}nnual {ACM}
  {S}ymposium on {T}heory of {C}omputing}, pages 1446--1457. ACM, New York,
  [2024] \copyright 2024.

\bibitem{Friedland_2002_Cheeger_type_ineqs}
Shmuel Friedland and Reinhard Nabben.
\newblock On {C}heeger-type inequalities for weighted graphs.
\newblock {\em J. Graph Theory}, 41(1):1--17, 2002.

\bibitem{Gromov_2010_expanders_and_top_II}
Mikhail Gromov.
\newblock Singularities, expanders and topology of maps. {P}art 2: {F}rom
  combinatorics to topology via algebraic isoperimetry.
\newblock {\em Geom. Funct. Anal.}, 20(2):416--526, 2010.

\bibitem{Hansen_2019_spectral_thy_of_sheaves}
Jakob Hansen and Robert Ghrist.
\newblock Toward a spectral theory of cellular sheaves.
\newblock {\em J. Appl. Comput. Topol.}, 3(4):315--358, 2019.

\bibitem{Horak_2013_spectra_of_Laplace_ops}
Danijela Horak and J\"{u}rgen Jost.
\newblock Spectra of combinatorial {L}aplace operators on simplicial complexes.
\newblock {\em Adv. Math.}, 244:303--336, 2013.

\bibitem{Kaufman_2016_isoperimetic_inequalities}
Tali Kaufman, David Kazhdan, and Alexander Lubotzky.
\newblock Isoperimetric inequalities for {R}amanujan complexes and topological
  expanders.
\newblock {\em Geom. Funct. Anal.}, 26(1):250--287, 2016.

\bibitem{Kaufman_2014_high_dim_expanders_property_testing}
Tali Kaufman and Alexander Lubotzky.
\newblock High dimensional expanders and property testing.
\newblock In {\em I{TCS}'14---{P}roceedings of the 2014 {C}onference on
  {I}nnovations in {T}heoretical {C}omputer {S}cience}, pages 501--506. ACM,
  New York, 2014.

\bibitem{Kaufman_2018_cosystolic_expanders}
Tali Kaufman and David Mass.
\newblock Cosystolic expanders over any abelian group.
\newblock {\em Electronic Colloquium on Computational Complexity (ECCC)},
  25:134, 2018.

\bibitem{Kaufman_Mass_2021_cosystolic_non_abelean}
Tali Kaufman and David Mass.
\newblock Unique-neighbor-like expansion and group-independent cosystolic
  expansion.
\newblock In {\em 32nd International Symposium on Algorithms and Computation,
  {ISAAC} 2021, December 6-8, 2021, Fukuoka, Japan}, volume 212 of {\em
  LIPIcs}, pages 56:1--56:17, 2021.

\bibitem{Kaufman_Mass_2022_improved_cosystole}
Tali Kaufman and David Mass.
\newblock {Double Balanced Sets in High Dimensional Expanders}.
\newblock In Amit Chakrabarti and Chaitanya Swamy, editors, {\em Approximation,
  Randomization, and Combinatorial Optimization. Algorithms and Techniques
  (APPROX/RANDOM 2022)}, volume 245 of {\em Leibniz International Proceedings
  in Informatics (LIPIcs)}, pages 3:1--3:17, Dagstuhl, Germany, 2022. Schloss
  Dagstuhl -- Leibniz-Zentrum f{\"u}r Informatik.

\bibitem{Kaufman_2021_coboundary_cosystolic_exp_strong_sym_preprint}
Tali Kaufman and Izhar Oppenheim.
\newblock Coboundary and cosystolic expansion from strong symmetry, 2021.

\bibitem{Kaufman_2021_amplified_local_testability_preprint}
Tali Kaufman and Izhar Oppenheim.
\newblock High dimensional expansion implies amplified local testability.
\newblock 2021.

\bibitem{Linial_2006_homological_connectivity}
Nathan Linial and Roy Meshulam.
\newblock Homological connectivity of random 2-complexes.
\newblock {\em Combinatorica}, 26(4):475--487, 2006.

\bibitem{Lobotzky_2019_random_steiner_systems}
Alexander Lubotzky, Zur Luria, and Ron Rosenthal.
\newblock Random {S}teiner systems and bounded degree coboundary expanders of
  every dimension.
\newblock {\em Discrete Comput. Geom.}, 62(4):813--831, 2019.

\bibitem{Lubotzky_2015_random_latin_squares}
Alexander Lubotzky and Roy Meshulam.
\newblock Random {L}atin squares and 2-dimensional expanders.
\newblock {\em Adv. Math.}, 272:743--760, 2015.

\bibitem{Lubotzky_2016_expansion_of_buildings}
Alexander Lubotzky, Roy Meshulam, and Shahar Mozes.
\newblock Expansion of building-like complexes.
\newblock {\em Groups Geom. Dyn.}, 10(1):155--175, 2016.

\bibitem{Meshulam_2009_homological_connectivity}
R.~Meshulam and N.~Wallach.
\newblock Homological connectivity of random {$k$}-dimensional complexes.
\newblock {\em Random Structures Algorithms}, 34(3):408--417, 2009.

\bibitem{Morgenstern_1994_explicit_Ramanujan_graphs}
Moshe Morgenstern.
\newblock Existence and explicit constructions of {$q+1$} regular {R}amanujan
  graphs for every prime power {$q$}.
\newblock {\em J. Combin. Theory Ser. B}, 62(1):44--62, 1994.

\bibitem{Oppenheim_2015_vanishing_of_cohomology}
Izhar Oppenheim.
\newblock Vanishing of cohomology and property ({T}) for groups acting on
  weighted simplicial complexes.
\newblock {\em Groups Geom. Dyn.}, 9(1):67--101, 2015.

\bibitem{Panteleev_2022_good_quantum_codes}
Pavel Panteleev and Gleb Kalachev.
\newblock Asymptotically good quantum and locally testable classical {LDPC}
  codes.
\newblock In {\em S{TOC} '22---{P}roceedings of the 54th {A}nnual {ACM}
  {SIGACT} {S}ymposium on {T}heory of {C}omputing}, pages 375--388. ACM, New
  York, [2022] \copyright 2022.

\bibitem{Shepard_1985_cellular_descrip_of_der_cat_PhD}
Allen~Dudley Shepard.
\newblock {\em A {C}ellular {D}escription of {T}he {D}erived {C}category of a
  {S}tratified {S}pace}.
\newblock ProQuest LLC, Ann Arbor, MI, 1985.
\newblock Thesis (Ph.D.)--Brown University.

\end{thebibliography}

\end{document}